\documentclass[11pt, a4paper]{scrartcl}
\usepackage{RGDM_linear_preamble}
\usepackage{aligned-overset}
\usepackage{physics}
\usepackage{tikz}
\usepackage{subcaption}

\makeatletter
\DeclareRobustCommand{\cev}[1]{%
  \mathpalette\do@cev{#1}%
}
\newcommand{\do@cev}[2]{%
  \fix@cev{#1}{+}%
  \reflectbox{$\m@th#1\vec{\reflectbox{$\fix@cev{#1}{-}\m@th#1#2\fix@cev{#1}{+}$}}$}%
  \fix@cev{#1}{-}%
}
\newcommand{\fix@cev}[2]{%
  \ifx#1\displaystyle
    \mkern#23mu
  \else
    \ifx#1\textstyle
      \mkern#23mu
    \else
      \ifx#1\scriptstyle
        \mkern#22mu
      \else
        \mkern#22mu
      \fi
    \fi
  \fi
}

\makeatother

\bibliography{bibliography_linear}
\usepackage{aligned-overset}

\title{\fontsize{16}{19}\selectfont Reflected diffusion models adapt to low-dimensional data}
\author{Asbjørn Holk\thanks{Aarhus University, Department of Mathematics, Ny Munkegade 118, 8000 Aarhus C, Denmark. \newline Email: \href{mailto:a.holk@math.au.dk}{a.holk@math.au.dk}} \and Claudia Strauch\thanks{Heidelberg University, Institute for Mathematics, Im Neuenheimer Feld 205, 69120 Heidelberg, Germany. \newline Email: \href{mailto:strauch@math.uni-heidelberg.de}{strauch@math.uni-heidelberg.de}} \and Lukas Trottner\thanks{University of Stuttgart, Department of Mathematics, Wankelstraße 5, 70563 Stuttgart, Germany. \newline Email: \href{mailto:lukas.trottner@isa.uni-stuttgart.de}{lukas.trottner@isa.uni-stuttgart.de}}}

\begin{document}

\maketitle

\begin{abstract}
While the mathematical foundations of score-based generative models are increasingly well understood for unconstrained Euclidean spaces, many practical applications involve data restricted to bounded domains. 
This paper provides a statistical analysis of reflected diffusion models on the hypercube $[0,1]^D$ for target distributions supported on $d$-di\-men\-si\-o\-nal linear subspaces. 
A primary challenge in this setting is the absence of Gaussian transition kernels, which play a central role in standard theory in $\mathbb{R}^D$. 
By employing an easily implementable infinite series expansion of the transition densities, we develop analytic tools to bound the score function and its approximation by sparse ReLU networks. 
For target densities with Sobolev smoothness $\alpha$, we establish a convergence rate in the $1$-Wasserstein distance of order $n^{-\frac{\alpha+1-\delta}{2\alpha+d}}$ for arbitrarily small $\delta > 0$, demonstrating that the generative algorithm fully adapts to the intrinsic dimension $d$. 
These results confirm that the presence of reflecting boundaries does not degrade the fundamental statistical efficiency of the diffusion paradigm, matching the almost optimal rates known for unconstrained settings.
\end{abstract}

\section{Introduction}
Deep generative models constitute a broad and rapidly evolving class of methods for learning complex data distributions from samples, with score-based diffusion models \cite{song21} emerging as a particularly powerful dynamic paradigm in recent years.
Motivated by the fact that many modern data sets are intrinsically low-dimensional yet embedded in high-dimensional bounded ambient spaces, we study the statistical performance diffusion-based generative modelling for probability measures supported on lower-dimensional manifolds within a bounded domain. 

The study of the statistical performance of diffusion models has become a central avenue of research in statistics for machine learning. Several papers \cite{oko23,azangulov24,tang24,puchkin24,chak26,dou24,zhang24,fan25,kwon26,chen23b} consider the convergence of such algorithms in the standard setting of an Ornstein--Uhlenbeck (OU in the following) noising model under different regularity and structural assumptions on the target  distribution as well as different score approximation classes such as neural networks with or without sparsity assumptions and kernel-type estimators. We provide more details on existing results for unconstrained models and how they relate to our work in the discussion in Section \ref{sec:discussion}, but  for now focus on the class of reflected diffusion models that we consider here. 

Such generative models were first introduced in \cite{lou23,fishman23} motivated by the fact that practical implementations of the generative backward process often rely on thresholding procedures to enforce geometric constraints on the data, even though the forward OU training model  ignores such constraints. To overcome such theoretical discrepancies \cite{lou23,fishman23} suggest to use a \textit{reflected} diffusion process as driver of noise instead. This allows to follow the same time-reversal rationale underlying unconstrained diffusion models since the time-reversal of a reflected diffusion is again a reflected diffusion with adjusted drift incorporating the information on the target data provided by the time evolution of the \textit{score}, that is the log-gradient of the forward marginals given initialisation in the data distribution. 

First statistical guarantees for such models were given in \cite{holk24} under the assumption that the data distribution has full support on the bounded reflection domain with Sobolev density bounded away from zero. The goal of this paper is to extend this analysis to singular target distributions supported on a lower-dimensional manifold $M \subset [0,1]^D$ and to demonstrate that the convergence rate of reflected diffusion models adapt to the intrinsic dimension of the data. In doing so, we provide the first rigorous statistical analysis of diffusion-based constrained generative models on bounded domains that explicitly accounts for low-dimensional data structures. This is particularly important from both an applied and a theoretical perspective in light of the so-called \textit{manifold hypothesis} \cite{ma12,loaiza24,fefferman16}. It postulates that image or text data (and many others) although being extremely high-dimensional have common structural features that make them supported on (unions of) much lower dimensional manifolds. Empirical evidence of this has, e.g., been provided by \cite{pope21,stan24,brown23}, which makes the manifold hypothesis a reasonable explanation for the tremendous success of deep generative models, provided that their adaptivity to intrinsic lower-dimensional geometric structures can be theoretically verified.

As a natural starting point, we focus on the simplified case where $M$ lies in a linear subspace of $\R^D$. This agrees with the route taken by first studies on adaptivity of unconstrained diffusion models to lower-dimensional data \cite{oko23,chen23b} and provides an important foundation for further investigations into more complex manifold structures.   
Our main contributions can be summarised as follows:
\begin{itemize}
\item 
While existing literature on manifold adaptation relies fundamentally on the Gaussian transition kernels of unconstrained OU processes, we develop analytic tools to control the score function associated with \emph{reflected Brownian motion on the hypercube}. 
Compared to \cite{holk24} we do not work with an eigenfunction expansion of the the score, but use the simple geometry of the hypercube together with symmetries of reflected Brownian motion to expand the transition densities as an infinite mixture of restricted Gaussian densities. This allows us to provide precise bounds on the spatial growth of the score and its singular behaviour as $t \searrow 0$, effectively \emph{decoupling the boundary effects of the domain} from the concentration of the measure $\mu$ around the low-dimensional subspace $M$. 
\item 
We prove that, despite the analytic complexities introduced by the reflecting boundaries and the resulting non-Gaussian transition densities, denoising score matching via sparse ReLU networks achieves the required approximation rates. 
In particular, we demonstrate that the \emph{complexity of the estimator depends only on the intrinsic dimension $d$} of the linear subspace supporting the data, rather than the ambient dimension $D$.
\item
We derive an \emph{upper bound for the $1$-Wasserstein distance} between the target distribution and the law of the generated samples. The established rate of $O(n^{-(\alpha+1-\delta)/(2\alpha+d)})$ confirms that imposing hard physical constraints via reflection on $[0,1]^D$ does \emph{not} degrade the fundamental statistical efficiency of the diffusion model, matching the almost optimal rates known for unconstrained dynamics.
\end{itemize}

In the following, we summarise the generative and statistical estimation procedure, introduce and discuss our assumptions on the target distribution and provide an informal version of our main result on 1-Wasserstein convergence rates of reflected diffusion models.

\paragraph{Forward reflected diffusion}
Let $\mu$ be a target probability distribution on $\R^D$, concentrated on a compact $d$-dimensional manifold $M \subset [0,1]^D$, where possibly $d\ll D$. Given an i.i.d.\ sample of data with distribution $\mu$, our aim is to generate approximate samples for $\mu$ in a two-step procedure via a time-reversal mechanism for reflected diffusions.

As a first step, we perturb $\mu$ by adding isotropic noise through a reflected Brownian motion on the hypercube. Specifically, we consider the reflected SDE
\begin{equation}
	\label{eq:forward_equation}
	\diff{X_t}
	=\diff{B_t}+n(X_t)\diff{L_t},\quad X_0\sim\mu,
\end{equation}
where $(B_t)_{t\ge0}$ is a standard $D$-dimensional Brownian motion, $n(x)$ denotes an inward-pointing normal vector at $x\in\partial[0, 1]^D$, and $(L_t)_{t\ge0}$ is the local time of $(X_t)_{t\ge0}$ at $\partial[0, 1]^D$, i.e., a one-dimensional, continuous and non-decreasing process  satisfying $L_t=\int_{0}^{t}\bm{1}_{\{X_s\in\partial[0, 1]^D\}}\,\mathrm{d}L_s$ and $\int_0^t \lvert n(X_s) \rvert \diff{L_s} < \infty$ almost surely. The presence of the stochastic forcing term $n(X_t) \diff{L_t}$ in the dynamics prevents the process from escaping the unit cube by normally reflecting it back into the interior when it hits the boundary.
Note that for boundary points $x \in \partial [0,1]^d$ where two or more faces of the cube intersect, the direction of $n(x)$ is not uniquely defined. If for $x\in\partial[0, 1]^D$, we let $I(x), J(x)\subseteq[D]$ denote the indices of $x$ for which $x_i=0$ and $x_j=1$, respectively, we specify $n(x)=\sum_{i\in I(x)}e_i-\sum_{j\in J(x)}e_j$, where $e_i$ is the $i$-th standard unit vector in $\R^D$. In particular, $n(x)$ is the unique inward pointing normal vector on smooth parts of the cube boundary, where faces do not intersect. The particular choice on the non-smooth part of the boundary $E\coloneqq\{x\mid \lvert I(x)\rvert + \lvert J(x) \rvert>1\}$ is without consequences, since the reflected Brownian motion will never hit $E$ almost surely when started in $[0,1]^D\setminus E$, cf.\ \cite[Theorem 1.1]{williams87} and we will assume without further mention that $\operatorname{supp}(\mu) \cap E = \varnothing$.

Existence and pathwise uniqueness of strong solutions for general reflected diffusions in bounded convex domains has been shown in \cite{tanaka79} under mild conditions on the coefficients, which are  satisfied for the Brownian case considered here.  In particular, because of the simple geometry of $[0,1]^D$ and the normal reflection direction, the $i$-th coordinate $X^i$ of the strong solution of \eqref{eq:forward_equation} is a strong solution to the one-dimensional reflected SDE 
\[\diff{X^i_t} = \diff{B^i_t} - \operatorname{sgn}(X_t^i)\diff{L^i_t},\]
where $L^i$ is the local time at $\{0,1\}$ and $\operatorname{sgn}(x) = -1$ for $x \leq 0$ and $\operatorname{sgn}(x) = 1$ for $x > 0$. Thus, conditional on the initialisation $X_0$, the components of $X$ are independent reflected Brownian motions on $[0,1]$ and $L = \sum_{i=1}^D L^i$ almost surely. The boundary local times $L^i$ at the faces are characterised via the occupation limit 
\[L^i_t = \lim_{\varepsilon \downarrow 0} \frac{1}{2\varepsilon} \int_{0}^t \one_{[0,\varepsilon] \cup [1-\varepsilon,1]}(X^i_s) \diff{s}, \]
which hold both almost surely and in $L^2$, uniformly on relatively compact sets in $t$, see \cite[Theorem 2.6]{burdzy04} in a more general context. Consequently, 
\begin{equation}\label{eq:occupation}
L_t = \lim_{\varepsilon \downarrow 0} \frac{1}{2\varepsilon}\int_0^t \sum_{i=1}^D \one_{[0,\varepsilon] \cup [1-\varepsilon,1]}(X_s^i) \diff{s},
\end{equation}
uniformly in $L^2$ and almost surely on relatively compact sets of $t$.
In the following, we let $p_t$ denote the density of $X_t$ wrt Lebesgue measure on $\R^D$. %For any $t>0$, $p_t$ is smooth and strictly positive on $[0,1]^D$, even though $\mu$ itself is singular. 

\paragraph{Time reversal and generative sampling}
Fix a terminal time $\overline{T}>0$, and define the time-reversed process $\cev{X}_t \coloneqq X_{\overline{T}-t}$ for $t\in[0,\overline{T}]$. 
Then, there exists a Brownian motion $(\overline{B}_t)_{t\ge0}$ such that $\cev{X}$ satisfies the reflected SDE
\begin{equation}
	\label{eq:backward_equation}
	\diff{\cev{X}_t} = \nabla \log p_{\overline{T}-t}(\cev{X}_t)\diff{t}
	+ \diff{\overline{B}_t}
	+ n(\cev{X}_t)\diff{\overline{L}_t},
	\qquad \cev{X}_0 \sim p_{\overline{T}},
\end{equation}
where $\overline{L}_t \coloneqq L_{\overline{T}} - L_{\overline{T}-t}$ is the local time of $\cev{X}$ at the boundary $\partial[0, 1]^D$. This result is proved in \cite[Theorem 2.5]{cattiaux88} for more general reflected diffusions on smooth domains, while \cite[Theorem 3.2]{fishman23} give an instructive probabilistic proof inspired by the non-reflected case \cite{haussmann86} for reflected Brownian motion on precompact, \textit{smooth} convex domains. Their proof relies on the boundary occupation limit characterisation of the local time, which in our case is provided by \eqref{eq:occupation}, and sufficient smoothness properties of the transition densities of $X_t$, which in our model can be verified based on the representation given in Lemma \ref{lem:explicit_score}. Thus, even though the unit cube $[0,1]^D$ is non-smooth, its simple geometry allows us to provide the technical tools needed to follow the proof of \cite[Theorem 2.3]{fishman23} to verify  \eqref{eq:backward_equation}.

If the score function $s_0(x,t)\coloneqq\nabla\log p_t(x)$ were known, then simulating \eqref{eq:backward_equation} would yield exact samples from $\mu$ at time $\overline{T}$. Since $p_t$ and $s_0$ depend on the unknown target distribution, they must be approximated from data.

\paragraph{Approximate backward diffusion.}
Given an approximation $s(x,t)$ of the score function, we instead consider the reflected SDE
\begin{equation}
	\label{eq:backward_equation_approx}
	\diff{\overline{X}_t^{s}}
	= s(\overline{X}_t^{s}, \overline{T}-t)\diff{t}
	+ \diff{\overline{B}_t}
	+ n(\overline{X}_t^{s})\diff{\overline{L}_t},
	\qquad \overline{X}_0^{s} \sim \mathcal U([0,1]^D).
\end{equation}
Several structural features of the proposed framework motivate the use of reflected diffusions on the hypercube.
First, the state space $[0,1]^D$ is natural in many applications, including image and signal generation, and reflection provides a principled mechanism to ensure that generated samples remain within prescribed bounds.
Second, the forward process \eqref{eq:forward_equation} has zero drift and unit diffusion, which allows for an explicit representation of its transition kernel (see Lemma~\ref{lem:explicit_solution}) and substantially simplifies the probabilistic analysis.
Third, the uniform distribution on $[0,1]^D$ is invariant for the forward reflected Brownian motion, yielding a simple and practically convenient initialisation for the backward dynamics. 
For numerical stability, we do not run the backward dynamics all the way to time $\overline{T}$, but instead output the sample $\overline{X}_{\overline{T}-\underline{T}}^{s}$ for some small $\underline{T}>0$. This introduces three distinct sources of error:
\begin{enumerate}[label=(\roman*)]
	\item the truncation error due to early stopping at time $\underline{T}$;
	\item the initialisation error from starting at stationarity rather than $p_{\overline{T}}$;
	\item the approximation error from using $s$ instead of the true score $s_0$.
\end{enumerate}

\paragraph{Error metric.}
Our goal is to quantify the discrepancy between $\mu$ and the law of the generated samples. Since the algorithm produces a random probability measure as the terminal law of a stochastic process, a natural notion of error is provided by Wasserstein distances. More specifically, our error criterion is the $1$-Wasserstein distance, for probability measures $\nu_1,\nu_2$ on $[0,1]^D$ defined by
\[
\mathcal W_1(\nu_1,\nu_2) \coloneqq \inf_{\pi\in\Pi(\nu_1,\nu_2)} \int_{[0,1]^D\times[0,1]^D} \lvert x-y \rvert\,\pi(\diff x,\diff y),
\]
where $\Pi(\nu_1,\nu_2)$ denotes the set of all couplings of $\nu_1$ and $\nu_2$.
Unlike divergences based on densities, $\mathcal W_1$ remains meaningful when $\nu_1$ or $\nu_2$ are supported on lower-dimensional sets, and it admits a natural interpretation in terms of couplings of stochastic processes, making it well suited for diffusion-based generative models.

\paragraph{Score estimation via denoising score matching.}
To estimate the score function, we discretise the time interval $[\underline{T},\overline{T}]$ into $K\in\N$ subintervals $\{[t_{i-1},t_i]\}_{i=1}^K$, where $K\asymp \log n$ and $t_i = \underline{T} c^i$  for some $c\in(1,2]$, $t_K=\overline{T}$.
On each subinterval, we approximate the map $(x,t)\mapsto\nabla\log p_t(x)$ separately.
The construction of the score estimator is based on the classical equivalence between the explicit score matching loss
\[
\int_{t_{i-1}}^{t_i}\mathbb{E}\big[\lvert s(X_t,t)-\nabla\log p_t(X_t)\rvert^2\big]\diff{t}
\]
and the denoising score matching loss
\[
\int_{t_{i-1}}^{t_i}\mathbb{E}\big[\lvert s(X_t,t)-\nabla\log q_t(X_0,X_t)\rvert^2\big]\diff{t},
\]
where $q_t(x,\cdot)$ denotes the transition density of the forward reflected diffusion at time $t$, started from $x$.
More precisely, for any measurable function $s$, one has
\[
\int_{t_{i-1}}^{t_i}\mathbb{E}\big[\lvert s(X_t,t)-\nabla\log p_t(X_t)\rvert^2\big]\diff{t}
= \mathbb{E}\bigg[\int_{t_{i-1}}^{t_i} \lvert s(X_t,t)-\nabla\log q_t(X_0,X_t)\rvert^2\diff{t}\bigg] + C_i,
\]
where
$C_i= -\mathbb{E}\big[\int_{t_{i-1}}^{t_i}\lvert\nabla\log p_t(X_t)-\nabla\log q_t(X_0,X_t)\rvert^2\diff{t}\big]$ is a constant independent of $s$.
Accordingly, for a given approximation class $\mathcal S_i$ and $i\in[K]$, we define
\[
L_s^{(i)}(x) \coloneqq \int_{t_{i-1}}^{t_i} \mathbb{E}\big[\lvert s(X_t,t)-\nabla\log q_t(x,X_t)\rvert^2 \mid X_0=x\big]\diff{t}.
\]
Minimising the explicit score matching loss over $\mathcal S_i$ is then equivalent to minimising $\mathbb{E}[L_s^{(i)}(X_0)]$ over the same class.
Given i.i.d.\ samples $Y_1,\ldots,Y_n\sim\mu$, a natural estimator of the score on the interval $[t_{i-1},t_i]$ is obtained by minimising the empirical denoising score matching loss
\[
\widehat{L}_s^{(i)} \coloneqq \frac{1}{n}\sum_{j=1}^n L_s^{(i)}(Y_j).
\]
For a collection of sparse ReLU neural network classes $\{\mathcal S_i\}_{i=1}^K$, we thus define the overall score estimator as the piecewise function
\[
\widehat{s}_n(x,t)
= \sum_{i=1}^K \widehat{s}_n^{(i)}(x,t)\,\bm{1}_{[t_{i-1},t_i)}(t),
\quad \text{ where }
\widehat{s}_n^{(i)} \in \argmin_{s\in\mathcal S_i} \widehat{L}_s^{(i)}.
\]
Conditional on $\widehat{s}_n$ we then simulate the reflected SDE \eqref{eq:backward_equation_approx} with $s =\hat{s}_n$ until time $\overline{T} - \underline{T}$ and use $\overline{X}^{\hat{s}_n}_{\overline{T} - \underline{T}}$ as an approximate sample for the target distribution $\mu$.

\paragraph{Assumptions and main result} 
The probabilistic results developed in Section~\ref{sec:prob} apply to general target measures supported on smooth submanifolds.
For the statistical analysis of score estimation and for deriving explicit convergence rates, however, we restrict attention to a setting in which the geometry of the support is sufficiently simple to permit sharp approximation bounds for neural networks.
Specifically, we introduce the following assumptions about $\mu$ and its support $M$:
\begin{enumerate}[label = ($\mathcal{H}$\arabic*), ref = ($\mathcal{H}$\arabic*)]
	\item \label{ass:H1} There exist orthonormal vectors $v_1,\ldots,v_d\in\R^D$ with $d \leq D$ and a shift $v_0\in[0,1]^D$ such that $M$ is connected with non-empty interior, has a Lipschitz boundary and is a closed subset of $(V+v_0)\cap[0, 1]^D$   where $V=\mathrm{Span}(v_1, v_2, \ldots, v_d)$.
	Moreover, there exist constants $c_0 \geq d, r_0>0$ such that, for all $x\in M$ and all $r>0$, we have $\mu\big(\mathcal B(x,r)\cap M\big)\gtrsim (r\wedge r_0)^{c_0}$ and $\mathrm{Vol}_d(\mathcal B(x, r)\cap M)\gtrsim(r\wedge r_0)^d$, where $\mathrm{Vol}_d$ denotes the restriction of the $d$-dimensional Lebesgue measure to $(V + v_0) \cap [0,1]^D$.
	\item \label{ass:H2} 
    The target distribution $\mu$ admits a density $p_0$ wrt to the $\mathrm{Vol}_d$ such that
	\begin{enumerate}[label = (\roman*)]
		\item $p_0\in H_0^{\alpha}(M)$ with $\alpha \in \N \cap (d/2,\infty)$, i.e., the density has Sobolev smoothness $\alpha$ on $M$ and the weak derivatives up to order $\alpha -1$ vanish at the boundary in the trace sense;
		\item $p_0$ is bounded and bounded away from zero on an interior region of $M$, i.e., there exist constants $0<p_{\min}\le p_{\max}<\infty$ and $\varepsilon_M>0$ satisfying $p_0(x)\le p_{\mathrm{max}}$ for all $x\in M$ and $p_0(x)\ge p_{\mathrm{min}}$ for all $x\in M_{-\varepsilon_M/2}\coloneqq M\setminus(\partial M)_{\varepsilon_M/2}$, where $(\partial M)_{\varepsilon_M/2}$ denotes the $\varepsilon_M/2$-fattening of $\partial M$.
	\end{enumerate}
\end{enumerate}

Existence of $r_0 > 0$ such that $\mathrm{Vol}_d(\mathcal B(x, r)\cap M)\gtrsim (r\wedge r_0)^d$ for all $r>0$ and $x\in M$ is guaranteed if $M$ is $\beta$-smooth for some $\beta\ge 2$ and has positive reach $\tau>0$; see, e.g., \cite[Lemma 20]{divol22}.
In this setting, existence of $c_0 \geq d$ such that  $\mu\big(\mathcal B(x,r)\cap M\big)\gtrsim (r\wedge r_0)^{c_0}$ is then further guaranteed if the target density $p_0$ decays polynomially towards $\partial M$, i.e., if $p_0(x)\gtrsim\mathrm{dist}(x, \partial M)^{c_0-d}$ for $x$ sufficiently close to the boundary. A typical construction of densities $p_0 \in H^\alpha_0(M)$ would model $p_0(x) = c\operatorname{dist}(x,\partial M)^{c_0 -d}$ in a neighbourhood of $\partial M$ with $c_0 \geq \alpha + d$. A variation of such an assumption on controlled decay at the boundary has also been used in \cite{steph25} and allows us to avoid rather artificial strict lower boundedness assumptions on the target density. A visualisation of our support assumption is given in Figure \ref{fig:support}.
For small times $t$, the forward density $p_t$ concentrates sharply around $M$, and the score $\nabla\log p_t(x)$ grows rapidly as $x$ moves away from $M$.
Accurate score estimation in this regime is statistically delicate, particularly near the boundary of $M$.
To avoid technical complications associated with boundary singularities, we  introduce the following auxiliary regularity and geometric conditions.
\begin{enumerate}[label = ($\mathcal{H}$\arabic*), ref = ($\mathcal{H}$\arabic*)]
\setcounter{enumi}{2}
	\item \label{ass:H3} When restricted to an area near the boundary, the target density $p_0$ is sufficiently smooth.
	Specifically, there exists $\varepsilon_M>0$ such that the restriction of $p_0$ to a neighbourhood of the boundary, $p_0\vert_{(\partial M)_{\varepsilon_M}\cap M}\in C^{\kappa}((\partial M)_{\varepsilon_M}\cap M, \R)$, where $\kappa\coloneqq \frac{d(c_0-d)}{2}+d+3\alpha+2$ and $(\partial M)_{\varepsilon_M}$ denotes the $\varepsilon_M$-fattening of $\partial M$.
    \item \label{ass:H4} $M$ does not intersect $\partial [0,1]^D$, i.e., there exists $\rho_{\min} > 0$ such that 
    \[\mathrm{dist}(M, \partial[0, 1]^D) \coloneq \inf_{x\in M, y\in\partial[0, 1]^D}|x-y| \geq \rho_{\min}.\]
\end{enumerate}
When imposing both \ref{ass:H2} and \ref{ass:H3}, we assume that the values of $\varepsilon_M$ coincide.
We note that \ref{ass:H3}  is comparable to assumptions made in related work on statistical estimation rates of unconstrained diffusion models, see, e.g., \cite[Assumption 6.3]{oko23}, \cite[Assumption (\textbf{B})]{kwon25}. If, e.g., $p_0(x) = c \operatorname{dist}(x,\partial M)^{c_0 -d}$ close to the boundary as in the discussion above, this assumption is always satisfied provided $\partial M$ is sufficiently smooth. Assumption \ref{ass:H4} can  always be enforced by rescaling the data and undo this scaling for the generated output. Generally, for any of our results, we will precisely state which (if any) of the above assumptions are needed.  

With this setup, we can state an informal version of our main theorem; the precise statement is given in Theorem~\ref{theo:main}.
\begin{theorem*}[informal]
	Assume \ref{ass:H1}--\ref{ass:H4}. For any $\delta>0$, choose $\underline{T}\in\mathrm{Poly}(n^{-1})$ and $\overline{T}\asymp\log n$.
	Then there exists a family $\{\mathcal S_i\}_{i=1}^K$ of sparse ReLU neural network classes such that the reflected diffusion generative algorithm driven by the empirical denoising score matching estimator $\widehat{s}_n$ satisfies
	\[
	\mathbb{E}\Big[\mathcal W_1\big(\mu,\,
	\mathcal L(\overline{X}^{\widehat{s}_n}_{\overline{T}-\underline{T}})\big)\Big]
	= O\big(n^{-\frac{\alpha+1-\delta}{2\alpha+d}}\big),
	\]
    where $\mathcal L(\overline{X}^{\widehat{s}_n}_{\overline{T}-\underline{T}})$ denotes the law of the output $\overline{X}^{\widehat{s}_n}_{\overline{T}-\underline{T}}$ conditional on the data.
\end{theorem*}

\paragraph{Organisation of the paper}
The remainder of the paper is organised as follows.
In Section~\ref{sec:prob}, we develop the probabilistic foundations of our model, including the construction of the forward reflected diffusion and the derivation of the explicit score representation. We also establish crucial bounds on the growth and regularity of the score function in Lemma~\ref{lem:affine_approx}. Section~\ref{sec:wasserstein} is dedicated to the statistical analysis and the proof of our main result. 
We present the error decomposition for the $1$-Wasserstein distance, specify the sparse ReLU network classes used for estimation, and combine these results to prove Theorem~\ref{theo:main}. 
Finally, Section~\ref{sec:discussion} places our findings in the context of recent minimax results and discusses extensions to general manifolds and discretisation errors. All technical proofs omitted from the main part  and auxiliary results are collected in the Appendix.

\section{Probabilistic analysis}\label{sec:prob}
Our statistical analysis relies on a detailed understanding of distributional and path properties of normally reflected Brownian motions in a hypercube, which we develop in this section. Generally, strong solutions to reflected SDEs can be constructed via the so-called Skorokhod map. In the present setting of a normally reflected Brownian motion in the hypercube, this is a mapping $\Gamma \colon C([0,\infty);\mathbb{R}^D) \to C([0,\infty);[0,1]^D)$ such that the strong solution of \eqref{eq:forward_equation} may be written as $X_\cdot = \Gamma(Y + B_\cdot)$.
Although existence and uniqueness of the Skorokhod map are well known, its explicit characterisation is generally intractable. For the purposes of simulation and for analysing distributional properties of the forward process, however, it is sufficient to work with weak solutions.
By pathwise uniqueness for reflected diffusions in convex domains, cf.\ \cite[Theorem 2.5.1]{pilipenko}, weak solutions are also unique in law (the classical Yamada--Watanabe argument assuming pathwise uniqueness for unconstrained SDEs, cf.\ \cite[Proposition~5.3.20]{karatzas91}, extends to the reflected setting).
We therefore begin by constructing a simple weak solution of \eqref{eq:forward_equation}, which is particularly convenient for training and theoretical analysis for the following reasons:
\begin{enumerate}[label = (\roman*), ref = (\roman*)] 
\item \label{weak_prop1} simulating the weak solution is as easy as simulating a Brownian motion.
\item the weak solution yields a simple, interpretable and numerically simple to approximate series representation of the transition densities $q_t(x,y)$, see Lemma \ref{lem:explicit_score}. Together with property \ref{weak_prop1}, this yields a simple recipe that does not require full forward path simulations to numerically approximate the empirical denoising score matching loss  by using Algorithm \ref{alg:approx_score_loss}.
\item the transition density formula from Lemma \ref{lem:explicit_score} is perfectly suited to capture the influence of the intrinsic dimensionality of the data support on theoretical approximation properties of the score $\nabla \log p_t(x)$, which eventually translates to faster estimation rates.
\end{enumerate}

All proofs of the lemmata stated in this section are deferred to Appendix~\ref{app:proofs_prob}.
\begin{lemma}
	\label{lem:explicit_solution}
	Let $\widehat{f} \colon \mathbb{R} \to [0,1]$ be the $2$-periodic function defined by
	\[
	\widehat{f}(x)
	\coloneqq
	\begin{cases}
		x, & \text{if } x \in [0,1),\\
		-x, & \text{if } x \in [-1,0),
	\end{cases}
	\]
	and extended periodically to all of $\mathbb{R}$.
	Define $f \colon \mathbb{R}^D \to [0,1]^D$ by applying $\widehat{f}$ component-wise, that is, $f(x) \coloneqq (\widehat{f}(x_i))_{i=1,\ldots,D}$.
	If $Y \sim \nu$ for some probability measure $\nu$ on $[0,1]^D$ and $(B_t)_{t \geq 0}$ is a $D$-dimensional Brownian motion independent of $Y$, then the process $(X_t)_{t \geq 0}$ defined by $X_t \coloneqq f(B_t + Y)$ is a weak solution to the reflected SDE
	\[
	\diff{X_t}= \diff{W_t} + n(X_t)\diff{L_t},\quad t \geq 0,
	\]
	with initial distribution $X_0 \sim \nu$ and Brownian motion $W$.
\end{lemma}

\begin{figure}[h]
	\centering
	\begin{tikzpicture}[scale=1.5]
		\draw[<->] (-3.2, 0) -- (3.2, 0) node[right] {$x$};
		\draw[->] (0, -0.2) -- (0, 1.2) node[above] {$y$};
		\draw[dashed] (-3.1, 1) -- (3.1, 1) node[above] {$y=1$};
		\draw plot coordinates {(-3.1, 0.9) (-3, 1) (-2, 0) (-1, 1) (0, 0) (1, 1) (2, 0) (3, 1) (3.1, 0.9)} node[right] {$\widehat{f}(x)$};
	\end{tikzpicture}
	\caption{
		Graph of the function $\widehat{f}$ from Lemma~\ref{lem:explicit_solution}.
		The function reflects the identity between the lines $y=0$ and $y=1$; applied component-wise, $f$ therefore essentially reflects the identity at the boundary $\partial[0,1]^D$.
	}
\end{figure}
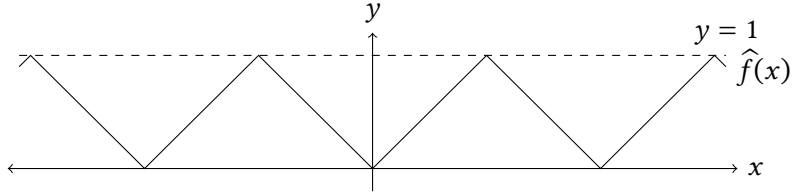

\begin{figure}[h]
	\centering
	\includegraphics[width=0.5\textwidth]{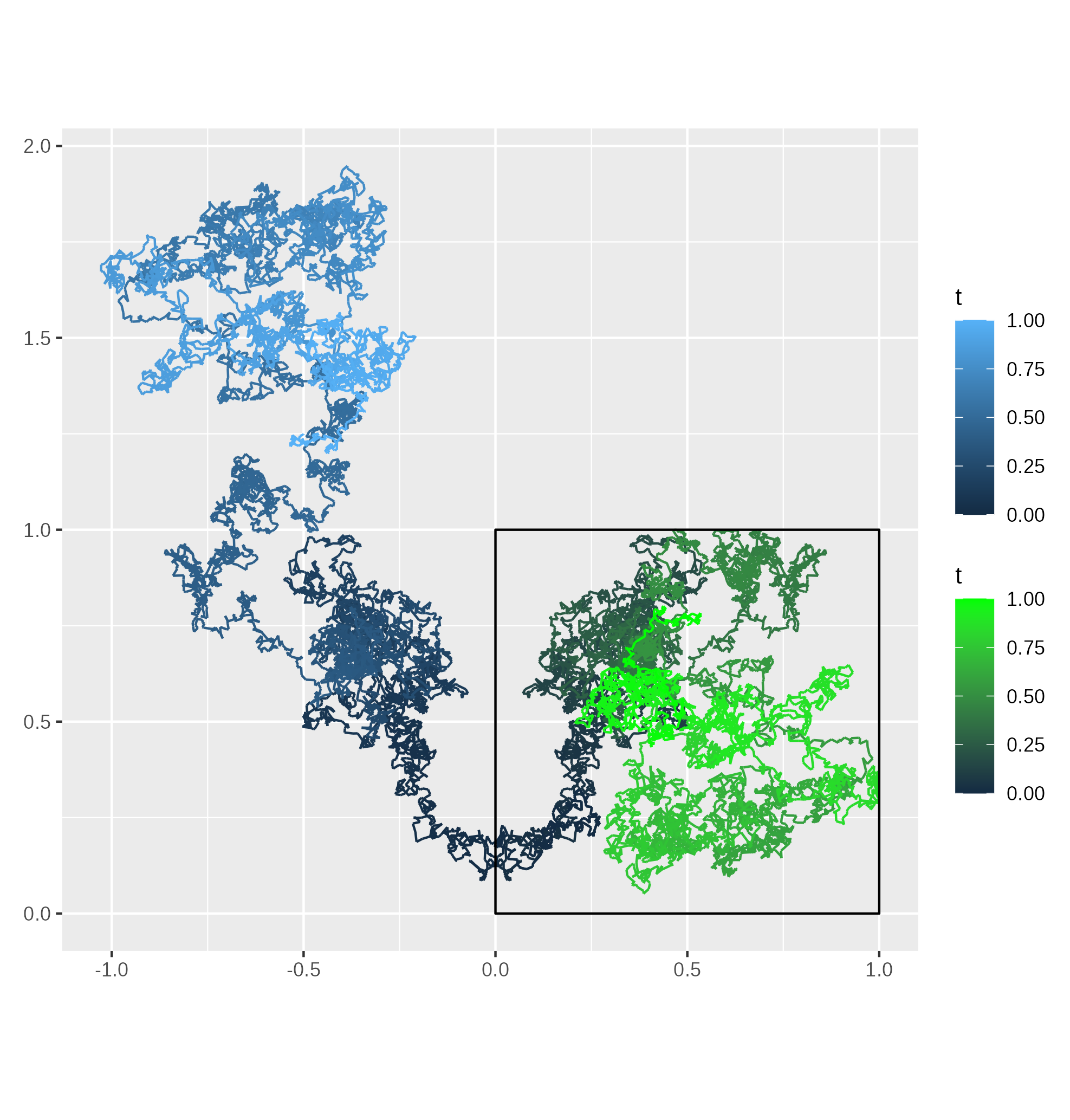}
	\caption{
	Simulation of a reflected Brownian motion (green) along with the non-reflected version (blue) that is used for its construction using Lemma \ref{lem:explicit_solution}.
	}
\end{figure}

Using the explicit construction from Lemma~\ref{lem:explicit_solution}, we can derive a closed-form expression for the density $p_t$ of the forward process $(X_t)_{t\geq 0}$ and, consequently, for the associated score function $s_0$.

\begin{lemma}\label{lem:explicit_score}
Let $(X_t)_{t \geq 0}$ be a solution to \eqref{eq:forward_equation} and define the reflection operator $R_z \colon [0,1]^D \to [0,1]^D$  component-wise by
\[
\bigl(R_z(x)\bigr)_i\coloneqq
	\begin{cases}
		x_i, & \text{if } z_i \text{ is even},\\
		1-x_i, & \text{if } z_i \text{ is odd},
	\end{cases}
\] 
where $i \in \N, z_i \in \mathbb{Z}, x_i \in [0,1]$.
Then, for all $x,y \in [0,1]^D$ and $t > 0$, the transition density of the reflected Brownian motion in $[0,1]^D$ is given by
\[q_t(y,x) = (2\pi t)^{-D/2}
\sum_{z \in \mathbb{Z}^D}\exp\left(-\frac{\lvert R_z(x)+z-y\rvert^2}{2t}\right),\]
and, for $p_t$ denoting the density of $X_t$ wrt.\ the Lebesgue measure on $\mathbb{R}^D$,
\[
p_t(x)=(2\pi t)^{-D/2}\sum_{z \in \mathbb{Z}^D}\int_{[0,1]^D}
	\exp\left(-\frac{\lvert R_z(x)+z-y\rvert^2}{2t}\right)\, \mu(\diff{y}), \quad x \in [0,1]^D.
\]
In particular, the score function $s_0(x,t)=\nabla \log p_t(x)$ admits the explicit representation
\begin{equation}\label{eq:score}
s_0(x,t)=-\frac{\sum_{z \in \mathbb{Z}^D}(-1)^z\int_{[0,1]^D}(R_z(x)+z-y)
			\exp\left(-\frac{\lvert R_z(x)+z-y\rvert^2}{2t}\right)\,\mu(\diff{y})}{
			t\sum_{z \in \mathbb{Z}^D}\int_{[0,1]^D}
			\exp\left(-\frac{\lvert R_z(x)+z-y\rvert^2}{2t}\right)
			\,\mu(\diff{y})},
\end{equation}
where we use the shorthand $(-1)^z \coloneqq \mathrm{diag}\bigl(((-1)^{z_i})_{i=1}^D\bigr)$.
\end{lemma}
%The proof can be found in Appendix \ref{app:proofs_prob}.

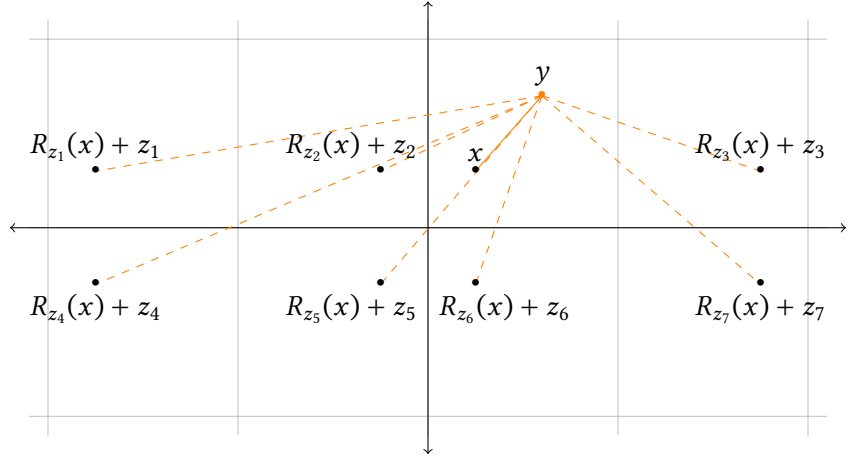
\begin{figure}[h]
\centering
\begin{tikzpicture}[scale=2.5]
\draw[very thin, color=lightgray] (-2.1, -1.1) grid (2.1, 1.1);
\draw[<->] (-2.2, 0) -- (2.2, 0);
\draw[<->] (0, -1.2) -- (0, 1.2);
\draw[-, color=orange] (0.6, 0.7) -- (0.25, 0.3);
\draw[-, dashed, color=orange] (0.6, 0.7) -- (-0.25, 0.3);
\draw[-, dashed, color=orange] (0.6, 0.7) -- (-1.75, 0.3);
\draw[-, dashed, color=orange] (0.6, 0.7) -- (-1.75, -0.3);
\draw[-, dashed, color=orange] (0.6, 0.7) -- (-0.25, -0.3);
\draw[-, dashed, color=orange] (0.6, 0.7) -- (0.25, -0.3);
\draw[-, dashed, color=orange] (0.6, 0.7) -- (1.75, -0.3);
\draw[-, dashed, color=orange] (0.6, 0.7) -- (1.75, 0.3);
\draw (0.25, 0.3) node[above] {$x$} node {\textbullet};
\draw (-0.25, 0.3) node[above] {$R_{z_2}(x)+z_2\quad\quad$} node {\textbullet};
\draw (0.25, -0.3) node[below] {$\quad\quad R_{z_6}(x)+z_6$} node {\textbullet};
\draw (-0.25, -0.3) node[below] {$R_{z_5}(x)+z_5\quad\quad$} node {\textbullet};
\draw (1.75, 0.3) node[above] {$R_{z_3}(x)+z_3$} node {\textbullet};
\draw (1.75, -0.3) node[below] {$R_{z_7}(x)+z_7$} node {\textbullet};
\draw (-1.75, 0.3) node[above] {$R_{z_1}(x)+z_1$} node {\textbullet};
\draw (-1.75, -0.3) node[below] {$R_{z_4}(x)+z_4$} node {\textbullet};
\draw (0.6, 0.7) node[above] {$y$} node[color=orange] {\textbullet};
\end{tikzpicture}
\caption{The points $R_{z_i}(x)+z_i$ are all mapped back to $x$ under the function $f$ from Lemma~\ref{lem:explicit_solution}.
Consequently, the reflected process $X_t$ moves from $y$ to $x$ if and only if the Brownian motion $B_t$ moves from $y$ to $x$ \emph{or} to any of the points $R_{z_i}(x)+z_i$.}
\end{figure}

Combining Lemma \ref{lem:explicit_solution} and Lemma \ref{lem:explicit_score} yields the simple numerical algorithm \ref{alg:approx_score_loss} to approximate the empirical denoising score matching loss, which is necessary for implementation of the score estimation procedure. 

\begin{algorithm}[ht]
   \caption{Numerical approximation of empirical denoising score matching loss $\hat{L}^{(i)}_s$}
   \label{alg:approx_score_loss}
\begin{algorithmic}
   \STATE {\bfseries Input:} data $\{Y_j\}_{j=1}^n \overset{iid}{\sim} \nu$,  $N \in \N$, time interval index $i \in [K]$, $s \in \mathcal{S}_i$
   \STATE set $\widetilde{L}^{(i)}_s = 0$
   \FOR{$k=1$ {\bfseries to} $N$}
   \STATE draw $y = Y_{j_k}$  uniformly from $\{Y_j\}_{j=1}^n$ 
   \STATE  draw independently $t \sim \mathcal{U}([t_{i-1}, t_i])$ and  $B_t \sim \mathcal{N}(0,t I_D)$
   \STATE set $x_t = f(Y_{j_k}+ B_{t})$
   \STATE choose $K \in \N$ and set 
   \[\widetilde{\nabla \log q_t}(y,x_t) = \frac{\sum_{z \in \mathbb{Z}^D, \lvert z \rvert \leq K} (-1)^z (R_z(x_t)+z-y) \exp\left(-\frac{\lvert R_z(x_t)+z-y\rvert^2}{2t}\right)}{
			t\sum_{z \in \mathbb{Z}^D, \lvert z \rvert \leq K}
			\exp\left(-\frac{\lvert R_z(x_t)+z-y\rvert^2}{2t}\right)} \]
   \STATE set $\widetilde{L}^{(i)}_s \leftarrow \widetilde{L}^{(i)}_s + \frac{1}{N} \lvert s(t,x_t) - \widetilde{\nabla \log q_t}(y,x_t) \rvert^2$
   \ENDFOR
   \STATE {\bfseries Output:} $\widetilde{L}^{(i)}_s$
\end{algorithmic}
\end{algorithm}
\begin{remark}
\begin{enumerate}[label = (\roman*), ref = (\roman*)]
\item choosing the cutoff parameter $K$ dependent on the initialisation $y$ and the drawn time $t$ is important since it should be proportional to the number of reflections along the path $y \to x_t$ (decreasing in $t$ and the distance of $y$ to $\partial [0,1]^D$). See also the discussion in \cite{lou23} on implementation of the model, where Gaussian approximations are made for small $t$ and spectral decompositions of the transition density are exploited for large $t$ approximations. 
\item in practice, one may simulate $(y_k, t_k, x_{t_k})_{k=1}^N$ only once and use these for Monte--Carlo approximation of $\hat{L}^{(i)}_s$ for the updated approximators $s$ in every optimisation step.
\end{enumerate}
\end{remark}

These results now allow us to give a precise analysis of the growth of the score $\nabla \log p_t(x)$ in $t$ depending on the distance of $x$ to the data manifold $M$ as well as the path behaviour of the reflected Brownian motion for small times. These properties will play a crucial role for constructing efficient neural network score approximators and for proving almost optimal rates in $1$-Wasserstein distance. %The proof of the next result is given in Section \ref{app:proofs_prob}.

\begin{lemma}
	\label{lem:affine_approx}
	Fix $t>0$ and let $M_{\rho, t}=\{x\in[0, 1]^D: \mathrm{dist}(x, M)\le\sqrt{t(D+2\rho)}\}$ denote the $\sqrt{t(D+2\rho)}$-fattening of $M$ for some $\rho>1$.
	Then, under \ref{ass:H1} the following hold:
	\begin{enumerate}[ref=\thelemma.(\alph*), label=(\alph*)]
		\item\label{lem:affine_approx_a} $|\nabla\log p_t(x)|\lesssim \frac{1}{t\wedge \sqrt{t}}$ for $x\in[0, 1]^D$;
		\item\label{lem:affine_approx_b} $\mathbb{E}[|\nabla\log p_t(X_t)|^2\bm{1}_{M_{\rho, t}^{\mathsf{c}}}(X_t)]
		\lesssim \frac{1}{t^2\wedge1}\mathrm{e}^{-\rho}$;
		\item \label{lem:} if $t\leq 1/2$, there exists a universal constant $C$ such that 
		\[
			\mathbb{P}\Big(\forall s \in [t,1]: \lvert X_s - X_0 \rvert \leq C\sqrt{D}\sqrt{s}\big(\log (1 + \log t^{-1}) + y\big)\Big)
			\geq 1 - 4D\mathrm{e}^{-2y^2}, \quad y > 0;
		\]
		\item\label{lem:affine_approx_d} $p_t(x)\gtrsim t^{\frac{c_0-D}{2}}\mathrm{e}^{-\rho}$ for $x\in M_{\rho, t}$;
		\item\label{lem:affine_approx_e} 
        for $t \in (0,1]$, 
		\begin{enumerate}[ref=(\roman*), label=(\roman*)]
		\item \label{lem:loggrad_score_e1}$\forall x,y \in [0,1]^D: \, \lvert \nabla_x \log q_t(y,x) \rvert \lesssim \tfrac{\lvert x-y\rvert}{t} + \tfrac{1}{\sqrt{t}}$; 
		\item \label{lem:loggrad_score_e2} $\forall x \in [0,1]^D: \, \lvert \nabla  \log p_t(x) \rvert = \lvert \E[\nabla_2 \log q_t(X_0, X_t) \mid X_t = x] \rvert \lesssim \tfrac{1}{t} \E[\lvert X_t - X_0 \rvert \mid X_t = x] + \tfrac{1}{\sqrt{t}}$;
		\item \label{lem:loggrad_score_e3}$\forall t \in (0,1], x\in M_{\rho, t}\colon \, \lvert \nabla\log p_t(x)\rvert \lesssim\frac{\sqrt{\rho+\log t^{-1}}}{\sqrt{t}}$.
		\end{enumerate}
	\end{enumerate}
\end{lemma}
%\begin{remark} 
Note that part \hyperref[lem:affine_approx_e]{(e)}.\ref{lem:loggrad_score_e2} together with the upper bound $\lvert X_t - X_0 \rvert \leq \sqrt{D}$ immediately implies the bound from part \hyperref[lem:affine_approx_a]{(a)} for $t \in (0,1]$. However, our combinatorial proof technique for \hyperref[lem:affine_approx_a]{(a)} can be translated directly to truncated versions of the score representation given in \eqref{eq:score}, which will form the basis of our neural network approximation strategy in the next section. Conversely, the proof of \hyperref[lem:affine_approx_e]{(e)}.\ref{lem:loggrad_score_e2} relies on the denoising score representation $\nabla \log p_t(x) = \E[\nabla_2 \log q_t(X_0, X_t) \mid X_t = x]$, which has no probabilistic analogue for the truncated score representation. 
%\end{remark}

\section{Wasserstein convergence rate}\label{sec:wasserstein}
This section establishes quantitative convergence guarantees in the $1$-Wasserstein distance for the reflected diffusion generative scheme driven by an estimated score.
Our analysis builds on a careful decomposition of the approximation error and combines statistical bounds for score estimation with probabilistic stability estimates for reflected stochastic dynamics.
In our earlier work \cite{holk24}, we derived upper bounds of order $n^{-\alpha/(2\alpha+D)}$ (up to polylogarithmic factors) for the total variation distance under Sobolev smoothness $\alpha>D/2$, expressed in terms of the ambient dimension $D$.
In the present bounded-domain setting, such bounds immediately imply corresponding guarantees in the $1$-Wasserstein distance. %, since $\mathcal W_1$ is controlled by total variation on compact spaces.
However, classical results from nonparametric density estimation in the i.i.d.\ setting (see, for instance, Theorem~2 in \cite{niwebe22}) suggest that these rates are suboptimal.
Indeed, \cite{oko23} obtained an improved upper bound of order $n^{-(\alpha+1-\delta)/(2\alpha+D)}$, for arbitrary $\delta>0$, by exploiting a refined multiscale analysis of the reverse diffusion.
While our overall strategy is inspired by the approach of \cite{oko23}, their arguments cannot be transferred verbatim to the present setting.
In particular, the pathwise stability estimates in \cite{oko23} rely heavily on properties of OU processes on $\R^D$, whereas our model is governed by reflected diffusions on a compact domain with boundary.
As a consequence, we must develop and invoke genuinely new probabilistic tools, including precise growth and regularity bounds for reflected Brownian paths and their associated scores, as established in Lemma~\ref{lem:affine_approx}.
At the same time, the compactness of the state space allows us to simplify several technical aspects of the construction and to avoid certain localisation arguments that are necessary in the unbounded setting.

\subsection{Error decomposition}
As outlined in the introduction, the overall approximation error decomposes into three distinct contributions: the error due to early stopping of the backward dynamics, the error incurred by approximating the score function, and the error arising from initialising the dynamics with the uniform distribution on $[0,1]^D$ rather than with the target distribution.
To disentangle these effects, we introduce, for a given score approximation $s$, an auxiliary reflected diffusion $(\widehat{X}^s_t)_{t\in[0,\overline T]}$ defined as the solution to
\[
\diff \widehat{X}_t^s = s(\widehat{X}_t^s,\overline{T}-t)\diff t + \diff B_t + n(\widehat{X}_t^s)\diff L_t,
\qquad t\in[0,\overline{T}],
\]
with initial condition $\widehat{X}_0^s\sim p_{\overline{T}}$.
The triangle inequality for the $1$-Wasserstein metric $\mathcal{W}_1$ yields for the process defined in \eqref{eq:backward_equation_approx} the error decomposition
\begin{equation}
	\label{eq:error_split}
	\mathcal W_1\bigl(\mu,\overline{X}^s_{\overline{T}-\underline{T}}\bigr) \le
	\mathcal W_1(\mu,X_{\underline{T}}) + \mathcal W_1\bigl(X_{\underline{T}},\widehat{X}^s_{\overline{T}-\underline{T}}\bigr) +
	\mathcal W_1\bigl(\widehat{X}^s_{\overline{T}-\underline{T}}, \overline{X}^s_{\overline{T}-\underline{T}}\bigr).
\end{equation}
Since $X_{\underline{T}}\sim\widehat{X}^{s_0}_{\overline{T}-\underline{T}}$, we may rewrite
\[
\mathcal W_1\bigl(X_{\underline{T}},\widehat{X}^s_{\overline{T}-\underline{T}}\bigr) =
\mathcal W_1\bigl(\widehat{X}^{s_0}_{\overline{T}-\underline{T}}, \widehat{X}^s_{\overline{T}-\underline{T}}\bigr),
\]
where the two reflected processes $\hat{X}^s$ and $\hat{X}^{s_0}$ are initialised in the same distribution $p_{\overline{T}}$, but have different drifts $s$ and $s_0$, respectively.
Likewise, $\widehat{X}^s$ and $\overline{X}^s$ share the same drift, but are started in different initial distributions $p_{\overline{T}}$ and $\mathcal{U}[0,1]^D$, respectively.
Consequently, the three terms on the right-hand side of \eqref{eq:error_split} correspond, respectively, to the error due to early stopping, the error caused by approximating the score function, and the error introduced by initialising the dynamics with the uniform distribution.
We begin by controlling the first and third term, which admit comparatively elementary bounds.

\begin{lemma}\label{lem:early_stopping}
	Let $\mu$ be an arbitrary probability distribution on $[0,1]^D$, and let $(X_t)_{t\ge0}$ be a solution to \eqref{eq:forward_equation} with initial condition $X_0\sim\mu$.
	Then, for all $t\ge0$,
	\[
	\mathcal W_1(\mu,X_t)\le \sqrt{D\,t}.
	\]
\end{lemma}
\begin{proof}
	Let $Y\sim\mu$, and define $X_t=f(B_t+Y)$, where $f$ is as in Lemma~\ref{lem:explicit_solution}, and $(B_t)_{t\ge0}$ is a Brownian motion independent of $Y$.
	By Lemma~\ref{lem:explicit_solution}, the process $(X_t)_{t\ge0}$ indeed solves \eqref{eq:forward_equation}.
	Since $f$ is $1$-Lipschitz, we obtain
	\[
	\mathcal W_1(\mu,X_t)\le \E\bigl[|f(B_t+Y)-Y|\bigr]= \E\bigl[|f(B_t+Y)-f(Y)|\bigr]\le \E[|B_t|]
	\le \sqrt{D\,t},
	\]
	where the final inequality follows from the Cauchy--Schwarz inequality.
\end{proof}

We next address the error arising from initialising the backward dynamics in the stationary distribution rather than in $p_{\overline T}$. This can be bounded using uniform ergodicity of  reflected Brownian motions in bounded convex domains, which is analysed in detail in \cite{loper20}.

\begin{lemma}\label{lem:stat_dist}
	Let $0<\underline{T}\le\overline{T}$, and suppose that $s$ is such that \eqref{eq:backward_equation_approx} admits a unique strong solution on $[0,\overline{T}-\underline{T}]$ for any initial distribution.
	Let $(\overline{X}^s_t)_{t\in[0,\overline{T}]}$ and $(\widehat{X}^s_t)_{t\in[0,\overline{T}]}$ denote such solutions, with $\overline{X}^s_0\sim\mathcal{U}[0,1]^D$ and $\widehat{X}^s_0\sim p_{\overline{T}}$, respectively.
	Then,
	\[
	\mathcal{W}_1\bigl(\widehat{X}^s_{\overline{T}-\underline{T}},\overline{X}^s_{\overline{T}-\underline{T}}\bigr)\le \frac{8\sqrt{D}}{\pi}\exp\Big(-\frac{\pi^2 \overline{T}}{2D}\Big).
	\]
\end{lemma}

\begin{remark}
For existence and uniqueness of strong solutions, it suffices that for each $t\in[0,\overline{T}-\underline{T}]$ the map $x\mapsto s(x,t)$ is Lipschitz continuous with a Lipschitz constant independent of $t$, cf.\ \cite{tanaka79}.
This condition is satisfied by all neural network score approximations $s$ considered in this paper.
\end{remark}
\begin{proof}[Proof of Lemma \ref{lem:stat_dist}]
	We begin by recalling that, for any two probability measures $\nu,\nu'$ on $[0,1]^D$,
	\begin{equation}\label{eq:tv}
	\mathcal W_1(\nu,\nu')\le 2\,\mathrm{diam}([0,1]^D)\,\TV(\nu,\nu')
	=2\sqrt{D}\,\TV(\nu,\nu').
	\end{equation}
	Let $(Q_{0,t})_{t\ge0}$ denote the transition kernels of the (possibly time-inhomogeneous) SDE \eqref{eq:backward_equation_approx}, and for any probability measure $\nu$ define $\nu Q_{0,t}(\diff x) \coloneqq \int_{[0,1]^D}Q_{0,t}(y,\diff x)\,\nu(\diff y)$.
	Writing $\mu_{\overline{T}}(\diff x)=p_{\overline{T}}(x)\diff x$ and letting $\rho$ denote the uniform distribution on $([0,1]^D,\mathcal B([0,1]^D))$, we have $\widehat{X}^s_{\overline{T}-\underline{T}}\sim\mu_{\overline{T}}Q_{0,\overline{T}-\underline{T}}$,
	while $\overline{X}^s_{\overline{T}-\underline{T}}\sim\rho Q_{0,\overline{T}-\underline{T}}$.
	Since $(Q_{0,t})_{t\ge0}$ is a contraction semigroup, it follows that
	\begin{align*}
	\mathcal W_1\bigl(\widehat{X}^s_{\overline{T}-\underline{T}}, \overline{X}^s_{\overline{T}-\underline{T}}\bigr)
	\le 2\sqrt{D}\,\TV\bigl(\mu_{\overline{T}}Q_{0,\overline{T}-\underline{T}},\rho Q_{0,\overline{T}-\underline{T}}\bigr)
	&\le 2\sqrt{D}\,\TV(\mu_{\overline{T}},\rho)\\
    &= 2\sqrt{D} \sup_{A \in \mathcal{B}([0,1]^D)} \Big\lvert\int_{[0,1]^D} \int_A (p_{\overline{T}}(x,y) - 1) \diff{y} \,\mu(\diff{x}) \big\rvert\\ 
    &\leq 2\sqrt{D} \sup_{x \in [0,1]^D} \sup_{A \in \mathcal{B}([0,1]^D)} \Big\lvert \int_A (p_{\overline{T}}(x,y) - 1) \diff{y} \big\rvert\\
    &= 2\sqrt{D} \sup_{x \in [0,1]^D}\operatorname{TV}(q_{\overline{T}}(x,\cdot), \rho).
    \end{align*}
    The result now follows from \cite[Theorem 4]{loper20}, which states that 
    \[ \sup_{x \in [0,1]^D}\operatorname{TV}(q_{\overline{T}}(x,\cdot), \rho) \leq \frac{4}{\pi}\exp\Big(-\frac{\pi^2 \overline{T}}{2D}\Big).\]
\end{proof}

We now turn to the second term in \eqref{eq:error_split} and follow the general strategy from \cite{oko23} to control it.  
We start by decomposing the time interval $[0,\overline{T}-\underline{T}]$ into a sequence of geometrically shrinking sub-intervals and introduce, on each such sub-interval, an auxiliary process in which the true score is only partially replaced by its approximation.
This multilevel construction allows us to localise the score approximation error in time and to derive sharper bounds on the resulting Wasserstein distance.
Fix a constant $c\in(1,2]$, and choose $K\in\mathbb{N}$ such that $\underline{T}c^{K}=\overline{T}$.
Define the intermediate times $t_i \coloneqq \underline{T}c^{i}$, $i=0,\ldots,K$.
For each $i\in\{0,\ldots,K\}$ and a given score approximation $s$, let $Y^{(i)}=(Y^{(i)}_t)_{t\in[0, \overline{T}-\underline{T}]}$ denote the solution to the reflected SDE
\begin{alignat*}{2}
	\diff Y^{(i)}_t
	&= \nabla \log p_{\overline{T}-t}\bigl(Y^{(i)}_t\bigr)\diff t
	+ \diff B_t
	+ n\bigl(Y^{(i)}_t\bigr)\diff L_t,
	\quad && t\in[0,\overline{T}-t_i), \\
	\diff Y^{(i)}_t
	&= s\bigl(Y^{(i)}_t,\overline{T}-t\bigr)\diff t
	+ \diff B_t
	+ n\bigl(Y^{(i)}_t\bigr)\diff L_t,
	\quad && t\in[\overline{T}-t_i,\overline{T}-\underline{T}],
\end{alignat*}
with initial condition $Y^{(i)}_0 \sim p_{\overline{T}}$.
Thus, the process $Y^{(i)}$ follows the exact reverse-time dynamics driven by the true score up to time $\overline{T}-t_i$, and subsequently evolves according to the approximate score $s$.
By construction, we have the distributional identities
\[
X_{\underline{T}} \sim \widehat{X}^{s_0}_{\overline{T}-\underline{T}} \sim Y^{(0)}_{\overline{T}-\underline{T}},
\qquad
\widehat{X}^{s}_{\overline{T}-\underline{T}} \sim Y^{(K)}_{\overline{T}-\underline{T}}.
\]
Applying the triangle inequality for the $1$-Wasserstein distance therefore yields
\[
\mathcal{W}_1\bigl(X_{\underline{T}}, \widehat{X}^{s}_{\overline{T}-\underline{T}}\bigr)
\le \sum_{i=1}^{K} \mathcal{W}_1\bigl(Y^{(i-1)}_{\overline{T}-\underline{T}}, Y^{(i)}_{\overline{T}-\underline{T}}\bigr).
\]
The following proposition provides a bound on each of the incremental Wasserstein distances, which essentially improves by a factor of $((t_i \wedge 1)\rho)^{1/2}$ the rough upper bound that can be derived from combining a total variation bound and Girsanov's theorem, provided that the score approximation satisfies $\lvert s(x,t) \rvert \leq C\sqrt{\rho /t}$ for $t \leq 1$. This growth control is motivated by Lemma \ref{lem:affine_approx}, whose combined conclusion tells us that for any $p \in \N$, with probability at least $1 - 1/n^p$, the true score satisfies
\[\forall t \geq \underline{T}: \lvert s_0(t,X_t) \rvert = \lvert \nabla \log p_t(X_t) \rvert \lesssim t^{-1/2} \sqrt{\log(1 + \log \underline{T}^{-1}) + \sqrt{p \log n}}.\]
The improved Wasserstein bound allows to compensate the higher difficulty of score approximation for small times $t$ caused by its increasing irregularity as $t \to 0$ and thereby obtain faster Wasserstein convergence rates. The detailed proof is postponed to Appendix \ref{app:wasserstein}.

\begin{proposition} \label{lem:wasserstein_Yi}
	Assume \ref{ass:H4}. Let $s$ be a score approximation satisfying $|s(x,t)| \le C\sqrt{\rho/t}$ for all $x\in[0,1]^D$, $t\in(0,1]$, for some constants $C>0$ and $\rho>1$.
	Assume moreover that $t_1\le 1$ and that $\log\bigl(1+\log(t_1^{-1})\bigr)\lesssim \sqrt{\rho}$.
	Then, for any $i=1,\ldots,K$, the corresponding processes $Y^{(i-1)}$ and $Y^{(i)}$ satisfy
	\[
	\mathcal W_1\bigl(Y_{\overline{T}-\underline{T}}^{(i-1)}, Y_{\overline{T}-\underline{T}}^{(i)}\bigr)
	\le \mathfrak{C}\Bigg(\mathrm{e}^{-\rho}+ \bigg( (t_i\wedge 1)\rho
	\int_{t_{i-1}}^{t_i} \mathbb{E}\bigl[ |s(X_t,t)-\nabla\log p_t(X_t)|^2\bigr]\,
	\diff t\bigg)^{\frac12}\Bigg),
	\]
	for some constant $\mathfrak{C}>0$ independent of $i$. In particular, 
    \[\mathcal{W}_1\bigl(X_{\underline{T}}, \widehat{X}^{s}_{\overline{T}-\underline{T}}\bigr)
    \le \mathfrak{C}\Bigg(K\mathrm{e}^{-\rho} + \sqrt{\rho} \sum_{i=1}^K\sqrt{t_i\wedge 1}
	\bigg(\int_{t_{i-1}}^{t_i} \mathbb{E}\bigl[ |s(X_t,t)-\nabla\log p_t(X_t)|^2\bigr]\,
	\diff t\bigg)^{\frac12}  \Bigg)\].
\end{proposition}
Recall that, for given $\underline{T},\overline{T}$ and $t_i = \underline{T}c^i$, $i =0,\ldots,K$, as above such that $c^K = \overline{T}$, we estimate the score on on $[t_{i-1},t_i)$ by minimising the empirical denoising score loss via 
\begin{equation}\label{eq:estimator1}
    \widehat s_n^{(i)} \in \argmin_{s\in\mathcal S_i}\frac{1}{n}\sum_{k=1}^n L_s^{(i)}(Y_k),
\end{equation}
where $Y_1,\ldots,Y_n \overset{iid}{\sim} \mu$ is our given data, 
\[L_s^{(i)}(x) = \E\Big[\int_{t_{i-1}}^{t_i} \lvert s(X_t,t) - \nabla \log q_t(x,X_t) \rvert^2 \diff{t}\Big],\]
and $\mathcal{S}_i$ is an approximating class of neural networks that needs to be chosen. 
The full score estimator is then obtained by concatenating these minimisers across time,
\begin{equation}\label{eq:estimator2}
    \widehat s_n(x,t) = \sum_{i=1}^K \widehat s_n^{(i)}(x,t)\1_{[t_{i-1},t_i)}(t),
\end{equation}
which reflects the multiscale structure of the diffusion and allows the approximation complexity to adapt to the effective noise level at time $t$. By the equivalence of explicit and denoising score matching, $\hat{s}_n^{(i)}$ therefore serves as an empirical risk minimiser for the true score $s_0$ on $[t_{i-1},t_i)$. For a given approximation class $\mathcal{S}_i$, the $L^2$ estimation error 
\[\E\Big[\int_{t_{i-1}}^{t_i} \lvert \hat{s}_n^{(i)}(X_t,t) - \nabla \log q_t(x,X_t) \rvert^2 \diff{t}\Big]\]
therefore naturally splits into the conflicting effects of an approximation error 
\[\min_{s \in \mathcal{S}_i} \E\Big[\int_{t_{i-1}}^{t_i} \lvert s(t,X_t) - \nabla \log q_t(x,X_t) \rvert^2 \diff{t}\Big],\]
which decreases with larger networks sizes that increase the expressivity of the approximation class, and a complexity term that increases with the size of the network class. 

\subsection{Score approximation}\label{subsec:approx}
In order to optimally balance these two effects, given a desired target accuracy, it is necessary to make parsimonious choices regarding the network sizes. To specify this, we now introduce the class of sparsity-constrained neural networks with ReLU activation function that we use for score approximation.
For $b, x\in\R^m$, define
\[
	\sigma_b(x)
	=\mat{\sigma(x_1-b_1) \\ \sigma(x_2-b_2) \\ \vdots \\ \sigma(x_m-b_m)},\quad \sigma(y)=y\vee0,
\]
and  for $L\in\N$, $W\in\N^{L+2}$, $S\in\N$ and $B>0$ denote by $\Phi(L, W, S, B)$ the class of neural networks with depth (i.e., number of hidden layers) $L$, layer widths (including input and output layers) $W$, sparsity constraint $S$, and norm constraint $B$.
We thus consider functions of the form
\[
	\varphi(x)
	=A_L\sigma_{b_L}A_{L-1}\sigma_{b_{L-1}}\cdots A_1\sigma_{b_1}A_0x,
\]
where $A_i\in\R^{W_{i+1}\times W_{i}},b_i\in\R^{W_{i+1}}$ for $i=0,\ldots,L$ (to ease notation, we always set $b_0=0$), and where there are at most a total of $S$ non-zero entries of the $A_i$'s and $b_i$'s and all entries are numerically at most $B$.
In an abuse of notation, we denote $\sigma_{0}$ simply by $\sigma$.
This can be written succinctly as
\[
	\Phi(L, W, S, B)
	\coloneqq\left\{
	\begin{aligned}
		&A_L\sigma_{b_L}A_{L-1}\sigma_{b_{L-1}}\cdots A_1\sigma_{b_1}A_0\mid A_i\in\R^{W_{i+1}\times W_i}, b_i\in\R^{W_{i+1}}, \\
		&\sum_{i=0}^L(\n{A_i}_0+\n{b_i}_0)\le S,\max_{i\in\{0, \ldots, L\}}(\n{A_i}_\infty\vee\n{b_i}_\infty)\le B
	\end{aligned}
	\right\}.
\]
For larger and more complicated neural networks, their exact sizes are often unavailable, and we only have access to their asymptotic sizes.
Due to this, we also introduce the following class of neural networks that eases network size analysis in the proofs that follow:
\[
	\widetilde{\Phi}(\widetilde{L}, \widetilde{W}, \widetilde{S}, \widetilde{B})
	\coloneqq \Big\{\varphi\in\Phi(L, W, S, B): L\lesssim\widetilde{L}, \n{W}_{\infty}\lesssim\widetilde{W}, S\lesssim\widetilde{S}\text{ and }B\lesssim\widetilde{B}\Big\}.
\]
With this notation, we have for arbitrary networks $\varphi_i\in\widetilde{\Phi}(L_i, W_i, S_i, B_i)$, that
\begin{align*}
	\varphi_1\circ\varphi_2
	&\in\widetilde{\Phi}(L_1+L_2, W_1\vee W_2, S_1+S_2, B_1\vee B_2)\qquad\text{and} \\
	\mat{\varphi_1 \\ \varphi_2}
	&\in\widetilde{\Phi}(L_1\vee L_2, W_1+W_2, S_1+S_2, B_1\vee B_2).
\end{align*}
In particular, since $\varphi_1+\varphi_2=\mat{1 & 1}\mat{\varphi_1 & \varphi_2}^\top$, we have also
\[
	\sum_{i=1}^{k}\varphi_i
	\in\widetilde{\Phi}\Big(\max{\{L_i\}}, \sum_{i=1}^{k}W_i, \sum_{i=1}^{k}S_i, \max{\{B_i\}}\Big).
\]
Some basic neural network approximation results that we shall frequently use in our analysis are given in Appendix \ref{app:basic_neural}. Our main approximation result is the following. 

\begin{theorem}
	\label{theo:score_approx}
	Under assumptions \ref{ass:H1}--\ref{ass:H4}, for any $\delta>0$, large enough $m\in\N$ and $\underline{t}>0$ with $m^{-\frac{2\alpha+2}{2\alpha+d}}\lesssim \underline{t}\lesssim \log m$, there exists a neural network
	\[
		\varphi_{s_0}
		\in \begin{cases}
			\widetilde{\Phi}\big((\log m)^2(\log\log m)^2, m(\log m)^{D+1},  m(\log m)^{D+2}, m^{\frac{\alpha}{d}}\underline{t}^{-1}\vee m^{\nu}\big), &\text{ if }\underline{t}\le\frac{1}{2}m^{-\frac{2-\delta}{d}}\\
			\widetilde{\Phi}\big((\log m)^2(\log\log m)^2, m'(\log m)^{D+1}, m'(\log m)^{D+2}, m'\big), &\text{ if }\underline{t}>\frac{1}{2}m^{-\frac{2-\delta}{d}}
		\end{cases},
	\]
	where $\nu=\frac{2d}{2\alpha-d}+\frac{1}{d}$ and $m'=\underline{t}^{-\frac{d}{2}}m^{\frac{\delta}{2}}$ satisfying
	\[
		\int_{\underline{t}}^{2\underline{t}}\mathbb{E}[|s_0(X_t)-\varphi_{s_0}(X_t)|^2]\diff{t}
		\lesssim \begin{cases}
			(\log m)^{d+2D+3} m^{-\frac{2\alpha}{d}}, &\text{ if }\underline{t}\le m^{-\frac{2-\delta}{d}}\\
			(\log m)^{d+2D+3} m^{-\frac{2(\alpha+1)}{d}}, &\text{ if }\underline{t}>m^{-\frac{2-\delta}{d}}.
		\end{cases}
	\]
	Moreover, this network can be chosen such that $|s_0(x, t)|\lesssim\frac{\sqrt{\log m}}{\sqrt{\underline{t}\wedge1}}$ for all $x\in[0, 1]^D$ and $t\in[\underline{t}, 2\underline{t}]$.
\end{theorem}
The proof is technically involved and proceeds through several stages. We provide a high-level overview of the argument here, while all details are deferred to Section~\ref{app:approx}. The approximation strategy exploits the explicit score representation established in Lemma~\ref{lem:explicit_score}, together with the general neural network approximation framework for space-time functions developed in \cite{holk24}.

Recall from Assumptions~\ref{ass:H1} and~\ref{ass:H2} that the target distribution $\mu$ is supported on a closed subset $M$ of a $d$-dimensional affine subspace $(V+v_0)\cap[0,1]^D$ with non-empty interior, where $V=\operatorname{Span}(v_1,\ldots,v_d)$, $v_0\in[0,1]^D$, and $v_1,\ldots,v_d\in\R^D$ are vectors in the $D$-dimensional ambient space.
Let $A\coloneqq (v_1,\ldots,v_d)\in\R^{D\times d}$ and $P\coloneqq AA^\top$, so that $P$ is the orthogonal projection onto $V$ and any $x\in V$ can be written as $x=Au$ for some $u\in\R^d$.
Then, for any integrable function $g\colon M\to\R$,
\[
\int_M g\diff\mu
=\int_{M^*} g(Au+v_0)p_0(Au+v_0)\diff u,
\]
where $M^*\coloneqq A^\top(M-v_0)=\{u\in\R^d:Au+v_0\in M\}$.
\begin{figure}[h]
    \centering
	\begin{subfigure}[b]{0.45\textwidth}
		\begin{tikzpicture}[
		scale = 3, 
		axis/.style={->, thick},
		cube/.style={dashed},
		support/.style={thick}
		]
			% Cube
			\draw[cube] (0, 0, 0) -- (1, 0, 0) -- (1, 0, 1) -- (0, 0, 1) -- cycle;
			\draw[cube] (0, 1, 0) -- (1, 1, 0) -- (1, 1, 1) -- (0, 1, 1) -- cycle;
			\draw[cube] (1, 0, 0) -- (1, 0.43, 0) ;
			\draw[cube] (0, 0, 1) -- (0, 0.36, 1) ;
			\draw[cube] (1, 0, 1) -- (1, 0.5, 1) ;

			% Axes
			\draw[axis] (0, 0, 0) -- (1.2, 0, 0);
			\draw[axis] (0, 0, 0) -- (0, 1.2, 0);
			\draw[axis] (0, 0, 0) -- (0, 0, 1.2);

			% Plane
			\path[draw=black, fill=black!20, thick, opacity=0.8] (1.2, 0.54, 1.2) -- (1.2, 0.44, -0.2) -- (-0.2, 0.24, -0.2) -- (-0.2, 0.34, 1.2) -- cycle;
			\node[shift={(-0.5, 0.25)}] at (1.2, 0.44, -0.2) {$V+v_0$};
			\draw[-, thick] (0, 0.29, 0) -- (0, 1, 0);
			\draw[cube] (1, 0.43, 0) -- (1, 1, 0);
			\draw[cube] (0, 0.36, 1) -- (0, 1, 1);
			\draw[cube] (1, 0.5, 1) -- (1, 1, 1);

			% Support
			\draw[support] (0.076, 0.299, 0.03) .. controls (-0.06, 0.298, 0.297) and (0.45, 0.374, 0.333) .. (0.096, 0.345, 0.636);
			\draw[support] (0.076, 0.299, 0.03) .. controls (-0.06, 0.298, 0.297) and (0.45, 0.374, 0.333) .. (0.096, 0.345, 0.636);
			\draw[support] (0.096, 0.345, 0.636) .. controls (-0.092, 0.332, 0.825) and (0.27, 0.386, 0.86) .. (0.54, 0.431, 0.955);
			\draw[support] (0.54, 0.431, 0.955) .. controls (0.927, 0.495, 1.082) and (0.562, 0.407, 0.57) .. (0.61, 0.412, 0.55);
			\draw[support] (0.61, 0.412, 0.55) .. controls (0.998, 0.454, 0.353) and (0.627, 0.366, -0.136) .. (0.617, 0.376, 0.03);
			\draw[support] (0.617, 0.376, 0.03) .. controls (0.61, 0.383, 0.14) and (0.235, 0.305, -0.2) .. (0.076, 0.299, 0.03);
			\node[] at (0.43, 0.40, 0.5) {$M$};
		\end{tikzpicture}
	\end{subfigure}
	\begin{subfigure}[b]{0.35\textwidth}
		\begin{tikzpicture}[
		scale = 3, 
		axis/.style={->, thick},
		cube/.style={dashed}
		]
			% Axes
			\draw[axis] (-0.1, 0) -- (1.1, 0);
			\draw[axis] (0, -0.1) -- (0, 1.1);

			% Support
			\draw[-] (0.076, 0.03) .. controls (-0.06, 0.297) and (0.45, 0.333) .. (0.096, 0.636);
			\draw[-] (0.096, 0.636) .. controls (-0.092, 0.825) and (0.27, 0.86) .. (0.54, 0.955);
			\draw[-] (0.54, 0.955) .. controls (0.927, 1.082) and (0.562, 0.57) .. (0.61, 0.55);
			\draw[-] (0.61, 0.55) .. controls (0.998, 0.353) and (0.627, -0.136) .. (0.617, 0.03);
			\draw[-] (0.617, 0.03) .. controls (0.61, 0.14) and (0.235, -0.2) .. (0.076, 0.03);
			\node[] at (0.45, 0.45) {$M^*$};
		\end{tikzpicture}
	\end{subfigure}
    \caption{Example of a domain $M\subseteq\R^3$ and its lower-dimensional representation $M^*\subseteq\R^2$.\label{fig:support}}
\end{figure}
Moreover, for any $x\in\R^D$ and $u\in M^*$,
\begin{align*}
	x-(Au+v_0)
	&=\big((v_0+P(x-v_0))-(Au+v_0)\big)+\big(x-(v_0+P(x-v_0))\big) \\
	&=\big(P(x-v_0)-Au\big)+\big((I-P)(x-v_0)\big),
\end{align*}
where the first term lies in $V$ and the second in $V^\perp$. By the Pythagorean theorem,
\[
	|x-(Au+v_0)|^2
	=|P(x-v_0)-Au|^2+|(I-P)(x-v_0)|^2.
\]
Consequently, for any $x\in\R^D$,
\begin{align*}
	\int_M \mathrm{e}^{-\frac{|x-y|^2}{2t}}\,\mu(\diff{y})
	&=\int_{M^*} \mathrm{e}^{-\frac{|x-(Au+v_0)|^2}{2t}}p_0(Au+v_0)\diff{u} \\
	&=\mathrm{e}^{-\frac{|(I-P)(x-v_0)|^2}{2t}}\int_{M^*}\mathrm{e}^{-\frac{|P(x-v_0)-Au|^2}{2t}}p_0(Au+v_0)\diff{u} \\
	&=\mathrm{e}^{-\frac{|(I-P)(x-v_0)|^2}{2t}}\int_{M^*}\mathrm{e}^{-\frac{|A^\top(x-v_0)-u|^2}{2t}}p_0(Au+v_0)\diff{u}.
\end{align*}
Here we used that $|Au|=|u|$ for all $u\in\R^d$.
A similar decomposition yields
\begin{align*}
	\int_M \frac{x-y}{t}\mathrm{e}^{-\frac{|x-y|^2}{2t}}\,\mu(\diff{y})
	=\mathrm{e}^{-\frac{|(I-P)(x-v_0)|^2}{2t}}&\bigg(\frac{(I-P)(x-v_0)}{t}\int_{M^*}\mathrm{e}^{-\frac{|A^\top(x-v_0)-u|^2}{2t}}p_0(Au+v_0)\diff{u} \\
	&\qquad+A\int_{M^*}\frac{A^\top(x-v_0)-u}{t}\mathrm{e}^{-\frac{|A^\top(x-v_0)-u|^2}{2t}}p_0(Au+v_0)\diff{u}\bigg).
\end{align*}
In view of Lemma~\ref{lem:explicit_score}, which gives
\[s_0(x,t)=-\frac{\sum_{z \in \mathbb{Z}^D}(-1)^z\int_{[0,1]^D}(R_z(x)+z-y)
			\exp\left(-\frac{\lvert R_z(x)+z-y\rvert^2}{2t}\right)\,\mu(\diff{y})}{
			t\sum_{z \in \mathbb{Z}^D}\int_{[0,1]^D}
			\exp\left(-\frac{\lvert R_z(x)+z-y\rvert^2}{2t}\right)
			\,\mu(\diff{y})}, \quad (x,t) \in [0,1]^D \times (0,\infty),\]
all dependence on $\mu$ enters through the lower-dimensional projection $A^\top(x-v_0)$.
Accordingly, a substantial part of the approximation task reduces to approximating the functions
\begin{equation}
	\label{eq:f_1f_2}
	f_1\colon(u, t)
	\mapsto \int_{M^*}\mathrm{e}^{-\frac{|u-v|^2}{2t}}p_0(Av+v_0)\diff{v},\qquad
	f_2\colon (u, t)
	\mapsto \int_{M^*}\frac{u-v}{t}\mathrm{e}^{-\frac{|u-v|^2}{2t}}p_0(Av+v_0)\diff{v},
\end{equation}
defined on $\R^d\times(0,\infty)$.
This dimensional reduction allows us to derive error bounds that depend on the intrinsic dimension $d$ rather than the ambient dimension $D$. 
A comparable mechanism appears in \cite{oko23}, where the Gaussian transition densities of the OU forward process give rise to analogous expressions.
In contrast, the transition densities in the present model are given by infinite series of Gaussian densities restricted to $[0,1]^D$, which introduces substantial additional technical difficulties.
These are addressed using an approximation strategy adapted from \cite{holk24}, where spectral representations of the forward density and its gradient were analysed.

The proof proceeds along the following steps:
\begin{enumerate}
    \item For fixed $\underline{t}>0$ and $t\in[\underline{t}, 2\underline{t}]$, we truncate the series representation of the forward density by 
    \[
p_t^K(x)
=(2\pi t)^{-\frac{D}{2}}
\sum_{\substack{z\in\mathbb{Z}^D \\ \|z\|_{\infty}\le \sqrt{2\underline{t}(D+2K)}}}
\int_{[0,1]^D}
\exp\Big(-\frac{|R_z(x)+z-y|^2}{2t}\Big)\,\mu(\diff{y})\]
	and define the corresponding truncated score 
\begin{equation}\label{eq:score_truncation}    
s_0^K(x,t)\coloneqq\frac{\nabla p_t^K(x)}{p_t^K(x)}
\end{equation}
Lemma~\ref{lem:truncation_error} establishes for $t\in[\underline{t}, 2\underline{t}]$ an exponential convergence of $s_0^K$ to $s_0$ in $L^2$ as $K\to\infty$ allowing us to restrict attention to truncation levels of order $\log n$.
	Since the truncated sums can be approximated termwise by neural networks, combining these approximations increases the network depth only by a logarithmic factor, leading to a negligible additional error.
\item For fixed $\delta>0$, we split the time domain into small- and large-time regimes according to $t\lessgtr m^{-(2-\delta)/d}$.
Using Lemma~\ref{lem:sobolev_network}, we approximate $f_1(\cdot,t)$ and $f_2(\cdot,t)$ in \eqref{eq:f_1f_2} for fixed $t$; see Lemmas~\ref{lem:fixed_time_approx_small_t} and~\ref{lem:fixed_time_approx_large_t}.
In the small-time regime, Assumption~\ref{ass:H3} yields improved approximation rates of order $m^{-\kappa/d}$ outside $M^*$, compared to $m^{-\alpha/d}$ on $M^*$ (up to logarithmic factors), while in the large-time regime the regularising effect of the forward diffusion leads to rates of order $m^{-(\kappa+1)/d}$.
\item We extend these fixed-time approximations to short time intervals using polynomial interpolation in $t$, as shown in Lemma~\ref{lem:time_interpolation}.
\item Finally, we combine the resulting constructions with the general neural network approximation results from Appendix~\ref{app:basic_neural} to prove Theorem~\ref{theo:score_approx}, which provides $L^2$ approximation bounds for the score in both time regimes.
\end{enumerate}

\subsection{Main result}
With these preparations, we are now in a position to give a concise proof our main result.

\begin{theorem}\label{theo:main}
	Assume \ref{ass:H1}--\ref{ass:H3}, and set
	\[
		\underline{T}
		=D^{-1}n^{-\frac{2(\alpha+1)}{2\alpha+d}}
		\quad\text{and}\quad
		\overline{T}
		=\frac{8}{\pi^2}\log\Big(\frac{8D^{\frac{3}{2}}}{\pi}n^{\frac{\alpha+1}{2\alpha+d}}\Big).
	\]
	Then, for any $\delta >  0$ and $n \in \mathbb{N}$ large enough, there exist neural network classes 
    \[\mathcal S_i=\Big\{\varphi\in\Phi(L, W_i, S_i, B)\colon |\varphi(x, t)|\lesssim \frac{\sqrt{\log n}}{\sqrt{t_{i}\wedge 1}}\Big\},\] 
    where
	\begin{align*}
		L
		&\lesssim \log n\log\log n, \\
		\n{W_i}_{\infty}
		&\lesssim \big(n^{\frac{d}{2\alpha+d}}\wedge [(t_i\wedge1)^{-\frac{d}{2}}n^{\frac{\delta d}{2\alpha+d}}]\big)(\log n)^{D+1}, \\
		S_i
		&\lesssim \big(n^{\frac{d}{2\alpha+d}}\wedge [(t_i\wedge1)^{-\frac{d}{2}}n^{\frac{\delta d}{2\alpha+d}}]\big)(\log n)^{D+2}, \\
		B
		&\lesssim n^{\frac{4(\alpha+1)+d(c_0-d+2)}{2(2\alpha+d)}}\vee n^{\frac{2d^2}{4\alpha^2-d^2}+\frac{1}{2\alpha+d}},
	\end{align*}
	such that the reflected diffusion generative algorithm associated to the empirical denoising score matching loss minimiser $\widehat{s}_n$ defined via \eqref{eq:estimator1} and \eqref{eq:estimator2} satisfies
	\[
		\mathbb{E}\big[\mathcal W_1(\mu, \overline{X}^{\widehat{s}_n}_{\overline{T}-\underline{T}})\big]
		\lesssim n^{-\frac{\alpha+1-\delta}{2\alpha+d}}.
	\]
\end{theorem}

\begin{proof}
	First, recalling the decomposition \eqref{eq:error_split}, it follows that
	\[
		\mathbb{E}\big[\mathcal W_1(\mu, \overline{X}^{\widehat{s}_n}_{\overline{T}-\underline{T}})\big]
		\le \mathcal W_1(\mu, X_{\underline{T}})+\mathbb{E}[\mathcal W_1(\widehat{X}^{\widehat{s}_n}_{\overline{T}-\underline{T}}, \overline{X}^{\widehat{s}_n}_{\overline{T}-\underline{T}})]+\mathbb{E}[\mathcal W_1(X_{\underline{T}}, \widehat{X}^{\widehat{s}_n}_{\overline{T}-\underline{T}})].
	\]
	Here, it immediately follows by Lemmas \ref{lem:early_stopping} and \ref{lem:stat_dist} that the first two terms are each bounded by $n^{-\frac{\alpha+1}{2\alpha+d}}$, and so we now focus on bounding $\mathbb{E}[\mathcal W_1(X_{\underline{T}}, \widehat{X}^{\widehat{s}_n}_{\overline{T}-\underline{T}})]$.
	By Lemma \ref{lem:wasserstein_Yi}, and its preceding discussion, we have for $\rho>0$ that
	\begin{align*}
		\mathbb{E}[\mathcal W_1(X_{\underline{T}}, \widehat{X}^{\widehat{s}_n}_{\overline{T}-\underline{T}})]
		&\lesssim K \mathrm{e}^{-\rho}+\sum_{i=1}^K \mathbb{E}\bigg[\bigg((t_i\wedge 1)\rho\int_{t_{i-1}}^{t_i}\mathbb{E}\big[|\widehat{s}_n(X_t, t)-\nabla\log p_t(X_t)|^2\mid \widehat{s}_n\big]\diff{t}\bigg)^{\frac{1}{2}}\bigg] \\
		&\le K \mathrm{e}^{-\rho}+\sum_{i=1}^K \bigg((t_i\wedge 1)\rho\,\mathbb{E}\bigg[\int_{t_{i-1}}^{t_i}\mathbb{E}\big[|\widehat{s}_n(X_t, t)-\nabla\log p_t(X_t)|^2\mid \widehat{s}_n\big]\diff{t}\bigg]\bigg)^{\frac{1}{2}},
	\end{align*}
	where in the last inequality we use Jensen's inequality.
	Recall here that $t_i=\underline{T}c^i$, where $c\in(1, 2]$ and $K\in\N$ are chosen such that $t_K=\overline{T}$, i.e. $K=\log_c\frac{\overline{T}}{\underline{T}}\asymp\log n$.
	Setting $\rho=\frac{\alpha+1}{2\alpha+d}\log n$, we thus have
	\[
		K \mathrm{e}^{-\rho}
		\asymp n^{-\frac{\alpha+1}{2\alpha+d}}\log n
		\lesssim n^{-\frac{\alpha+1-\delta}{2\alpha+d}}.
	\]
	Now, to further analyse each term, we fix $i\in[K]$ and introduce the induced function class $\mathcal L^{(i)}=\{L^{(i)}_s\mid s\in \mathcal S_i\}$.
	We then have by \cite[Theorem 3.4, Theorem B.2]{holk24} if $\sup_{s\in \mathcal S_i\cup\{s_0\}}\n{L_s^{(i)}}_{\infty}\le C(\mathcal L^{(i)})<\infty$ that for suitable $\Delta>0$,
	\begin{align*}
		\mathbb{E}&\bigg[\int_{t_{i-1}}^{t_i}\mathbb{E}\big[|\widehat{s}_n(X_t, t)-\nabla\log p_t(X_t)|^2\mid \widehat{s}_n\big]\diff{t}\bigg] \\
		&\le 2\inf_{s\in \mathcal S_i}\int_{t_{i-1}}^{t_i}\mathbb{E}\big[|s(X_t, t)-\nabla\log p_t(X_t)|^2\big]\diff{t}+2\frac{C(\mathcal L^{(i)})}{n}\Big(\frac{145}{9}\log \mathcal N(\mathcal L^{(i)}, \n{\cdot}_{\infty}, \Delta)+160\Big)+5\Delta.
	\end{align*}
	Following the proof of \cite[Lemma 3.5]{holk24}, we also have that if for $t\in[t_{i-1}, t_i)$, $\sup_{s\in \mathcal S_i}\n{s(\,\cdot, t)}_{\infty}\le \frac{C(\mathcal S_i)}{\sqrt{t\wedge 1}}$ for some $C(\mathcal S_i)<\infty$ and
	\begin{equation}
		\label{eq:bound_grad_log_q_t}
		\int_{t_{i-1}}^{t_i}\int_{[0, 1]^D}|\nabla \log q_t(x, y)|q_t(x, y)\,\diff{y}\diff{t}
		\lesssim \sqrt{t_i},
	\end{equation}
	then there exists a constant $c_0$ such that
	\begin{equation*}\label{eq:cov}
		\mathcal N(\mathcal L^{(i)}, \n{\cdot}_{\infty}, \delta)
		\le \mathcal N\Big(\mathcal S_i, \n{\cdot}_{\infty}, \frac{\Delta}{c_0 C(\mathcal S_i)(t_i+\sqrt{t_i})}\Big).
	\end{equation*}
	To this last condition, we have
	\begin{align*}
		\int_{[0, 1]^D}|\nabla\log q_t(x, y)|q_t(x, y)\,\diff{y}
		&=\int_{[0, 1]^D}|\nabla q_t(x, y)|\,\diff{y} \\
		&\le\frac{1}{t}\sum_{z\in\Z^D}\int_{[0, 1]^D}(2\pi t)^{-\frac{D}{2}}|R_z(x)+z-y|\mathrm{e}^{-\frac{|R_z(x)+z-y|^2}{2t}}\,\diff{y} \\
		&=\frac{1}{t}\int_{\R^D}(2\pi t)^{-\frac{D}{2}}|y|\mathrm{e}^{-\frac{|y|^2}{2t}}\,\diff{y} \\
		&\le \frac{\sqrt{D}}{\sqrt{t}},
	\end{align*}
	implying \eqref{eq:bound_grad_log_q_t}.
	Furthermore,  using the Li--Yau bound from \cite[Theorem 1.1]{liyau}, which gives $|\nabla \log q_t(x, y)|^2\lesssim\frac{D}{t}+\partial_t\log q_t(x, y)$, it follows that for $s\in \mathcal S_i$
	\begin{align*}
		L_s(x)
		&=\int_{t_{i-1}}^{t_i}\int_{[0, 1]^D}|s(y, t)-\nabla\log q_t(x, y)|^2q_t(x, y)\,\diff{y}\diff{t} \\
		&\le 2\int_{t_{i-1}}^{t_i}\int_{[0, 1]^D}\big(|s(y, t)|^2+|\nabla\log q_t(x, y)|^2\big)q_t(x, y)\,\diff{y}\diff{t} \\
		&\lesssim 2\int_{t_{i-1}}^{t_i}\int_{[0, 1]^D}\Big(\frac{C(\mathcal S_i)^2+D}{t}+\partial_t\log q_t(x, y)\Big)q_t(x, y)\,\diff{y}\diff{t} \\
		&=2\big(C(\mathcal S_i)^2+D\big)\log c,
	\end{align*}
	where we used that
	\[
		\int_{t_{i-1}}^{t_i}\int_{[0, 1]^D}\partial_t\log q_t(x, y)q_t(x, y)\,\diff{y}\diff{t}
		=\int_{[0, 1]^D}\int_{t_{i-1}}^{t_i}\partial_tq_t(x, y)\diff{t}\,\diff{y}
		=\int_{[0, 1]^D}q_{t_i}(x, y)-q_{t_{i-1}}(x, y)\,\diff{y}
		=0.
	\]
	Thus, it follows that $C(\mathcal L^{(i)})\lesssim (C(\mathcal S_i)^2+D)\log c\lesssim C(\mathcal S_i)^2$, and hence by the above
	\begin{align*}
		\mathbb{E}&\bigg[\int_{t_{i-1}}^{t_i}\mathbb{E}\big[|\widehat{s}_n(X_t, t)-\nabla\log p_t(X_t)|^2\mid \widehat{s}_n\big]\diff{t}\bigg] \\
		&\lesssim \inf_{s\in \mathcal S_i}\int_{t_{i-1}}^{t_i}\mathbb{E}\big[|s(X_t, t)-\nabla\log p_t(X_t)|^2\big]\diff{t}+\frac{C(\mathcal S_i)^2\log \mathcal N(\mathcal S_i, \n{\cdot}_{\infty}, \frac{\Delta}{c_0C(\mathcal S_i)(t_i+\sqrt{t_i})})}{n}+\Delta.
	\end{align*}
	Here, by a small modification of \cite[Lemma C.2]{oko23}, if $\mathcal S_i\subseteq\Phi(L, W, S, B)$, then
	\[
		\log \mathcal N\Big(\mathcal S_i, \n{\cdot}_{\infty}, \frac{\Delta}{c_0C(\mathcal S_i)(t_i+\sqrt{t_i})}\Big)
		\lesssim LS\log(c_0 C(\mathcal S_i)(t_i+\sqrt{t_i})\Delta^{-1}L\n{W}_{\infty}(B\vee 1)).
	\]
	Next, setting $m=\lceil n^{\frac{d}{2\alpha+d}}\rceil$, if $t_i\le n^{-\frac{2-\delta}{2\alpha+d}}$, we have by Theorem \ref{theo:score_approx} that there exists a neural network $\varphi_{s_0}^{(i)}\in\Phi(L, W, S, B)$ with
	\begin{alignat*}{2}
		L
		&\lesssim (\log m)^2(\log \log m)^2
		&&\lesssim \mathrm{Poly}(\log n) \\
		\n{W}_{\infty}
		&\lesssim m(\log m)^{D+1}
		&&\lesssim \mathrm{Poly}(n) \\
		S
		&\lesssim m(\log m)^{D+2}
		&&\lesssim n^{\frac{d}{2\alpha+d}}\,\mathrm{Poly}(\log n) \\
		B
		&\lesssim m^{\frac{2(\alpha+1)}{d}+\frac{c_0-d}{2}+1}\vee m^\nu
		&&\lesssim \mathrm{Poly}(n),
	\end{alignat*}
	satisfying $|\varphi_{s_0}^{(i)}|\lesssim\frac{\sqrt{\log n}}{\sqrt{t_{i}\wedge 1}}\le\frac{\sqrt{\log n}}{\sqrt{t\wedge1}}$ for $t\in[t_{i-1}, t_i)$ and
	\begin{align*}
		\int_{t_{i-1}}^{t_i}\mathbb{E}\big[|\varphi_{s_0}^{(i)}(X_t, t)-\nabla\log p_t(X_t)|^2\big]\diff{t}
		&\lesssim
		(\log m)^{2+2D+3}m^{-\frac{2\alpha}{d}}
		\le n^{-\frac{2\alpha}{2\alpha+d}}\,\mathrm{Poly}(\log n)
	\end{align*}
	Thus, setting $\mathcal S_i=\{\varphi\in\Phi(L, W, S, B)\mid |\varphi|\lesssim\frac{\sqrt{\log n}}{\sqrt{t_{i-1}}}\}$ and $\Delta=n^{-\frac{2\alpha}{2\alpha+d}}$, it follows that
	\[
		\mathbb{E}\bigg[\int_{t_{i-1}}^{t_i}\mathbb{E}\big[|\widehat{s}_n(X_t, t)-\nabla\log p_t(X_t)|^2\mid \widehat{s}_n\big]\diff{t}\bigg]
		\lesssim n^{-\frac{2\alpha}{2\alpha+d}}\,\mathrm{Poly}(\log n).
	\]
	Conversely, if $t_i>n^{-\frac{2-\delta}{2\alpha+d}}$, the same theorem yields a network $\varphi_{s_0}^{(i)}\in\Phi(L, W_i, S_i, B)$ with
	\begin{alignat*}{2}
		L
		&\lesssim (\log m)^2(\log \log m)^2
		&&\lesssim \mathrm{Poly}(\log n) \\
		\n{W_i}_{\infty}
		&\lesssim  (t_i\wedge 1)^{-\frac{d}{2}}m^{\frac{\delta}{2}}(\log m)^{D+1}
		&&\lesssim \mathrm{Poly}(n) \\
		S_i
		&\lesssim  (t_i\wedge 1)^{-\frac{d}{2}}m^{\frac{\delta}{2}}(\log m)^{D+2}
		&&\lesssim (t_i\wedge 1)^{-\frac{d}{2}}n^{\frac{\delta d}{2(2\alpha+d)}}\,\mathrm{Poly}(\log n) \\
		B_i
		&\lesssim  m^{\frac{2(\alpha+1)}{d}+\frac{c_0-d}{2}+1}
		&&\lesssim \mathrm{Poly}(n),
	\end{alignat*}
	satisfying
	\[
		\int_{t_{i-1}}^{t_i}\mathbb{E}\big[|\varphi_{s_0}^{(i)}(X_t, t)-\nabla\log p_t(X_t)|^2\big]\diff{t}
		\lesssim
		(\log m)^{2+2D+3}m^{-\frac{2(\alpha+1)}{d}}
		\le n^{-\frac{2(\alpha+1)}{2\alpha+d}}\,\mathrm{Poly}(\log n).
	\]
	Thus, setting $\Delta=n^{-\frac{2(\alpha+1)}{2\alpha+d}}$, we have now
	\[
		\mathbb{E}\bigg[\int_{t_{i-1}}^{t_i}\mathbb{E}\big[|\widehat{s}_n(X_t, t)-\nabla\log p_t(X_t)|^2\mid \widehat{s}_n\big]\diff{t}\bigg]
		\lesssim \Big(n^{-\frac{2(\alpha+1)}{2\alpha+d}}+\frac{(t_i\wedge1)^{-\frac{d}{2}}n^{\frac{\delta d}{2(2\alpha+d)}}}{n}\Big)\,\mathrm{Poly}(\log n).
	\]
	Now, letting $K^*=\max\{i\in[K]:t_i\le n^{-\frac{2-\delta}{2\alpha+d}}\}$, we have
	\begin{align*}
		\sum_{i=1}^K &\bigg((t_i\wedge 1)\rho\,\mathbb{E}\bigg[\int_{t_{i-1}}^{t_i}\mathbb{E}\big[|\widehat{s}_n(X_t, t)-\nabla\log p_t(X_t)|^2\mid \widehat{s}_n\big]\diff{t}\bigg]\bigg)^{\frac{1}{2}} \\
		&\lesssim \log n\bigg[\sum_{i=1}^{K^*} \bigg(t_i\,\mathbb{E}\bigg[\int_{t_{i-1}}^{t_i}\mathbb{E}\big[|\widehat{s}_n(X_t, t)-\nabla\log p_t(X_t)|^2\mid \widehat{s}_n\big]\diff{t}\bigg]\bigg)^{\frac{1}{2}} \\
		&\qquad\qquad\sum_{i=K^*+1}^{K} \bigg((t_i\wedge 1)\,\mathbb{E}\bigg[\int_{t_{i-1}}^{t_i}\mathbb{E}\big[|\widehat{s}_n(X_t, t)-\nabla\log p_t(X_t)|^2\mid \widehat{s}_n\big]\diff{t}\bigg]\bigg)^{\frac{1}{2}}\bigg].
	\end{align*}
	Here, the first sum is easily bounded by
	\[
		\sum_{i=1}^{K^*} \bigg(t_i\,\mathbb{E}\bigg[\int_{t_{i-1}}^{t_i}\mathbb{E}\big[|\widehat{s}_n(X_t, t)-\nabla\log p_t(X_t)|^2\mid \widehat{s}_n\big]\diff{t}\bigg]\bigg)^{\frac{1}{2}}
		\lesssim n^{-\frac{1-\delta/2}{2\alpha+d}}n^{-\frac{\alpha}{2\alpha+d}}\,\mathrm{Poly}(\log n) \\
		\lesssim n^{-\frac{\alpha+1-\delta}{2\alpha+d}},
	\]
	where we use that $K^*\le K\lesssim \log n$.
	As for the second sum, we have
	\begin{align*}
		\sum_{i=K^*+1}^{K} &\bigg((t_i\wedge 1)\,\mathbb{E}\bigg[\int_{t_{i-1}}^{t_i}\mathbb{E}\big[|\widehat{s}_n(X_t, t)-\nabla\log p_t(X_t)|^2\mid \widehat{s}_n\big]\diff{t}\bigg]\bigg)^{\frac{1}{2}} \\
		&\lesssim\mathrm{Poly}(\log n)\bigg(Kn^{-\frac{\alpha+1}{2\alpha+d}}+\frac{n^{\frac{\delta d}{4(2\alpha+d)}}}{\sqrt{n}}\sum_{i=K^*+1}^K(t_i\wedge1)^{\frac{2-d}{4}}\bigg) \\
		&\lesssim\mathrm{Poly}(\log n)\bigg(n^{-\frac{\alpha+1}{2\alpha+d}}+t_{K^*+1}^{\frac{2-d}{4}}\frac{n^{\frac{\delta d}{4(2\alpha+d)}}}{\sqrt{n}}\bigg) \\
		&\lesssim \mathrm{Poly}(\log n)\Big(n^{-\frac{\alpha+1}{2\alpha+d}}+n^{\frac{-(2-d)(2-\delta)+\delta d-2(2\alpha+d)}{4(2\alpha+d)}}\Big) \\
		&\lesssim n^{-\frac{\alpha+1-\delta}{2\alpha+d}}.
	\end{align*}
	Thus, it follows that
	\[
		\mathbb{E}[\mathcal W_1(X_{\underline{T}}, \widehat{X}^{\widehat{s}_n}_{\overline{T}-\underline{T}})]
		\lesssim n^{-\frac{\alpha+1-\delta}{2\alpha+d}},
	\]
	as desired.
\end{proof}

\section{Discussion}\label{sec:discussion}
We conclude by placing our results in the context of the existing statistical theory of diffusion-based generative models and by clarifying the scope and limitations of the present analysis.

\paragraph{Minimax optimal estimation}
We do \emph{not} aim to establish results on minimax optimality, as such questions, while important, lie beyond the scope of the present work. 
The first near-optimal convergence rates for diffusion models under the $1$-Wasserstein metric were obtained by \cite{oko23} who consider target distributions on $\R^D$ that admit an $\alpha$-smooth, compactly supported density wrt Lebesgue measure, and analyse classical Ornstein--Uhlenbeck (OU) dynamics. Under the assumption of perfect SDE simulation and working without the manifold hypothesis, they show that diffusion models achieve the rate $n^{-\frac{\alpha+1-\delta}{2\alpha+D}}$, for arbitrary $\delta>0$, which is near minimax-optimal in the ambient dimension $D$ and consistent with results established in the classical i.i.d.\ density estimation setting by \cite{niwebe22}; see in particular Theorem~1 therein for densities bounded away from zero.

In the context of diffusion-based generative models with non-singular target distributions, \cite{steph25} establish an upper bound of order $n^{-\frac{\alpha+1}{2\alpha+D}}$, up to polylogarithmic factors, for the $1$-Wasserstein distance between the true distribution and the law induced by the generative model. 
Their analysis uses empirical score matching over suitably chosen  neural network classes with $\tanh$-activation function and holds uniformly over classes of Hölder-$\alpha$ smooth densities that are bounded below on any compact subset of the interior of the support and exhibit controlled decay near its boundary. The avoidance of an arbitrarily small polynomial inefficiency in the  rate comes here at the price of having to consider polynomially many (in terms of the data) separate time intervals in the score estimation procedure as compared to the logarithmic dependence in our and previous work on 1-Wasserstein estimation rates \cite{oko23,tang24,azangulov24}.  This theoretical understanding has been further advanced by \cite{chak26}, who provide sharp finite-sample error bounds measured in the $p$-Wasserstein distance for arbitrary $p \ge 1$. Notably, they relax typical compact-support or smooth density assumptions, requiring only finite-moment conditions on the target distribution.

While these works constitute significant progress in the statistical theory of generative modelling, their analyses fundamentally rely on the explicit Gaussian structure of the transition kernels associated with unconstrained OU or standard Brownian dynamics on $\R^D$. In particular, \cite{steph25} and \cite{chak26} exploit a control on the temporal score regularity, which is  heavily tied to the specific analytical smoothing properties of the  Gaussian transition kernels. 
Our reflected diffusion framework on $[0,1]^D$ inherently breaks these structural properties. 
The corresponding transition densities are instead governed by an infinite series expansion of restricted Gaussian densities (or alternatively, by a spectral decomposition as in \cite{holk24}), which leads to a fundamentally different and analytically more complex score regularity, in particular due to the influence of boundary reflections, which become increasingly notable for larger $t$. 
Extending the minimax optimality proofs from the unconstrained setting to bounded domains with reflections would require entirely new analytic techniques to rigorously control these boundary effects and is consequently beyond the scope of this paper.

Furthermore, a strictly minimax-optimal convergence rate, completely free of extraneous logarithmic terms, has recently been derived by \cite{dou24} in the unconstrained variance-exploding setting. For target distributions with Hölder smoothness $\alpha > 0$, they establish rates wrt the score matching loss, which allows them to prove that the generated distribution achieves the minimax-optimal rate in terms of the expected squared total variation distance and, for $\alpha \ge 1$, the $1$-Wasserstein distance. 
However, these bounds are achieved by departing from neural network approximations and employing kernel-based score estimators instead. While this constitutes an important theoretical contribution to the understanding of the fundamental statistical limits of diffusion models, it diverges from common algorithmic practice, where gradient descent methods for score matching exploit the flexibility and inductive biases of expressive neural architectures. 
Our analysis maintains this connection by explicitly studying neural network-based score estimators. Moreover, from an analytical perspective, translating kernel-based techniques to reflected diffusions on bounded domains would introduce severe boundary biases, necessitating the construction of highly specialised boundary-correction kernels. Given the practical prevalence of neural networks in generative modelling and these domain-specific analytic challenges, a theoretical investigation of kernel-based score matching is conceptually distinct from our objectives and therefore falls outside the scope of the present work.

\paragraph{Extension to general manifolds}
In this work, we restrict our geometric setup to data supported on a linear subspace $d \ll D$ intersecting the hypercube. 
While simpler than general non-linear manifolds, this setting already poses significant mathematical challenges due to the aforementioned absence of Gaussian transition kernels. 
In unconstrained OU settings, this Gaussian structure allowed \cite{tang24, azangulov24} to extend convergence rates to unknown compact $d$-dimensional manifolds, achieving bounds that scale with the intrinsic dimension $d$.
Independently of any assumptions on the target distribution, the reflected diffusion case  requires managing complex transition densities (cf.\ Lemma \ref{lem:explicit_score}), making the bounding of the score function and its estimation considerably more involved. Let us emphasize here that all of the results given in Section \ref{sec:prob} do not rely on any specific geometric assumptions on the support of $\mu$ but only partially on the assumption that $p_0$ has controlled decay at the boundary.  For curved manifolds without boundary such that the density wrt the volume measure is bounded away from zero as in \cite{tang24,azangulov24} all statements therefore remain true if the lower bound in Lemma \ref{lem:affine_approx_d}   is replaced by $t^{-d/2}\mathrm{e}^{-\rho}$. In particular, parts of Lemma \ref{lem:affine_approx} provide a natural analogue to the crucial technical Lemma C.1 in \cite{tang24} and can therefore serve as a natural starting point for score approximation for general manifold data via considering linear sub-problems associated to local chart parametrisations of the data density.
% Extending our results to general curved manifolds would require controlling the interaction between manifold curvature and the reflecting boundary of the domain. This is a non-trivial problem in stochastic analysis requiring entirely new differential geometric tools and is therefore strictly beyond the scope of the present investigation.

\paragraph{Discretisation and sampling errors}
A further extension of the present analysis involves the relaxation of the assumption of exact simulation for the backward dynamics. In practice, sampling from the generative model requires the discretisation of the backward SDE. For standard unconstrained processes, this is typically achieved via the Euler--Maruyama scheme, whose error properties have been extensively quantified in recent literature.

For reflected diffusions, however, the discretisation is more involved, as simulated paths must be strictly constrained to the domain $[0,1]^D$. This necessitates the use of alternative schemes, such as projected or penalised Euler--Maruyama methods, which explicitly account for the local time at the boundary. The numerical analysis of these methods requires bounding the error in the presence of reflecting barriers, where the discretisation error is coupled with the approximation of the score function near the boundary. Establishing a comprehensive end-to-end bound that incorporates these discretisation effects is a significant objective in numerical stochastic analysis. This constitutes a natural direction for future research, as also pointed out in the survey article \cite{tang25}.
	
\paragraph{Acknowledgements} We thank Iskander Azangulov and Judith Rousseau for helpful discussions during the \textit{MFO Mini-Workshop: Statistical Challenges for Deep Generative Models} and Simon Bienewald for helpful remarks on Wasserstein generalisation bounds.

\printbibliography

\appendix
\section{Proofs for Section \ref{sec:prob}}\label{app:proofs_prob}

\begin{proof}[Proof of Lemma \ref{lem:explicit_solution}]
	Since $f(Y)=Y$ and $B_0=0$ a.s., we have $X_0 \sim \nu$.
	Fix $i \in [D]$ and define the one-dimensional process $Z_t^{(i)} \coloneqq B_t^{(i)} + Y^{(i)}$, $t \geq 0$.
	For $x \in \mathbb{R}$, let $L^{(i)}_{t,x}$ denote the local time of $Z^{(i)}$ at level $x$, that is, the unique process satisfying
	\[
	L^{(i)}_{t,x} = \lim_{\varepsilon \searrow 0}
	\frac{1}{2\varepsilon} \int_0^t \mathbf{1}_{(x-\varepsilon,x+\varepsilon)}(Z_s^{(i)})\diff{s} .
	\]
	Since $\widehat{f}$ is the difference of two convex functions, the It\^o--Tanaka formula implies that $X^{(i)}_t \coloneqq \widehat{f}(Z_t^{(i)})$ is a continuous semimartingale satisfying
	\[
	X_t^{(i)}= Y^{(i)} + \int_0^t \widehat{f}'_{-}(Z_s^{(i)})\diff{Z_s^{(i)}}
	+ \frac{1}{2} \int_{\mathbb{R}} L^{(i)}_{t,x}\,\widehat{f}''(\diff{x}),
	\]
	where $\widehat{f}'_{-}$ denotes the left derivative of $\widehat{f}$ and $\widehat{f}''$ its distributional second derivative.
	Define $W_t^{(i)} \coloneqq \int_0^t \widehat{f}'_{-}(Z_s^{(i)})\diff{Z_s^{(i)}}$.
	As $\widehat{f}'_{-}(x) \in \{-1,1\}$ for all $x \in \mathbb{R}$, we obtain
	\[
	\langle W^{(i)} \rangle_t = \int_0^t \widehat{f}'_{-}(Z_s^{(i)})^2\diff{\langle B^{(i)}\rangle_s}=\int_0^t 1\diff{s} = t,
	\]
	and hence $(W_t^{(i)})_{t \geq 0}$ is a standard Brownian motion by L\'evy's characterisation.
	Next, define
	\[
	L_t^{(i),0} \coloneqq \sum_{k \in \mathbb{Z}} L^{(i)}_{t,2k},
	\quad
	L_t^{(i),1} \coloneqq \sum_{k \in \mathbb{Z}} L^{(i)}_{t,2k+1}.
	\]
	Then, $\int_{\mathbb{R}} L^{(i)}_{t,x}\,\widehat{f}''(\diff{x}) = L_t^{(i),0} - L_t^{(i),1}$.
	Let $T_k \coloneqq \inf\{t \geq 0 : |Z_t^{(i)}| \geq k\}$, $k\in\mathbb Z$, and denote by $L^{X^{(i)}}$ the local time of $X^{(i)}$ at $0$.
	For all $t \geq 0$ and $n \in \mathbb{N}$, we have a.s.
	\begin{align*}
		L^{X^{(i)}}_{t \wedge T_{2n}}
		= \lim_{\varepsilon \searrow 0}
		\frac{1}{2\varepsilon}
		\int_0^{t \wedge T_{2n}} \mathbf{1}_{(-\varepsilon,\varepsilon)}(X_s^{(i)})\diff{s}
		 &= \lim_{\varepsilon \searrow 0} \frac{1}{2\varepsilon} \sum_{k \in \mathbb{Z}, \lvert k \rvert \leq n} \int_0^{t \wedge T_{2n}} \one_{(2k - \varepsilon,2k + \varepsilon)}(Z^{(i)}_s) \diff{s}\\
		&= \sum_{k \in \mathbb{Z},\,|k| \leq n}
		\lim_{\varepsilon \searrow 0}
		\frac{1}{2\varepsilon}
		\int_0^{t \wedge T_{2n}} \mathbf{1}_{(2k-\varepsilon,2k+\varepsilon)}(Z_s^{(i)})\diff{s} \\
		&= \sum_{k \in \mathbb{Z},\,|k| \leq n} L^{(i)}_{t \wedge T_{2n},2k},
	\end{align*}
	where we used that $X_s^{(i)}=0$ if and only if $Z_s^{(i)}=2k$ for some $k \in \mathbb{Z}$.
	By monotone convergence,
	\[
	L_t^{X^{(i)}} = \sum_{k \in \mathbb{Z}} L^{(i)}_{t,2k} = L_t^{(i),0},
	\quad \text{for all } t \geq 0 \text{ a.s.}
	\]
	An analogous argument shows that $L_t^{(i),1}$ is the local time of $X^{(i)}$ at $1$.
	Consequently, $X^{(i)}$ satisfies the one-dimensional reflected SDE
	\[
	\diff{X_t}^{(i)}
	= \diff{W_t^{(i)}}
	+ \frac{1}{2}\bigl(\diff{L_t^{(i),0}} - \diff{L_t^{(i),1}}\bigr).
	\]
	Finally, define
	$L_t \coloneqq \frac{1}{2}\sum_{i=1}^D \bigl(L_t^{(i),0} + L_t^{(i),1}\bigr)$ and $W_t \coloneqq (W_t^{(1)},\ldots,W_t^{(D)})$.
	Then,
	\begin{align*}
			X_t
			&=X_0+W_t+\frac{1}{2}\sum_{i=1}^{D}e_i\Big(L_t^{(i), 0}-L_t^{(i), 1}\Big) \\
			&=X_0+W_t+\sum_{i=1}^{D}\int_{s \in [0,t]: \{X^{(i)}_s\in\{0, 1\}\}}n(X_t)\diff{L_s}^{(i), X_s^{(i)}} \\
			&=X_0+W_t+\int_{s \in [0,t]:\bigcup_{i=1}^D\{X^{(i)}_s\in\{0, 1\}\}}n(X_s)\,\mathrm{d}\Big(\frac{1}{2}\sum_{j=1}^{D}L_s^{(j), X_s^{(j)}}\Big) \\
			&=X_0+W_t+\int_{s \in [0,t]:\{X_s\in\partial[0, 1]^D\}}n(X_s)\diff{L_s} =X_0+W_t+\int_{0}^tn(X_s)\diff{L_s}.
		\end{align*}
	Moreover, for $i,j \in [D]$,
	\[
	\langle W^{(i)}, W^{(j)} \rangle_t
	=\int_{0}^{t}\widehat{f}'_{-}(Z_s^{(i)})\widehat{f}'_{-}(Z_s^{(j)})\diff{\langle B^{(i)}, B^{(j)}\rangle_s}=
	\begin{cases}
		t, & i=j,\\
		0, & i\neq j,
	\end{cases}
	\]
	so that $(W_t)_{t \geq 0}$ is a $D$-dimensional Brownian motion.
	Since $(L_t)_{t\ge0}$ is a local time of $(X_t)_{t\ge0}$, this completes the proof.
\end{proof}

\begin{proof}[Proof of Lemma \ref{lem:explicit_score}]
Let $f$ be defined as in Lemma~\ref{lem:explicit_solution}.
By that lemma and by uniqueness in law of weak solutions to \eqref{eq:forward_equation}, the associated transition density $q_t$ satisfies
\[
q_t(y,x)\,\diff{x}= \mathbb{P}(X_t \in \diff{x} \mid X_0=y)
	= \mathbb{P}\bigl(f(B_t+y) \in \diff{x}\bigr).
\]
For any $z \in \mathbb{Z}^D$ and $x \in [0,1]^D + z$, we have $f(x)=R_z(x-z)$.
Since the collection $\bigl([0,1]^D + z\bigr)_{z \in \mathbb{Z}^D}$ forms a partition of $\mathbb{R}^D$ up to Lebesgue null sets, it follows that for $x,y \in [0,1]^D$,
\begin{align*}
q_t(y,x)\,\diff{x}&= \sum_{z \in \mathbb{Z}^D}\mathbb{P}\bigl(f(B_t+y) \in \diff{x},\; B_t+y \in [0,1]^D+z\bigr) \\
&= \sum_{z \in \mathbb{Z}^D}\mathbb{P}\bigl(R_z(B_t+y-z) \in \diff{x},\; B_t+y \in [0,1]^D+z\bigr) \\
&= \sum_{z \in \mathbb{Z}^D}
\mathbb{P}\bigl(B_t \in R_z(\diff{x})+z-y\bigr)= (2\pi t)^{-D/2}
\sum_{z \in \mathbb{Z}^D}\exp\left(-\frac{\lvert R_z(x)+z-y\rvert^2}{2t}\right)\diff{x}.
\end{align*}
In the third equality, we used that, for each $z \in \mathbb{Z}^D$, the mapping $R_z$ is an involution on $[0,1]^D$, and that
$R_z(\diff{x})+z-y \subseteq [0,1]^D+z-y$ for all $x \in [0,1]^D$. The fourth equality follows from the transformation theorem and the fact that the determinant of Jacobian $J_{R_z}(x) = (-1)^z$ has unit absolute value. 
Applying monotone convergence and integrating $q_t(y,x)$ against $\mu(\diff{y})$ yields
\[
p_t(x)=(2\pi t)^{-D/2}\sum_{z \in \mathbb{Z}^D}\int_{[0,1]^D}
	\exp\left(-\frac{\lvert R_z(x)+z-y\rvert^2}{2t}\right)	\,\mu(\diff{y}),
\]
as claimed.
To compute the score, note that for $z \in \mathbb{Z}^D$ and $x,y \in [0,1]^D$, the chain rule gives
\[
\nabla \e^{-\frac{\lvert R_z(x)+z-y\rvert^2}{2t}}= -\frac{\nabla \lvert R_z(x)+z-y\rvert^2}{2t}\e^{-\frac{\lvert R_z(x)+z-y\rvert^2}{2t}}= -(-1)^z\frac{R_z(x)+z-y}{t}
	\e^{-\frac{\lvert R_z(x)+z-y\rvert^2}{2t}}.
\]
Hence, dominated convergence again yields
\[
s_0(x, t)=\frac{\nabla p_t(x)}{p_t(x)}
	=-\frac{\sum_{z\in\Z^D}(-1)^z\int_{[0, 1]^D}(R_z(x)+z-y)\mathrm{e}^{-\frac{|R_z(x)+z-y|^2}{2t}}\,\mu(\diff{y})}{t\sum_{z\in\Z^D}\int_{[0, 1]^D}\mathrm{e}^{-\frac{|R_z(x)+z-y|^2}{2t}}\,\mu(\diff{y})},
\]
as desired.
\end{proof}

\begin{proof}[Proof of Lemma \ref{lem:affine_approx}]
	To show \hyperref[lem:affine_approx_a]{(a)}, notice that, by the triangle inequality,
	\[
		t|s_0(x, t)|
		\le\frac{\sum_{z\in\Z^D}\int_{[0, 1]^D}|R_z(x)+z-y|\mathrm{e}^{-\frac{|R_z(x)+z-y|^2}{2t}}\,\mu(\diff{y})}{\sum_{z\in\Z^D}\int_{[0, 1]^D}\mathrm{e}^{-\frac{|R_z(x)+z-y|^2}{2t}}\,\mu(\diff{y})},
	\]
	where, since the numerator is convergent, we must have for every $\varepsilon>0$ and $x, y\in[0, 1]^D$ that there exists a $K>0$ such that
	\[
		\sum_{\substack{z\in\Z^D \\ |R_z(x)+z-y|>K}}|R_z(x)+z-y|\mathrm{e}^{-\frac{|R_z(x)+z-y|^2}{2t}}
		<\varepsilon.
	\]
	In particular, there exists a $K(x, y, t)>0$ such that
	\[
		\sum_{\substack{z\in\Z^D \\ |R_z(x)+z-y|>K(x, y, t)}}|R_z(x)+z-y|\mathrm{e}^{-\frac{|R_z(x)+z-y|^2}{2t}}
		\le\sum_{z\in\Z^D}\mathrm{e}^{-\frac{|R_z(x)+z-y|^2}{2t}}.
	\]
	Thus, if we can show that $K(x, y, t)\le K(t)$, independent of $x$ and $y$, integrating both sides yields
	\begin{align*}
		\sum_{z\in\Z^D}\int_{[0, 1]^D}|R_z(x)+z-y|\mathrm{e}^{-\frac{|R_z(x)+z-y|^2}{2t}}\,\mu(\diff{y})
		&=\int_{[0, 1]^D}\bigg(\sum_{\substack{z\in\Z^D \\ |R_z(x)+z-y|\le K(t)}}|R_z(x)+z-y|\mathrm{e}^{-\frac{|R_z(x)+z-y|^2}{2t}} \\
		&\qquad+\sum_{\substack{z\in\Z^D \\ |R_z(x)+z-y|>K(t)}}|R_z(x)+z-y|\mathrm{e}^{-\frac{|R_z(x)+z-y|^2}{2t}}\bigg)\,\mu(\diff{y}) \\
		&\le\big(K(t)+1\big)\int_{[0, 1]^D}\sum_{z\in\Z^D}\mathrm{e}^{-\frac{|R_z(x)+z-y|^2}{2t}}\,\mu(\diff{y}),
	\end{align*}
	which then implies 
    \begin{equation}\label{eq:score_bound1}
    |s_0(x, t)|\le\frac{K(t)+1}{t}.
    \end{equation}
	To do this, we first note that, for each $x, y\in[0, 1]^D$ and $z\in\Z^D$, the point $R_z(x)+z-y$ lies in $[-1, 1]^D+z$, and that the function $|u|\mathrm{e}^{-\frac{|u|^2}{2t}}$ is decreasing in $|u|$ whenever $|u|>\sqrt{t}$.
	Thus, when $|z|>2\sqrt{D}+\sqrt{t}$, we have the rough estimate of
	\[
		|R_z(x)+z-y|\mathrm{e}^{-\frac{|R_z(x)+z-y|^2}{2t}}
		\le\int_{[-2, 2]^D}|z+u|\mathrm{e}^{-\frac{|z+u|^2}{2t}}\diff{u}.
	\]
	Now, notice that the set $\{z+[-2, 2]^D\mid z\in\Z^D\}$ constitutes a covering of $\R^D$, where each hypercube $z+[0, 1]^D$ is covered a total of $4^D$ times.
	Thus, we have for each integrable $g\colon \R^D\to\R$,
	\[
		\sum_{z\in\Z^D}\int_{[-2, 2]^D}g(u+z)\diff{u}
		=4^D\int_{\R^D}g(u)\diff{u}.
	\]
	Similarly, the set $\{z+[-2, 2]^D\mid z\in\Z^D, |z|>K\}$ is a covering of a subset of $\R^D$ containing $\{u\in\R^D\mid |u|>K-2\sqrt{D}\}$, and each hypercube $z+[0, 1]^D$ is covered at most $4^D$ times, whence
	\[
		\sum_{\substack{z\in\Z^D \\ |z|>K}}\int_{[-2, 2]^D}g(u+z)\diff{u}
		\le 4^D\int_{\{|u|>K-2\sqrt{D}\}}g(u)\diff{u}.
	\]
	Applying this to the above for some $K>2\sqrt{D}+\sqrt{t}$, we get
	\begin{align*}
		\sum_{\substack{z\in\Z^D \\ |z|>K}}|R_z(x)+z-y|\mathrm{e}^{-\frac{|R_z(x)+z-y|^2}{2t}}
		&\le \sum_{\substack{z\in\Z^D \\ |z|>K}}\int_{[-2, 2]^D}|u+z|\mathrm{e}^{-\frac{|u+z|^2}{2t}}\diff{u} \\
		&\le 4^D\int_{\{|u|>K-2\sqrt{D}\}}|u|\mathrm{e}^{-\frac{|u|^2}{2t}}\diff{u} \\
		&=(32\pi t)^{\frac{D}{2}}\mathbb{E}[|B_t|\bm{1}_{\{|B_t|>K-2\sqrt{D}\}}].
	\end{align*}
	Setting $K(t)=2\sqrt{D}+\sqrt{t(D+\frac{D}{t})}\lesssim1+ \sqrt{t} \vee1$, it follows by Lemma \ref{lem:brownian_bound_b} that
	\[
		\mathbb{E}[|B_t|\bm{1}_{\{|B_t|>K(t)-2\sqrt{D}\}}]
		\lesssim t^{-\frac{D}{2}}\mathrm{e}^{-\frac{D}{2t}},
	\]
	whence
	\[
		\sum_{\substack{z\in\Z^D \\ |z|>K(t)}}|R_z(x)+z-y|\mathrm{e}^{-\frac{|R_z(x)+z-y|^2}{2t}}
		\lesssim \mathrm{e}^{-\frac{D}{2t}}
		\le \mathrm{e}^{-\frac{|x-y|^2}{2t}}
		\le \sum_{z\in\Z^D}\mathrm{e}^{-\frac{|R_z(x)+z-y|^2}{2t}}.
	\]
	By the previous discussion discussion leading to \eqref{eq:score_bound1}, this proves part \hyperref[lem:affine_approx_a]{(a)}.
    
	To prove \hyperref[lem:affine_approx_b]{(b)}, we first note that \hyperref[lem:affine_approx_a]{(a)} gives
	\[
		\mathbb{E}[|\nabla\log p_t(X_t)|^2\bm{1}_{M_{\rho, t}^{\mathsf{c}}}(X_t)]
		\lesssim \frac{1}{t^2\wedge1}\mathbb{P}(X_t\in M_{\rho, t}^{\mathsf{c}}).
	\]
	To control the right hand side, let $f, \widehat{f}$ be as in Lemma \ref{lem:explicit_solution} such that $X_t\sim f(B_t+Y)$.
	For all $x\in\R$ and $y\in[0, 1]$, we have $|\widehat{f}(x+y)-y|\le|x|$, implying that $|f(B_t+Y)-Y|\le|B_t|$, and hence that if $|B_t|\le\sqrt{t(D+2\rho)}$, then also $\mathrm{dist}(f(B_t+Y), M)\le\sqrt{t(D+2\rho)}$.
	Using this and Lemma \ref{lem:brownian_bound_a}, we have
	\[
		\mathbb{P}(X_t\in M_{\rho, t}^{\mathsf{c}})
		\le\mathbb{P}(|B_t|>\sqrt{t(D+2\rho)})
		\lesssim \mathrm{e}^{-\rho},
	\]
	showing \hyperref[lem:affine_approx_b]{(b)}.
	% {\color{blue}
	
	We continue with the proof of \hyperref[lem:affine_approx_b]{(c)}. Using the same reasoning as above, we find for $Z_u \coloneq B^{(1)}_{\mathrm{e}^{-u}}/\sqrt{\mathrm{e}^{-u}}$  and $z > 0$ that
	\begin{equation}\label{eq:range}
		\PP\big(\exists s \in [t, 1]: \lvert X_s - X_0 \rvert > \sqrt{s}z\big) 
		\leq \PP\Big(\sup_{s \in [t, 1]} \lvert B_s/\sqrt{s} \rvert > z\Big)
		\leq D\PP\Big(\sup_{u \in [0,\log t^{-1}]} \lvert Z_u \rvert > z/ \sqrt{D}\Big).
	\end{equation}
	Note that $(Z_u)_{u \in [0,\log t^{-1}]}$ is a Gaussian process with canonical distance $d_Z$ given by
	\[
		d_Z(u,v)^2
		\coloneq \E\big[\lvert Z_u - Z_v \rvert^2\big] 
		= 2\big(1 - \tfrac{\mathrm{e}^{-(u \vee v)}}{\mathrm{e}^{-(u+v)/2}}\big) 
		= 2\big(1 - \mathrm{e}^{-\lvert u - v \rvert/2}\big).
	\]
	It is readily verified that for $\varepsilon \in (0,\sqrt{2})$ the covering number $N([0,\log t^{-1}], d_Z, \varepsilon)$ is bounded by 
	\[
		N([0,\log t^{-1}],d_Z,\varepsilon) 
		\leq 1 + \frac{\log t^{-1}}{-2 \log(1- \varepsilon^2/2)} 
		\leq 1 + \frac{\log t^{-1}}{\varepsilon^2}
	\]
	and thus we obtain for the entropy integral 
	\begin{align*} 
		\int_0^\infty \log N([0, \log t^{-1}], d_Z, \varepsilon) \diff{\varepsilon} 
		&\leq  \int_0^{\sqrt{2}} \log\Big(1 + \frac{\log t^{-1}}{\varepsilon^2}\Big) \diff{\varepsilon} \\ 
		&\leq \sqrt{2} \log (2 + \log t^{-1}) + 2 \int_0^{\sqrt{2}} \log(1/\varepsilon) \diff{\varepsilon} \\ 
		&\leq \sqrt{2} \log (2 + \log t^{-1}) + 2\\ 
		&\leq C_1\log (1 + \log t^{-1}),
	\end{align*}
	for some universal constant $C_1$.
	Moreover, we find for the diameter 
	\[
		\Delta_{d_Z}([0,\log t^{-1}]) 
		\coloneq \sup_{s, s' \in [0,\log t^{-1}]} d_Z(s,s') 
		\leq \sqrt{2}.
	\]  
	Thus, Dudley's entropy concentration bound for suprema of Gaussian processes, cf.\, e.g., \cite[Remark 8.1.6]{vershynin18}, yields for any $y > 0$ and some constant $C_1 > 0$ that
	\[
		\sup_{ u \in [0,\log t^{-1}]} \lvert Z_u - Z_0 \rvert 
		\lesssim \int_0^\infty \log N([0, \log t^{-1}], d_Z, \varepsilon) \diff{\varepsilon} + \Delta_{d_Z}([0,\log t^{-1}]) u \leq C_1\log (1 + \log t^{-1}) + \sqrt{2}y,
	\]
	with probability larger than $1 - 2\mathrm{e}^{-y^2/2}$.
	Since $Z_0$ is standard normal, we therefore conclude that there exists a universal constant $C \geq 1$ such that 
	\[
		\sup_{ u \in [0,\log t^{-1}]} \lvert Z_u \rvert 
		\leq C(\log (1 + \log t^{-1}) + y),
	\]
	with probability larger than $1 - 4 \mathrm{e}^{-2y^2}$. 
	Consequently, it follows from \eqref{eq:range} that 
	\[
		\PP\Big(\forall s \in [t,1]: \lvert X_s - X_0 \rvert \leq C\sqrt{Ds}\big(\log (1 + \log t^{-1}) + y\big)\Big) \geq 1 - 4D\mathrm{e}^{-2y^2},
	\]
	which proves \hyperref[lem:]{(c)}.
	
	To show \hyperref[lem:affine_approx_d]{(d)}, let $x\in M_{\rho, t}$ be given and choose some $y_0\in\overline{M}$ with $|y_0-x|\le\sqrt{t(D+2\rho)}$.
	Then, set $y_1=y_0+\frac{\sqrt{t}}{|y_0-x|}(y_0-x)$ such that $y_1$ lies on the line containing $x$ and $y_0$, but a distance of $\sqrt{t}$ further away from $x$ than $y_0$.
	Then, $\mathcal B(y_0, \sqrt{t})\subseteq \mathcal B(y_1, 2\sqrt{t})$ while $|x-y|\le|x-y_1|+|y_1-y|\le \sqrt{t(D+3+2\rho)}$ for all $y\in \mathcal B(y_0, \sqrt{t})$.
	Thus for all such $y$ we have $\mathrm{e}^{\frac{|x-y|^2}{2t}}\gtrsim \mathrm{e}^{-\rho}$, and it follows by Assumption \ref{ass:H1} that 
	\[
		\int_{[0, 1]^D}\mathrm{e}^{-\frac{|x-y|^2}{2t}}\,\mu(\diff{y})
		\ge \mathrm{e}^{-\frac{D+3+2\rho}{2}}\mu\big(\mathcal B(y_0, \sqrt{t})\big)
		\gtrsim t^{\frac{c_0}{2}}\mathrm{e}^{-\rho}.
	\]
	This implies that for such $x$, the same is true of $(2\pi t)^{\frac{D}{2}}p_t(x)$, showing \hyperref[lem:affine_approx_d]{(d)}.

	Finally, we prove part \hyperref[lem:affine_approx_e]{(e)}. We first show part \ref{lem:loggrad_score_e1}, that is $|\nabla_x\log q_t(y, x)|\lesssim\frac{|x-y|}{t}+\frac{1}{\sqrt{t}}$ for all $x, y\in[0, 1]^D$.
	To this end, for such $x, y$ let $Z_1(x, y)=\{z\in\Z^D:|R_z(x)+z-y|\le\sqrt{2}|x-y|\}$ and $Z_2=\Z^D\setminus Z_1$.
	Then,
	\begin{align*}
		|\nabla_x\log q_t(y, x)|
		&\le\frac{\sum_{z\in\Z^D}\frac{|R_z(x)+z-y|}{t}\mathrm{e}^{-\frac{|R_z(x)+z-y|^2}{2t}}}{\sum_{z\in\Z^D}\mathrm{e}^{-\frac{|R_z(x)+z-y|^2}{2t}}} \\
		&=\frac{\sum_{z\in Z_1(x, y)}\frac{|R_z(x)+z-y|}{t}\mathrm{e}^{-\frac{|R_z(x)+z-y|^2}{2t}}}{\sum_{z\in\Z^D}\mathrm{e}^{-\frac{|R_z(x)+z-y|^2}{2t}}} +
		\frac{\sum_{z\in Z_2(x, y)}\frac{|R_z(x)+z-y|}{t}\mathrm{e}^{-\frac{|R_z(x)+z-y|^2}{2t}}}{\sum_{z\in\Z^D}\mathrm{e}^{-\frac{|R_z(x)+z-y|^2}{2t}}} \\
		&\le\frac{\sqrt{2}|x-y|}{t}+\frac{\sum_{z\in Z_2(x, y)}\frac{|R_z(x)+z-y|}{t}\mathrm{e}^{-\frac{|R_z(x)+z-y|^2}{2t}}}{\sum_{z\in\Z^D}\mathrm{e}^{-\frac{|R_z(x)+z-y|^2}{2t}}}.
	\end{align*}
	Now, note that for $z\in Z_2(x, y)$ we have $|R_z(x)+z-y|^2-|x-y|^2 > \frac{1}{2}|R_z(x)+z-y|^2$, whence
	\begin{align*}
		\frac{\sum_{z\in Z_2(x, y)}\frac{|R_z(x)+z-y|}{t}\mathrm{e}^{-\frac{|R_z(x)+z-y|^2}{2t}}}{\sum_{z\in\Z^D}\mathrm{e}^{-\frac{|R_z(x)+z-y|^2}{2t}}}
		&\le\sum_{z\in Z_2(x, y)}\frac{|R_z(x)+z-y|}{t}\mathrm{e}^{\frac{|x-y|^2-|R_z(x)+z-y|^2}{2t}} \\
		&\le2\sum_{z\in\Z^D}\frac{|R_z(x)+z-y|}{2t}\mathrm{e}^{-\frac{|R_z(x)+z-y|^2}{4t}}.
	\end{align*}
	Now, to evaluate this final sum, let $f_t(x)=\frac{|x|}{t}\mathrm{e}^{-\frac{|x|^2}{2t}}$, and note that for all $z\in\Z^D$, we have
	\[
		\#\big\{z'\in\Z^D:R_{z'}(x)+z'-y\in[0, 1]^D+z\big\}
		\le 2^D,
	\]
	whereby
	\[
		\sum_{z\in\Z^D}f_{2t}(R_z(x)+z-y)
		\le 2^D\sum_{z\in\Z^D}\sup_{x\in[0, 1]^D+z}f_{2t}(x)
	\]
	Since $f_t(x)\le f_t(\sqrt{t}\frac{x}{|x|})\le\frac{1}{\sqrt{t}}$, it follows that for the $2^D$ terms where $[0, 1]^D+z\subseteq[-1, 1]^D$, we have $\sup_{x\in[0, 1]^D+z}f_{2t}(x) \leq \frac{1}{\sqrt{2t}}$, while for all others, it is the point in $[0, 1]^D+z$ closest to the origin since $f_t$ is decreasing in $|x|$ for $|x|>\sqrt{t}$.
	The set of all such points is merely $\Z^D\setminus\{0\}$ with points near the axes being repeated a maximum of $2^D$ times, whence
	\begin{equation}\label{eq:log_grad_trans1}
		\sum_{z\in\Z^D}\frac{|R_z(x)+z-y|}{2t}\mathrm{e}^{-\frac{|R_z(x)+z-y|^2}{4t}}
		\le 4^D\Big(\frac{1}{\sqrt{2t}}+\sum_{z\in\Z^D}\frac{|z|}{2t}\mathrm{e}^{-\frac{|z|^2}{4t}}\Big).
	\end{equation}
	Since
	\[
		\sum_{z\in\Z^D}\frac{|z|}{2t}\mathrm{e}^{-\frac{|z|^2}{4t}}
		\asymp\int_{\R^D}\frac{|x|}{2t}\mathrm{e}^{-\frac{|x|^2}{4t}}\,\mathrm{d}x
		=\frac{(2\pi t)^{\frac{D}{2}}}{2t}\mathbb{E}\big[|B_{2t}|\big]
		\le\frac{(2\pi t)^{\frac{D}{2}}}{\sqrt{2t}},
	\]
	it follows from \eqref{eq:log_grad_trans1} that
	\[
		\sum_{z\in\Z^D}\frac{|R_z(x)+z-y|}{2t}\mathrm{e}^{-\frac{|R_z(x)+z-y|^2}{4t}}
		\lesssim\frac{1}{\sqrt{t}},
	\]
	and hence $|\nabla_x\log q_t(y, x)|\lesssim\frac{|x-y|}{t}+\frac{1}{\sqrt{t}}$ as claimed.
	By the score matching identity 
	\[\nabla\log p_t(x)= \int \nabla_x \log q_t(y,x) \frac{q_t(y,x) \,\mu(\diff{y})}{p_t(x)} = \mathbb{E}[\nabla_2\log q_t(X_0, X_t)\mid X_t=x],\] 
	it follows that
	\[
		|\nabla\log p_t(x)|
		\le \mathbb{E}[|\nabla_2\log q_t(X_0, X_t)|\mid X_t=x]
		\lesssim \frac{1}{t}\mathbb{E}[|X_0-X_t|\mid X_t=x]+\frac{1}{\sqrt{t}},
	\]
    proving part \hyperref[lem:affine_approx_e]{(e)}.\ref{lem:loggrad_score_e2}. Let now $t< 1$. We have
	\begin{align*}
		\mathbb{E}[|X_0-X_t|\mid X_t=x]
		&=\frac{\int_{M}|x-y|q_t(y, x)\,\mu(\diff{y})}{p_t(x)} \\
		&=\frac{\int_{M\cap \mathcal B(x, \rho_t)}|x-y|q_t(y, x)\,\mu(\diff{y})}{p_t(x)}+\frac{\int_{M\cap \mathcal B(x, \rho_t)^{\mathsf{c}}}|x-y|q_t(y, x)\,\mu(\diff{y})}{p_t(x)},
	\end{align*}
	where $\rho_t=\sqrt{3t(\rho+\frac{c_0+1}{2}\log t^{-1})}$.
	Clearly, for the first term we have
	\[
		\frac{\int_{M\cap \mathcal B(x, \rho_t)}|x-y|q_t(y, x)\,\mu(\diff{y})}{p_t(x)}
		\le\frac{\rho_t\int_{M}q_t(y, x)\,\mu(\diff{y})}{p_t(x)}
		=\rho_t
		\lesssim \sqrt{t(\rho+\log t^{-1})},
	\]
	while for the second term, we have by \hyperref[lem:affine_approx_d]{(d)} that $p_t(x)\gtrsim t^{\frac{c_0-D}{2}}\mathrm{e}^{-\rho}$ for $x\in M_{\rho, t}$, and since the diameter of $[0, 1]^D$ is $\sqrt{D}$, we get
	\[
		\frac{\int_{M\cap \mathcal B(x, \rho_t)^{\mathsf{c}}}|x-y|q_t(y, x)\,\mu(\diff{y})}{p_t(x)}
		\le\sqrt{D}\frac{\int_{M\cap \mathcal B(x, \rho_t)^{\mathsf{c}}}q_t(y, x)\,\mu(\diff{y})}{p_t(x)}
		\lesssim t^{\frac{D-c_0}{2}}\mathrm{e}^{\rho}\int_{M\cap \mathcal B(x, \rho_t)^{\mathsf{c}}}q_t(y, x)\,\mu(\diff{y}).
	\]	
	Now, by \cite[Theorem 3.2]{liyau} we have $q_t(y, x)\lesssim t^{-\frac{D}{2}}\mathrm{e}^{-\frac{|x-y|^2}{3t}}$, whence
	\[
		t^{\frac{D-c_0}{2}}\mathrm{e}^{\rho}\int_{M\cap \mathcal B(x, \rho_t)^{\mathsf{c}}}q_t(y, x)\,\mu(\diff{y})
		\lesssim t^{-\frac{c_0}{2}}\mathrm{e}^{\rho-\frac{\rho_t^2}{3t}}
		=\sqrt{t}.
	\]
	Thus, using part \hyperref[lem:affine_approx_e]{(e)}.\ref{lem:loggrad_score_e2}, we find
	\[
		|\nabla\log p_t(x)| \lesssim \frac{1}{t}\mathbb{E}[|X_0-X_t|\mid X_t=x]+\frac{1}{\sqrt{t}}
		\lesssim\frac{\sqrt{t(\rho+\log t^{-1})}+\sqrt{t}}{t}+\frac{1}{\sqrt{t}}
		\lesssim\frac{\sqrt{\rho+\log t^{-1}}}{\sqrt{t}},
	\]
	showing \hyperref[lem:affine_approx_e]{(e)}.\ref{lem:loggrad_score_e3}.
\end{proof}

\section{Remaining proofs for Section \ref{sec:wasserstein}}\label{app:wasserstein}
\begin{proof}[Proof of Proposition \ref{lem:wasserstein_Yi}]
Combining \eqref{eq:tv} with Pinsker's inequality and Girsanov's theorem for reflected diffusions, cf.\ \cite[Theorem~A.1]{holk24}, we obtain
	\begin{align*}
		\mathcal W_1\bigl(Y_{\overline T-\underline T}^{(i-1)},Y_{\overline T-\underline T}^{(i)}\bigr)
		&\le 2\sqrt{D}\,\mathrm{TV}\bigl(Y_{\overline T-\underline T}^{(i-1)},Y_{\overline T-\underline T}^{(i)}\bigr) \\
		&\le \sqrt{2D\,\KL{Y_{\overline{T}-\underline{T}}^{(i-1)}}{Y_{\overline{T}-\underline{T}}^{(i)}}} = \sqrt{D\int_{t_{i-1}}^{t_i}
			\mathbb E\left[\bigl|s(X_t,t)-\nabla\log p_t(X_t)\bigr|^2\right]\diff t }.
	\end{align*}
	This bound already yields the claim whenever $t_i\gtrsim \rho^{-1}$. Hence, in the remainder of the proof we may and do assume that $t_i \le 1/(C_1\rho)$, for a constant $C_1>0$ to be chosen later.
	Let $\mathbb Q^{(i)}$ denote the law of the full path $Y^{(i)}$ on $C([0,\overline T-\underline T],[0,1]^D)$. By the Kantorovich--Rubinstein duality,
	\[
		\mathcal W_1\bigl(Y_{\overline T-\underline T}^{(i-1)},Y_{\overline T-\underline T}^{(i)}\bigr)
		= \sup_{\|f\|_{\mathrm{Lip}}\le1}
		\left|\int f\bigl(p(\overline T-\underline T)\bigr)\,(\mathbb Q^{(i-1)}-\mathbb Q^{(i)})(\diff p)\right|.
	\]
	Since the processes $Y^{(i-1)}$ and $Y^{(i)}$ coincide on $[0,\overline T-t_i)$, their marginals at time $\overline T-t_i$ agree, and thus
	\[
	\sup_{\|f\|_{\mathrm{Lip}}\le1}
	\left|\int f\bigl(p(\overline T-t_i)\bigr)\,(\mathbb Q^{(i-1)}-\mathbb Q^{(i)})(\diff p)\right|=\mathcal W_1\bigl(Y_{\overline T-t_i}^{(i-1)},Y_{\overline T-t_i}^{(i)}\bigr)
	=0.
	\]
	Subtracting this null term yields
	\[
		\mathcal W_1\bigl(Y_{\overline T-\underline T}^{(i-1)},Y_{\overline T-\underline T}^{(i)}\bigr)
		= \sup_{\|f\|_{\mathrm{Lip}}\le1}
		\left|\int \bigl[f(p(\overline T-\underline T))-f(p(\overline T-t_i))\bigr]
		(\mathbb Q^{(i-1)}-\mathbb Q^{(i)})(\diff p)\right|.
	\]	
Fix $C_2>0$, and introduce the event $A_i=\bigl\{|p(\overline T-\underline T)-p(\overline T-t_i)|\le C_2\sqrt{t_i\rho}\bigr\}\subset C([\overline T-\underline T],[0,1]^D)$.
Splitting the integral over $A_i$ and $A_i^\mathsf c$, and denoting by $|\nu|$ the total variation of a signed measure $\nu$, we obtain
\begin{align*}
&\sup_{\|f\|_{\mathrm{Lip}}\le1}
\left|\int \bigl[f(p(\overline T-\underline T))-f(p(\overline T-t_i))\bigr]
(\mathbb Q^{(i-1)}-\mathbb Q^{(i)})(\diff p)\right| \\
&\qquad\le C_2\sqrt{t_i\rho}\int |\mathbb Q^{(i-1)}-\mathbb Q^{(i)}|(\diff p)
+ \sqrt{D}\bigl(\mathbb Q^{(i-1)}(A_i^\mathsf c)+\mathbb Q^{(i)}(A_i^\mathsf c)\bigr).
\end{align*}
The first term is bounded using Pinsker's inequality as
\[
C_2\sqrt{t_i\rho}\,\mathrm{TV}\bigl(Y_{\overline T-\underline T}^{(i-1)},Y_{\overline T-\underline T}^{(i)}\bigr)\lesssim \sqrt{t_i\rho} \left(\int_{t_{i-1}}^{t_i}
\mathbb E\left[\bigl|s(X_t,t)-\nabla\log p_t(X_t)\bigr|^2\right]\diff t\right)^{1/2}.
\]
Hence, it remains to control the tail probabilities
\[
\mathbb P\bigl(|Y^{(i)}_{\overline T-\underline T}-Y^{(i)}_{\overline T-t_i}|>C_2\sqrt{t_i\rho}\bigr),
	\qquad
	\mathbb P\bigl(|Y^{(i-1)}_{\overline T-\underline T}-Y^{(i-1)}_{\overline T-t_i}|>C_2\sqrt{t_i\rho}\bigr).
\]
To this end, let $Z^{(i)}$ denote a solution to the SDE
\[
\diff Z^{(i)}_t = s\bigl(Z^{(i)}_t,\overline T-t_i-t\bigr)\diff t + \diff B_{\overline T-t_i+t},
\qquad Z^{(i)}_0 = Y^{(i)}_{\overline T-t_i},
\]
and define $\tau \coloneqq \inf\{t\ge 0 : Z^{(i)}_t \in \partial[0,1]^D\}$.
By construction, we have $Z^{(i)}_t = Y^{(i)}_{\overline T-t_i+t}$ for all $t\in[0,\tau\wedge (t_i-\underline T)]$.
Consequently,
\[
\mathbb P\bigl(|Y^{(i)}_{\overline T-\underline T}-Y^{(i)}_{\overline T-t_i}|>C_2\sqrt{t_i\rho}\bigr)
\le 1 - \mathbb P\bigl(|Z^{(i)}_{t_i-\underline T}-Z^{(i)}_0|\le C_2\sqrt{t_i\rho},\ \tau\ge t_i-\underline T\bigr).
\]
We now introduce the events
\[
A_1\coloneqq \bigl\{\forall t\in[t_i,1]: |X_t-X_0|\le C_3\sqrt{t\rho}\bigr\}, \quad
A_2 \coloneqq \Bigl\{\sup_{t\in[0,t_i-\underline T]} \bigl|B_{\overline T-t_i+t}-B_{\overline T-t_i}\bigr|
	\le C_4\sqrt{t_i\rho}\Bigr\},
\]
where the constants $C_3,C_4>0$ will be specified below.
On the event $A_1\cap A_2$, we use that $Z^{(i)}_0 = Y^{(i)}_{\overline T-t_i} = X_{t_i}$ and obtain
$\mathrm{dist}(Z^{(i)}_0,M) \le |X_{t_i}-X_0| \le C_3\sqrt{t_i\rho}$.
Moreover, using the assumed bound $|s(x,t)|\le C\sqrt{\rho/t}$, we have for all $t\in[0,t_i-\underline T]$,
\begin{align*}
	|Z^{(i)}_t - Z^{(i)}_0|
	&\le \int_0^t |s(Z^{(i)}_u,\overline T-t_i-u)|\diff u + \bigl|B_{\overline T-\underline T}-B_{\overline T-t_i}\bigr| \\
	&\le C\sqrt{\rho}\int_0^t \frac{1}{\sqrt{\overline T-t_i-u}}\diff u	+ C_4\sqrt{t_i\rho} \\
	&\le (2C+C_4)\sqrt{t_i\rho}
\end{align*}
on $A_1\cap A_2$.
Combining the previous two bounds yields
\[
\mathrm{dist}(Z^{(i)}_t,M)\le \mathrm{dist}(Z^{(i)}_0,M) + |Z^{(i)}_t-Z^{(i)}_0|
\le (C_3+2C+C_4)\sqrt{t_i\rho}
\]
on $A_1 \cap A_2$. Since $t_i\le (C_1\rho)^{-1}$ by assumption, we further obtain that on $A_1 \cap A_2$,
\[
\mathrm{dist}(Z^{(i)}_t,\partial[0,1]^D)
\ge \rho_{\min} - \mathrm{dist}(Z^{(i)}_t,M)
\ge \rho_{\min} - \frac{C_3+2C+C_4}{\sqrt{C_1}}.
\]
Choosing $C_1 \ge 1 \vee \Bigl(\frac{2(C_3+2C+C_4)}{\rho_{\min}}\Bigr)^2$ ensures that $\mathrm{dist}(Z^{(i)}_t,\partial[0,1]^D)>0$ for all $t\in[0,t_i-\underline T]$, and hence $\tau\ge t_i-\underline T$ on $A_1\cap A_2$. On this event, we therefore have
\[
|Z^{(i)}_{t_i-\underline T}-Z^{(i)}_0|\le C_{2,1}\sqrt{t_i\rho}, \quad C_{2,1}\coloneqq 2C+C_4.
\]
It follows that
\[
\mathbb P\bigl(|Y^{(i)}_{\overline T-\underline T}-Y^{(i)}_{\overline T-t_i}|>C_{2,1}\sqrt{t_i\rho}\bigr)
\le \mathbb P(A_1^\mathsf c)+\mathbb P(A_2^\mathsf c).
\]
We now bound the two probabilities on the right-hand side.  
Choosing $C_3 = \widetilde C\sqrt D\, \tfrac{\log\bigl(1+\log(t_1^{-1})\bigr)+\sqrt\rho}{\sqrt\rho}\lesssim 1$, where $\widetilde C$ is the universal constant from Lemma~\ref{lem:}, that lemma yields $\mathbb P(A_1^\mathsf c)\lesssim \mathrm e^{-\rho}$.
	Similarly, choosing $C_4=\frac{\sqrt{D+4\rho}}{\sqrt{\rho}}\lesssim 1$, it follows by Lemma \ref{lem:brownian_bound_a} that
	\begin{align*}
		\mathbb{P}(A_2^{\mathsf{c}})
		&=\mathbb{P}\Big(\sup_{t\in[0, t_i-\underline{T}]}|B_{\overline{T}-t_i+t}-B_{\overline{T}-t_i}|>\sqrt{t_i(D+4\rho)}\Big) \\
		&=\mathbb{P}\Big(\sup_{t\in[0, t_i-\underline{T}]}|B_t|>\sqrt{t_i(D+4\rho)}\Big) \\
		&\le 2 \mathbb{P}\big(|B_{t_i-\underline{T}}|>\sqrt{t_i(D+4\rho)}\big) \\
		&\lesssim \mathrm{e}^{-\rho}.
	\end{align*}
    For the first inequality we used the following classical argument for Brownian motion: let $t, a>0$ and  $\tau_a \coloneq \inf\{s\ge0:|B_s|=a\}$.
	Then
	\[
		\mathbb{P}\Big(\sup_{s\le t}|B_s|>a\Big)
		=\mathbb{P}(\tau_a\le t)
		=\mathbb{P}(\tau_a\le t, |B_t|>a)+\mathbb{P}(\tau_a\le t, |B_t|\le a) \leq \mathbb{P}(|B_t|>a)+\mathbb{P}(\tau_a\le t, |B_t|\le a),
	\]
	and by the tower rule and strong Markov property,
	\begin{equation}\label{eq:refl_princ}
    \begin{split}
		\mathbb{P}(\tau_a\le t, |B_t|\le a)
		=\mathbb{E}\big[\bm{1}_{\{\tau_a\le t\}}\mathbb{P}^{B_{\tau_a}}(|B_{t-\tau_a}|\le a)\big] 
		\le \frac{1}{2}\mathbb{P}(\tau_a\le t),
    \end{split}
	\end{equation}
	where we used that for any $s\ge0$ and $x\in\R^D$ with $|x|=a$ we have
	\[
		\mathbb{P}^x(|B_s|\le a) = \PP^0(\lvert B_s -x \rvert \leq a) 
		\le \mathbb{P}( B_s\cdot x \geq 0 )
		=\frac{1}{2},
	\]
	because $B_s \cdot x$ is a centered Gaussian random variable. Inserting this into \eqref{eq:refl_princ} and rearranging yields 
    \[\mathbb{P}\Big(\sup_{s\le t}|B_s|>a\Big) \leq 2 \PP(\lvert B_t \rvert \geq a).\]
	Next, since $|Y^{(i-1)}_{\overline{T}-\underline{T}}-Y^{(i)}_{\overline{T}-t_i}|\le |Y^{(i-1)}_{\overline{T}-\underline{T}}-Y^{(i-1)}_{\overline{T}-t_{i-1}}|+|Y^{(i-1)}_{\overline{T}-t_{i-1}}-Y^{(i-1)}_{\overline{T}-t_i}|$, it follows that
	\[
		\mathbb{P}(|Y^{(i-1)}_{\overline{T}-\underline{T}}-Y^{(i-1)}_{\overline{T}-t_i}|>C_2\sqrt{t_i\rho})
		\le\mathbb{P}(|Y^{(i-1)}_{\overline{T}-\underline{T}}-Y^{(i-1)}_{\overline{T}-t_{i-1}}|>C_2\sqrt{t_{i}\rho})+\mathbb{P}(|Y^{(i-1)}_{\overline{T}-t_{i-1}}-Y^{(i-1)}_{\overline{T}-t_i}|>C_2\sqrt{t_i\rho}).
	\]
	Here, since $t_{i-1}\le t_i$, we have
	\[
		\mathbb{P}(|Y^{(i-1)}_{\overline{T}-\underline{T}}-Y^{(i-1)}_{\overline{T}-t_{i-1}}|>C_2\sqrt{t_{i}\rho})
		\le\mathbb{P}(|Y^{(i-1)}_{\overline{T}-\underline{T}}-Y^{(i-1)}_{\overline{T}-t_{i-1}}|>C_2\sqrt{t_{i-1}\rho})
		\lesssim \mathrm{e}^{-\rho}
	\]
	by the exact same argument as above.
	Meanwhile, since $Y^{(i-1)}_{\overline{T}-t}=X_t$ for $t\in[t_{i-1}, t_i]$, we have
	\[
		|Y^{(i-1)}_{\overline{T}-t_{i-1}}-Y^{(i-1)}_{\overline{T}-t_i}|
		=|X_{t_{i-1}}-X_{t_i}|
		\le|X_{t_{i-1}}-X_0|+|X_{t_i}-X_0|,
	\]
	and so once again by Lemma \ref{lem:}
	\begin{align*}
		\mathbb{P}(|Y^{(i-1)}_{\overline{T}-t_{i-1}}-Y^{(i-1)}_{\overline{T}-t_i}|\le C_2\sqrt{t_i\rho})
		&\ge \mathbb{P}\Big(|X_{t_{i-1}}-X_0|\le \frac{C_2\sqrt{t_i\rho}}{2}, |X_{t_i}-X_0|\le \frac{C_2\sqrt{t_i\rho}}{2}\Big) \\
		&\ge \mathbb{P}\Big(\forall t\in[t_{i-1}, 1]:|X_t-X_0|\le\frac{C_2\sqrt{t\rho}}{2}\Big) \\
		&\gtrsim 1-\mathrm{e}^{-\rho},
	\end{align*}
	where $C_2=(2C+C_4)\vee(2\widetilde{C}\sqrt{D}\frac{\log(1+\log t_{i-1}^{-1})+\sqrt{\rho}}{\sqrt{\rho}}) \geq C_{2,1}.$ In summary, we conclude that 
    \begin{align*} 
    &\mathbb{P}(|Y^{(i-1)}_{\overline{T}-\underline{T}}-Y^{(i-1)}_{\overline{T}-t_i}|>C_2\sqrt{t_i\rho})+\mathbb{P}(|Y^{(i)}_{\overline{T}-\underline{T}}-Y^{(i)}_{\overline{T}-t_i}|>C_2\sqrt{t_i\rho})\\
    &\,\leq \mathbb{P}(|Y^{(i-1)}_{\overline{T}-\underline{T}}-Y^{(i-1)}_{\overline{T}-t_i}|>C_2\sqrt{t_i\rho})+\mathbb{P}(|Y^{(i)}_{\overline{T}-\underline{T}}-Y^{(i)}_{\overline{T}-t_i}|>C_{2,1}\sqrt{t_i\rho}) \lesssim \mathrm{e}^{-\rho},
    \end{align*}
    which finishes the proof.
\end{proof}

\subsection{Proof of the score approximation accuracy}\label{app:approx}
We follow the approximation programme outlined in Section \ref{subsec:approx}.

\paragraph{Step 1: score truncation error}
\begin{lemma}
	\label{lem:truncation_error}
	Fix $\underline{t}>0$, let $K\in\N_0$ be given and let $s_0^K$ be given by \eqref{eq:score_truncation}.
	Then we have for $t\in[\underline{t}, 2\underline{t}]$
	\[
		\mathbb{E}[|s_0(X_t, t)-s_0^K(X_t, t)|^2]
		\lesssim \frac{1}{t^2\wedge1}K^{\frac{D}{2}}\mathrm{e}^{-K}.
	\]
\end{lemma}

\begin{proof}
	For notation, set $K_{\underline{t}}=\sqrt{2\underline{t}(D+2K)}$.
	We first have by the triangle inequality that
	\begin{align}
		\mathbb{E}[|s_0(X_t, t)-s_0^K(X_t, t)|^2]
		&=\mathbb{E}\Big[\Big|\frac{1}{p_t(X_t)}\Big(s_0^K(X_t, t)\big(p_t(X_t)-p_t^K(X_t)\big)-\nabla\big(p_t(X_t)-p_t^K(X_t)\big)\Big)\Big|^2\Big]\nonumber \\
		&\lesssim \mathbb{E}\Big[|s_0^K(X_t, t)|^2\frac{\big(p_t(X_t)-p_t^K(X_t)\big)^2}{p_t(X_t)^2}\Big]+\mathbb{E}\Big[\frac{|\nabla(p_t(X_t)-p_t^K(X_t))|^2}{p_t(X_t)^2}\Big].\label{eq:trunc_error_decomp}
	\end{align}
	We first concentrate on the first term.
	From (the proof of) Lemma \ref{lem:affine_approx_a}, it follows that also $|s_0^K(x, t)|\lesssim\frac{1}{t\wedge1}$, whence
	\begin{align*}
		\mathbb{E}\Big[|s_0^K(X_t, t)|^2\frac{\big(p_t(X_t)-p_t^K(X_t)\big)^2}{p_t(X_t)^2}\Big]
		&\lesssim\frac{1}{t^2\wedge1}\mathbb{E}\Big[\frac{\big(p_t(X_t)-p_t^K(X_t)\big)^2}{p_t(X_t)^2}\Big] \\
		&=\frac{1}{t^2\wedge1}\int_{[0, 1]^D}\frac{\big(p_t(x)-p_t^K(x)\big)^2}{p_t(x)}\,\mathrm{d}x \\
		&\le\frac{1}{t^2\wedge1}\int_{[0, 1]^D}p_t(x)-p_t^K(x)\,\mathrm{d}x,
	\end{align*}
	where in the last inequality we use that the terms of $p_t(x)$ are non-negative, whence $p_t(x)-p_t^K(x)\le p_t(x)$.
	Next, notice that for all $z\in\Z^D$, we have $R_z([0, 1]^D)=[0, 1]^D$, whence $\{R_z([0, 1]^D)+z\mid z\in\Z^D\}$ is a disjoint (apart from a measure $0$ set) partition of $\R^D$, and similarly $\{R_z([0, 1]^D)+z\mid z\in\Z^D, \n{z}_{\infty}>K_{\underline{t}}\}$ is a partition of the set $\R^D\setminus[-\lfloor K_{\underline{t}}\rfloor, \lfloor K_{\underline{t}}\rfloor+1]^D$, whence for any integrable $g$,
	\[
		\sum_{\substack{z\in\Z^D \\ \n{z}_{\infty}>K_{\underline{t}}}}\int_{[0, 1]^D}g(R_z(x)+z)\,\mathrm{d}x
		=\int_{\R^D\setminus[-\lfloor K_{\underline{t}}\rfloor, \lfloor K_{\underline{t}}\rfloor+1]^D}g(x)\,\mathrm{d}x.
	\]
	In particular, by Fubini--Tonelli's theorem
	\begin{align*}
		\int_{[0, 1]^D}p_t(x)-p_t^K(x)\,\mathrm{d}x
		&=(2\pi t)^{-\frac{D}{2}}\sum_{\substack{z\in\Z^D \\ \n{z}_{\infty}>K_{\underline{t}}}}\int_{[0, 1]^D}\int_{[0, 1]^D}\mathrm{e}^{-\frac{|R_z(x)+z-y|^2}{2t}}\,\mu(\diff{y})\,\mathrm{d}x \\
		&=\int_{[0, 1]^D}\Big(\int_{\R^D\setminus[-\lfloor K_{\underline{t}}\rfloor, \lfloor K_{\underline{t}}\rfloor+1]^D}(2\pi t)^{-\frac{D}{2}}\mathrm{e}^{-\frac{|x-y|^2}{2t}}\,\mathrm{d}x\Big)\,\mu(\diff{y}).
	\end{align*}
	Now, for each $y\in M$ and $x\in\R^D\setminus[-\lfloor K_{\underline{t}}\rfloor, \lfloor K_{\underline{t}}\rfloor+1]^D$, we necessarily have that $|x-y|\ge\n{x-y}_{\infty}\ge \lfloor K_{\underline{t}}\rfloor$, whence Lemma \ref{lem:brownian_bound_a} yields
	\begin{align*}
		\int_{\R^D\setminus[-\lfloor K_{\underline{t}}\rfloor, \lfloor K_{\underline{t}}\rfloor+1]^D}(2\pi t)^{-\frac{D}{2}}\mathrm{e}^{-\frac{|x-y|^2}{2t}}\,\mathrm{d}x
		&\le\int_{\{|x-y|\ge \lfloor{K_{\underline{t}}\rfloor}\}}(2\pi t)^{-\frac{D}{2}}\mathrm{e}^{-\frac{|x-y|^2}{2t}}\,\mathrm{d}x \\
		&=\mathbb{P}(|B_t|\ge \lfloor{K_{\underline{t}}\rfloor}) \\
		&=\mathbb{P}(|B_t|\in[\lfloor K_{\underline{t}}\rfloor, K_{\underline{t}}))+\mathbb{P}(|B_t|\ge K_{\underline{t}}) \\
		&\lesssim \mathbb{P}(|B_t|\in[\lfloor K_{\underline{t}}\rfloor, K_{\underline{t}}))+K^{\frac{D}{2}}\mathrm{e}^{-K},
	\end{align*}
	since $t\le 2\underline{t}$.
	Meanwhile, we have
	\begin{align*}
		\mathbb{P}(|B_t|\in[\lfloor K_{\underline{t}}\rfloor, K_{\underline{t}}))
		&\propto t^{-\frac{D}{2}}\int_{\lfloor K_{\underline{t}}\rfloor}^{K_{\underline{t}}}r^{D-1}\mathrm{e}^{-\frac{r^2}{2t}}\,\mathrm{d}r \\
		&\lesssim t^{-\frac{D}{2}}K_{\underline{t}}^D \mathrm{e}^{-\frac{K_{\underline{t}}^2}{2t}} \\
		&\lesssim K^{\frac{D}{2}}\mathrm{e}^{-K},
	\end{align*}
	and hence
	\[
		\int_{\R^D\setminus[-\lfloor K_{\underline{t}}\rfloor, \lfloor K_{\underline{t}}\rfloor+1]^D}(2\pi t)^{-\frac{D}{2}}\mathrm{e}^{-\frac{|x-y|^2}{2t}}\,\mathrm{d}x
		\lesssim K^{\frac{D}{2}}\mathrm{e}^{-K}.
	\]

	By the above, this implies
	\[
		\mathbb{E}\Big[|s_0^K(X_t, t)|^2\frac{\big(p_t(X_t)-p_t^K(X_t)\big)^2}{p_t(X_t)^2}\Big]
		\lesssim \frac{1}{t^2\wedge 1}K^{\frac{D}{2}}\mathrm{e}^{-K}.
	\]
	As for the second term of \eqref{eq:trunc_error_decomp}, we again note that by following the proof of Lemma \ref{lem:affine_approx_a}, we have $\frac{|\nabla(p_t(x)-p_t^K(x))|}{p_t(x)}\lesssim\frac{1}{t\wedge1}$, whence we have similar to before
	\[
		\mathbb{E}\Big[\frac{|\nabla(p_t(X_t)-p_t^K(X_t))|^2}{p_t(X_t)^2}\Big]
		\lesssim\frac{1}{t\wedge1}\int_{[0, 1]^D}\big|\nabla\big(p_t(x)-p_t^K(x)\big)\big|\,\mathrm{d}x.
	\]
	Furthermore, we have by the triangle inequality and similar calculations as before
	\begin{align*}
		\int_{[0, 1]^D}\big|\nabla\big(p_t(x)-p_t^K(x)\big)\big|\,\mathrm{d}x
		&\le(2\pi t)^{-\frac{D}{2}}\sum_{\substack{z\in\Z^D \\ \n{z}_{\infty}>K_{\underline{t}}}}\int_{[0, 1]^D}\int_{[0, 1]^D}\frac{|R_z(x)+z-y|}{t}\mathrm{e}^{-\frac{|R_z(x)+z-y|^2}{2t}}\,\mu(\diff{y})\,\mathrm{d}x \\
		&=\int_{[0, 1]^D}\Big(\int_{\R^D\setminus[-K_{\underline{t}}, K_{\underline{t}}+1]^D}(2\pi t)^{-\frac{D}{2}}\frac{|x-y|}{t}\mathrm{e}^{-\frac{|x-y|^2}{2t}}\,\mathrm{d}x\Big)\,\mu(\diff{y}),
	\end{align*}
	where we have by Lemma \ref{lem:brownian_bound_b}
	\begin{align*}
		\int_{\R^D\setminus[-K_{\underline{t}}, K_{\underline{t}}+1]^D}(2\pi t)^{-\frac{D}{2}}\frac{|x-y|}{t}\mathrm{e}^{-\frac{|x-y|^2}{2t}}\,\mathrm{d}x
		&\le\frac{1}{t}\mathbb{E}[|B_t|\bm{1}_{\{|B_t|_{\infty}\ge K_{\underline{t}}\}}] \\
		&\lesssim \frac{1}{\sqrt{t}}K^{\frac{D}{2}}\mathrm{e}^{-K}.
	\end{align*}
	Inserting this into the above, we have
	\[
		\mathbb{E}\Big[\frac{|\nabla(p_t(X_t)-p_t^K(X_t))|^2}{p_t(X_t)^2}\Big]
		\lesssim \frac{1}{t^2\wedge1}K^{\frac{D}{2}}\mathrm{e}^{-K},
	\]
	which combined with the above and \eqref{eq:trunc_error_decomp} yields the result.
\end{proof}

\paragraph{Step 2: approximation of $f_1(\cdot,t), f_2(\cdot,t)$ for fixed $t$}
We start with the small time regime.
\begin{lemma}
	\label{lem:fixed_time_approx_small_t}
	Under assumptions \ref{ass:H1}--\ref{ass:H3}, for large enough $m\in\N$ and fixed $t\le m^{-\frac{2-\delta}{d}}$ there exist neural networks $\varphi_{1, t}\in\widetilde{\Phi}(\log m, m, m\log m, m^\nu)$ and $\varphi_{2, t}\in\widetilde{\Phi}(\log m, m, m\log m, t^{-\frac{1}{2}}\vee m^\nu)$, where $\nu=\frac{2d}{2\alpha-d}+\frac{1}{d}$ such that for $u\in\R^d$,
	\begin{align*}
		|(2\pi t)^{-\frac{d}{2}}f_1(u, t)-\varphi_{1, t}(u)|
		&\lesssim \begin{cases}
			m^{-\frac{\alpha}{d}}, &\text{ if }u\in M^*_{-\varepsilon_M/2}\\
			\mathrlap{(\log m)^{\frac{d}{2}}m^{-\frac{\kappa}{d}},}\phantom{\frac{1}{\sqrt{t}}(\log m)^{\frac{d+1}{2}}m^{-\frac{\kappa}{d}},} &\text{ if }u\notin M^*_{-\varepsilon_M/2}
		\end{cases}
		\intertext{and}
		|(2\pi t)^{-\frac{d}{2}}f_2(u, t)-\varphi_{2, t}(u)|
		&\lesssim \begin{cases}
			\frac{1}{\sqrt{t}}m^{-\frac{\alpha}{d}}, &\text{ if }u\in M^*_{-\varepsilon_M/2} \\
			\frac{1}{\sqrt{t}}(\log m)^{\frac{d+1}{2}}m^{-\frac{\kappa}{d}}, &\text{ if }u\notin M^*_{-\varepsilon_M/2}
		\end{cases}
	\end{align*}
	where $f_1, f_2$ are as in \eqref{eq:f_1f_2}.
\end{lemma}

\begin{proof}
	To construct the neural networks $\varphi_{i, t}$, we first construct separate networks $\varphi_{i, t}^{(1)}$ and $\varphi_{i, t}^{(2)}$ corresponding to the cases where either $u\in M^*_{-\varepsilon_M/2}$ or $u\notin M^*_{-\varepsilon_M/2}$, in the latter case utilizing the increased smoothness near the boundary $\partial M$ of $M$ per assumption \ref{ass:H3}, which we then stitch together into one network.
	To this end, we first establish the following common notation:
	let $q(u)=p_0(Au+v_0)$ and $n_t(u)=(2\pi t)^{-\frac{d}{2}}\mathrm{e}^{-\frac{|u|^2}{2t}}$, such that $(2\pi t)^{-\frac{d}{2}}f_1(u)=n_t*q(u)$, while $(2\pi t)^{-\frac{d}{2}}f_2(u)=\nabla n_t*q(u)$.
	Suppose then that $\mathrm{dist}(u, M^*)>\varepsilon_M/2$.
	Since $t\le m^{-\frac{2-\delta}{d}}$, we have $t(d+2\frac{\kappa}{d}\log m)\to0$ for $m\to\infty$, so we may assume that $m$ is large enough that $\sqrt{t(d+2\frac{\kappa}{d}\log m)}\le\varepsilon_M/2$, and it follows by Lemma \ref{lem:brownian_bound} that
	\[
		|n_t*q(u)|
		\le p_{\mathrm{max}}\mathbb{P}\bigg(|B_t^*|>\sqrt{t\Big(d+2\frac{\kappa}{d}\log m\Big)}\bigg)
		\lesssim(\log m)^{\frac{d}{2}}m^{-\frac{\kappa}{d}}
	\]
	and
	\[
		|\nabla n_t*q(u)|
		\le \frac{p_{\mathrm{max}}}{t}\mathbb{E}\big[|B_t^*|\bm{1}_{\{|B_t^*|>\sqrt{t(d+2\frac{\kappa}{d}\log m)}\}}\big]
		\lesssim \frac{1}{\sqrt{t}}(\log m)^{\frac{d+1}{2}}m^{-\frac{\kappa}{d}},
	\]
	where $(B_t^*)_{t\ge0}$ is a $d$-dimensional Brownian motion.
	Thus for $u\notin M^*_{-\varepsilon_M/2}$, it suffices to approximate $n_t*q$ and $\nabla n_t*q$ on the set $(\partial M^*)_{3\varepsilon_M/4}\supset(\partial M^*)_{\varepsilon_M/2}$.
	Now since both $M^*_{-\varepsilon_M/2}$ and $(\partial M^*)_{3\varepsilon_M/4}$ are compact and have Lipschitz boundary, it follows by Lemma \ref{lem:sobolev_network} that there exists neural networks $\widetilde{\varphi}_{1, t}^{(i)}\in\widetilde{\Phi}(\log m, m, m\log m, C_{i, 1}\vee m^\nu)$ and $\widetilde{\varphi}_{2, t}^{(i), j}\in\widetilde{\Phi}(\log m, m, m\log m, C_{i, 2}\vee m^\nu)$ for $i=1, 2$ and $j\in[d]$ such that
	\begin{align*}
		|\varphi^{(i)}_{1, t}(u)-n_t*q(u)|
		&\lesssim \begin{cases}
			\n{n_t*q}_{H^{\alpha}(M^*_{-\varepsilon_M/2})}m^{-\frac{\alpha}{d}}, &\text{ if }u\in M^*_{-\varepsilon_M/2}, i=1\\
			\n{n_t*q}_{H^{\kappa}( (\partial M^*)_{3\varepsilon_M/4})}m^{-\frac{\kappa}{d}}, &\text{ if }u\in (\partial M^*)_{3\varepsilon/4}, i=2
		\end{cases}
		\intertext{and}
		|\varphi^{(i), j}_{2, t}(u)-\partial_jn_t*q(u)|
		&\lesssim \begin{cases}
			\n{\partial_jn_t*q}_{H^{\alpha}(M^*_{-\varepsilon_M/2})}m^{-\frac{\alpha}{d}}, &\text{ if }u\in M^*_{-\varepsilon_M/2}, i=1\\
			\n{\partial_jn_t*q}_{H^{\kappa}( (\partial M^*)_{3\varepsilon_M/4})}m^{-\frac{\kappa}{d}}, &\text{ if }u\in (\partial M^*)_{3\varepsilon/4}, i=2,
		\end{cases}
		\intertext{and where}
		C_{i, j}
		&=\begin{cases}
			\n{n_t*q}_{H^{\alpha}(M^*_{-\varepsilon_M/2})}, &\text{ if }(i, j)=(1, 1)\\
			\n{n_t*q}_{H^{\kappa}( (\partial M^*)_{3\varepsilon_M/4})}, &\text{ if }(i, j)=(2, 1)\\
			\n{\partial_jn_t*q}_{H^{\alpha}(M^*_{-\varepsilon_M/2})}, &\text{ if }(i, j)=(1, 2)\\
			\n{\partial_jn_t*q}_{H^{\kappa}( (\partial M^*)_{3\varepsilon_M/4})}, &\text{ if }(i, j)=(2, 2).
		\end{cases}
	\end{align*}
	In order to bound $\n{n_t*q}_{H^{\alpha}(M^*_{-\varepsilon_M/2})}$, first note that by Assumption \ref{ass:H2} we may extend $q$ to a global $\alpha$-Sobolev function on $\R^d$ with compact support  by simply setting $q(u)=0$ for $u\notin M^*$, i.e. we can assume $q\in H^{\alpha}_c(\R^d)$.
	It then follows by Young's convolution inequality that
	\begin{equation}
        \label{eq:young}
	    \n{n_t*q}_{H^{\alpha}(\R^d)}^2
		=\sum_{|\beta|\le \alpha}\n{\partial^{\beta}(n_t*q)}_{L^2(\R^d)}^2
		=\sum_{|\beta|\le \alpha}\n{n_t*(\partial^{\beta}q)}_{L^2(\R^d)}^2
		\le\n{n_t}^{2}_{L^1(\R^d)}\sum_{|\beta|\le \alpha}\n{\partial^{\beta}q}_{L^2(\R^d)}^2
		=\n{q}_{H^{\alpha}(\R^d)}^2,
	\end{equation}
	and hence also $\n{n_t*q}_{H^{\alpha}(M^*_{-\varepsilon_M/2})}\lesssim 1$.
	Similarly,
	\begin{equation}
	    \label{eq:dn_t_l1}
        \n{\partial_jn_t*q}_{H^{\alpha}(M^*_{-\varepsilon_M/2})}
		\lesssim\n{\partial_j n_t}_{L^1(\R^d)}
		=\int_{\R^d}(2\pi t)^{-\frac{d}{2}}\frac{|u_j|}{t}\mathrm{e}^{-\frac{|u|^2}{2t}}\diff{u}
		=\sqrt{\frac{2}{\pi t}}\int_{0}^{\infty}\frac{r}{t}\mathrm{e}^{-\frac{r^2}{2t}}\diff{r}
		=\sqrt{\frac{2}{\pi t}},
	\end{equation}
	and setting $(\varphi^{(1)}_{2, t})_j=\varphi^{(1), j}_{2, t}$ yields the desired networks on $M^*_{-\varepsilon_M/2}$.
	Next, to bound $\n{n_t*q}_{H^{\kappa}( (\partial M^*)_{3\varepsilon_M/4})}$, fix $u\in(\partial M^*)_{3\varepsilon/4}$ and $\beta\in\N_0^d$ with $|\beta|\le\kappa$.
	Then we have
	\[
		(\partial^{\beta}n_t)*q(u)
		=\int_{\mathcal B(u, \varepsilon_M/4)}\partial^{\beta}n_t(u-v)q(v)\,\mathrm{d}v+\int_{\mathcal B(u, \varepsilon_M/4)^\mathsf{c}}\partial^{\beta}n_t(u-v)q(v)\,\mathrm{d}v
		\coloneqq{I_1(u)+I_2(u)},
	\]
	where we note that in the first integral, $u-v\in(\partial M^*)_{\varepsilon_M}$ for all $v\in \mathcal B(u, \varepsilon_M/4)$, whence $q$ here is in $C^{\kappa}$.
	Thus we have by integration by parts for $i\in[d]$
	\begin{align*}
		\int_{\mathcal B(u, \varepsilon_M/4)}\partial_{u_i}n_t(u-v)q(v)\,\mathrm{d}v
		&=-\int_{\mathcal B(u, \varepsilon_M/4)}\partial_{v_i}n_t(u-v)q(v)\,\mathrm{d}v \\
		&=\int_{\mathcal B(u, \varepsilon_M/4)}n_t(u-v)\partial_{v_i}q(v)\,\mathrm{d}v-\frac{4}{\varepsilon_M}\int_{\partial\mathcal B(u, \varepsilon_M/4)}n_t(u-v)q(v)(v_i-u_i)\,\mathcal H^{d-1}(\mathrm{d}v),
	\end{align*}
	where we use that $\frac{4(v_i-u_i)}{\varepsilon_M}$ is the outward pointing normal vector of $\mathcal B(u, \varepsilon_M/4)$ at $v\in\partial\mathcal B(u, \varepsilon_M/4)$.
	Repeating this and setting $\beta=\beta^{(1)}+\beta^{(2)}+\ldots+\beta^{(|\beta|)}$ with $|\beta^{(i)}|=1$, we have
	\[
		I_1(u)
		=\int_{\mathcal B(u, \varepsilon/4)}n_t(u-v)\partial^{\beta}q(v)\,\mathrm{d}v-\frac{4}{\varepsilon_M}\sum_{i=1}^{|\beta|}B_i(u),
	\]
	where
	\[
		B_i(u)
		=\int_{\partial\mathcal B(u, \varepsilon_M/4)}\partial^{\beta-\sum_{j=1}^{i}\beta^{(j)}}n_t(u-v)\partial^{\sum_{j=1}^{i-1}\beta^{(j)}}q(v)(v-u)^{\beta^{(i)}}\,\mathcal H^{d-1}(\mathrm{d}v).
	\]
	Clearly we have
	\[
		\bigg|\int_{\mathcal B(u, \varepsilon/4)}n_t(u-v)\partial^{\beta}q(v)\,\mathrm{d}v\bigg|
		\le \sup_{|\beta|\le\kappa}\sup_{u\in(\partial M^*)_{\varepsilon_M}}|\partial^{\beta}q(u)|
	\]
	independently of $\beta$, while
	\[
		|B_i(u)|
		\le \sup_{|\beta|\le\kappa}\n{\partial^{\beta}q}_{\infty}\int_{\partial\mathcal B(u, \varepsilon_M/4)}\big|\partial^{\beta-\sum_{j=1}^{i}\beta^{(j)}}n_t(u-v)(v-u)^{\beta^{(i)}}\big|\,\mathcal H^{d-1}(\mathrm{d}v).
	\]
	Here, the integrand is of the form $\mathrm{Poly}(t^{-1})\mathrm{Poly}( (u-v) )\mathrm{e}^{-\frac{|u-v|^2}{2t}}\lesssim t^{-(d+|\beta|)}\mathrm{e}^{-\frac{\varepsilon_M^2}{32t}}$, and is hence uniformly bounded by $\sup_{t>0}t^{-(d+\kappa)}\mathrm{e}^{-\frac{\varepsilon_M^2}{32t}}<\infty$, ultimately implying that $|I_1(u)|\lesssim1$.
	The same is true of the integrand in $I_2(u)$, implying that also $|I_2(u)|\lesssim1$.
	Putting things together, we have
	\begin{align*}
		\n{n_t*q}_{H^{\kappa}( (\partial M^*)_{3\varepsilon_M/4})}
		&=\sum_{|\beta|\le\kappa}\n{\partial^{\beta}(n_t*q)}_{L^2( (\partial M^*)_{3\varepsilon_M/4})} \\
		&=\sum_{|\beta|\le\kappa}\n{(\partial^{\beta}n_t)*q}_{L^2( (\partial M^*)_{3\varepsilon_M/4})} \\
		&\lesssim \sup_{|\beta|\le\kappa}\sup_{u\in(\partial M^*)_{3\varepsilon_M/4}}\sup_{t>0}|(\partial^{\beta}n_t)*q(u)| \\
		&\lesssim 1.
	\end{align*}
	A similar analysis can be used to bound $\n{\partial_jn_t*q}_{H^{\kappa}( (\partial M^*)_{3\varepsilon_M/4})}$, the only difference being that we can no longer move all derivatives from $n_t$ to $q$ using integration by parts as this would require $q$ to be locally $C^{\kappa+1}$.
	In particular, following the same notation as before, we find that
	\[
		|I_1(u)|
		\lesssim \bigg|\int_{\mathcal B(u, \varepsilon_M/4)}\partial_j n_t(u-v)\partial^{\beta}q(v)\,\mathrm{d}v\bigg|+1
		\lesssim \n{\partial_jn_t}_{L^1(\R^d)}
		\lesssim \frac{1}{\sqrt{t}},
	\]
	implying in the same way as before that $\n{\partial_jn_t*q}_{H^{\kappa}( (\partial M^*)_{3\varepsilon_M/4})}\lesssim\frac{1}{\sqrt{t}}$.
	Once again setting $(\varphi_{2, t}^{(2)})_j=\varphi_{2, t}^{(2), j}$ yields the desired networks on $(\partial M^*)_{\varepsilon_M/2}$.
	Now, on the overlap $M^*_{-\varepsilon_M/2}\cap (\partial M^*)_{3\varepsilon_M/4}$, both $\varphi^{(i)}_{1, t}$ approximate $n_t*q$ at the desired rate, and hence the same is true of any convex combination of the two.
	In particular, if we let $\varphi_{\bm{1}_{M^*_{-3\varepsilon_M/4}}}$ and $\varphi_{\bm{1}_{M^*_{\varepsilon_M/4}}}$ be as in Lemma \ref{lem:covering_network} with $\varepsilon=\varepsilon_M/4$, it follows that
	\[
		\widetilde{\varphi}_{1, t}
		\coloneq\varphi_{\bm{1}_{M^*_{\varepsilon_M/4}}}\Big(\varphi^{(1)}_{1, t}\varphi_{\bm{1}_{M^*_{-3\varepsilon_M/4}}}+\varphi_{1, t}^{(2)}\big(1-\varphi_{\bm{1}_{M^*_{-3\varepsilon_M/4}}}\big)\Big)
	\]
	has the desired error rate for all $u\in\R^d$.
	Setting
	\[
		\varphi_{1, t}
		\coloneq \varphi_{\ell}^{\mathrm{mult}}\Big(\varphi_{\bm{1}_{M^*_{\varepsilon_M/4}}}, \varphi_{\ell}^{\mathrm{mult}}\big(\varphi_{\bm{1}_{M^*_{-3\varepsilon_M/4}}}, \varphi^{(1)}_{1, t}\big)+\varphi_{\ell}^{\mathrm{mult}}\big(1-\varphi_{\bm{1}_{M^*_{-3\varepsilon_M/4}}}, \varphi^{(2)}_{1, t}\big)\Big)
	\]
	with $\ell=\lceil\frac{\kappa}{d}\log m\rceil$ yields the desired network, as the error and network size stemming from the multiplication networks are negligible compared to the rest.
	The exact same method can be used to construct $\varphi_{2, t}$, finishing the proof.
\end{proof}

Having approximated $f_1$ and $f_2$ for fixed small $t$, we now use the induced smoothness of the forward process to approximate these for fixed large $t$.

\begin{lemma}
	\label{lem:fixed_time_approx_large_t}
	Under assumptions \ref{ass:H1} and \ref{ass:H2}, for $\delta>0$, large enough $m\in\N$ and fixed $t>0$ with $\frac{1}{2}m^{-\frac{2-\delta}{d}}<t\lesssim\log m$, there exists neural networks $\varphi_{1, t}, \varphi_{2, t}\in\widetilde{\Phi}(\log m', m', m'\log m', m')$, where $m'=(t\wedge1)^{-\frac{d}{2}}m^{\frac{\delta}{2}}$ such that for $u\in\R^d$ and $i=1, 2$
	\begin{align*}
		|(2\pi t)^{-\frac{d}{2}}f_i(u, t)-\varphi_{i, t}(u)|
		&\lesssim
			m^{-\frac{\kappa+1}{d}}
	\end{align*}
	where $f_1$ and $f_2$ are as in \eqref{eq:f_1f_2}.
\end{lemma}

\begin{proof}
We start by bounding the Sobolev norm of Gaussian densities. 
As in the previous proof, let $n_t(u)=(2\pi t)^{-\frac{d}{2}}\mathrm{e}^{-\frac{|u|^2}{2t}}$ denote the density of $N(0, tI_d)$.
Also, for $s\in\R$, let $\psi(s)=\mathrm{e}^{-s^2}$ and $\eta_t(s)=\frac{s}{\sqrt{2t}}$ such that $n_t(u)=(2\pi t)^{-\frac{d}{2}}\prod_{i=1}^d\psi\circ\eta_t(u_i)$.
Thus, for $\beta\in\N_0^d$,
\begin{align*}
\partial^{\beta}n_t(u)
=(2\pi t)^{-\frac{d}{2}}\prod_{i=1}^d\Big(\frac{\mathrm{d}^{\beta_i}}{\mathrm{d}s^{\beta_i}}\psi\circ\eta_t\Big)(u_i)
		=(2\pi t)^{-\frac{d}{2}}\prod_{i=1}^d(2t)^{-\frac{\beta_i}{2}}\Big(\frac{\mathrm{d}^{\beta_i}}{\mathrm{d}s^{\beta_i}}\psi\Big)\big(\eta_t(u_i)\big).
\end{align*}
Also, for any function $g\in L^2(\R)$,
\[
		\n{g\circ\eta_t}_{L^2}^2
		=\int_{\R}g\big(\eta_t(s)\big)^2\diff{s}
		=\sqrt{2t}\int_{\R}g(r)^2\diff{r}
		=\sqrt{2t}\n{g}_{L^2}^2.
\]
Combining these, we have that
	\begin{align*}
		\n{\partial^{\beta}n_t}_{L^2}^2
		&=\int_{\R^d}(2\pi t)^{-d}\prod_{i=1}^d(2t)^{-\beta_i}\Big(\frac{\mathrm{d}^{\beta_i}}{\mathrm{d}s^{\beta_i}}\psi\Big)\big(\eta_t(u_i)\big)^2\diff{u} \\
		&=\pi^{-d}(2t)^{-(d+|\beta|)}\prod_{i=1}^d\n[\Big]{\Big(\frac{\mathrm{d}^{\beta_i}}{\mathrm{d}s^{\beta_i}}\psi\Big)\circ \eta_t}_{L^2}^2 \\
		&=\pi^{-d}(2t)^{-\frac{d+2|\beta|}{2}}\prod_{i=1}^d\n[\Big]{\frac{\mathrm{d}^{\beta_i}}{\mathrm{d}s^{\beta_i}}\psi}_{L^2}^2,
	\end{align*}
	implying that $\n{\partial^\beta n_t}_{L^2}^2=t^{-\frac{d+2|\beta|}{2}}\n{\partial^\beta n_1}_{L^2}^2$, and hence for $\gamma \in \N_0$
	\begin{equation}
		\label{eq:n_t_sobolev}
		\n{n_t}_{H^\gamma}
		=\sqrt{\sum_{|\beta|\le\gamma}\n{\partial^\beta n_t}_{L^2}^2}
		\le(t\wedge 1)^{-\frac{d+2\gamma}{4}}\n{n_1}_{H^\gamma}.
	\end{equation}
	Next, in order to actually approximate $n_t*q$ and $\nabla n_t*q$, we first restrict the set on which to approximate these in order to apply Lemma \ref{lem:sobolev_network}.
	To this end, let $c^*, r^*$ be as in the previous proof and set $\rho_{t, \gamma}=\sqrt{t(d+2\frac{\gamma}{d}\log m)}$ for $\gamma\ge0$, and note that for $u\in\R^d$ with $\n{u-c^*}_{\infty}>r^*+\rho_{t, \gamma}$, we have $\mathrm{dist}(u, M^*)>\rho_{t, \gamma}$ and hence by Lemma \ref{lem:brownian_bound} we have
	\[
		n_t*q(u)
		\lesssim\Big(\frac{\gamma}{d}\log m\Big)^{\frac{d}{2}}m^{-\frac{\gamma}{d}}\quad\text{and}\quad
		|\nabla n_t*q(u)|
		\lesssim \frac{1}{\sqrt{t}}\Big(\frac{\gamma}{d}\log m\Big)^{\frac{d+1}{2}}m^{-\frac{\gamma}{d}},
	\]
	whence we need only approximate $n_t*q$ and $\nabla n_t*q$ on $[-(r^*+\rho_{t, \gamma}), r^*+\rho_{t, \gamma}]^d+c^*$.
	As such, let $\widetilde{\varphi}_{1}, \widetilde{\varphi}_{2, j}\in\widetilde{\Phi}(\gamma^2\log m', \gamma^2 m', \gamma^4 m'\log m', (m')^\nu)$ where $\nu=\frac{2d}{2\gamma-d}+\frac{1}{d}$ be such that
	\begin{align*}
		|\widetilde{\varphi}_1(u)-n_t*q(u)|
		&\lesssim(1+\rho_{t, \gamma})^{\gamma-\frac{d}{2}}\n{n_t*q}_{H^\gamma}(m')^{-\frac{\gamma}{d}}
		\qquad\text{and} \\
		|\widetilde{\varphi}_{2, j}(u)-\partial_jn_t*q(u)|
		&\lesssim(1+\rho_{t, \gamma})^{\gamma-\frac{d}{2}}\n{\partial_jn_t*q}_{H^\gamma}(m')^{-\frac{\gamma}{d}}
	\end{align*}
	for all $u$ with $\n{u-c^*}_{\infty}\le r^*+\rho_{t, \gamma}+1$ in accordance with Lemma \ref{lem:sobolev_network}.
	Then, letting $\ell=\lceil\frac{\gamma}{d}\log_2 m\rceil\asymp\gamma\log m$ and $\varphi_{\rho_{t, \gamma}}=(1\wedge(r^*+\rho_{t, \gamma}+1-\n{u-c^*}_{\infty}))\vee0$, set
	\[
		\varphi_{1, t}(u)
		=\varphi^{\mathrm{mult}}_{\ell}(\varphi_{\rho_{t, \gamma}}(u), \widetilde{\varphi}_{1}(u))
		\quad\text{and}\quad
		\varphi_{2, t}(u)
		=\varphi^{\mathrm{mult}, d}_{\ell}(\varphi_{\rho_{t, \gamma}}(u), \widetilde{\varphi}_2(u)),
	\]
	where $(\widetilde{\varphi}_2)_j=\widetilde{\varphi}_{2, j}$.
	Once again, as the sizes of the multiplication networks and $\varphi_{\rho_{t, \gamma}}$ are negligible compared to those of $\widetilde{\varphi}_{i}$, it follows that also $\varphi_{1, t}, \varphi_{2, t}\in\widetilde{\Phi}(\gamma^2\log m', \gamma^2 m', \gamma^4 m' \log m', (m')^\nu)$, while
	\begin{align*}
		|\varphi_{1, t}(u)-n_t*q(u)|
		&\lesssim \Big(\Big(\frac{\gamma}{d}\log m\Big)^{\frac{d}{2}}\vee\big((1+\rho_{t, \gamma})^{\gamma-\frac{d}{2}}\n{n_t*q}_{H^\gamma}\big)\Big)(m')^{-\frac{\gamma}{d}}
		\quad\text{and} \\
		|\varphi_{2, t}(u)-\nabla n_t*q(u)|
		&\lesssim \Big(\frac{1}{\sqrt{t}}\Big(\frac{\gamma}{d}\log m\Big)^{\frac{d+1}{2}}\vee\big((1+\rho_{t, \gamma})^{\gamma-\frac{d}{2}}\n{n_t*q}_{H^{\gamma+1}}\big)\Big)(m')^{-\frac{\gamma}{d}},
	\end{align*}
	where we use that $\n{\partial_j n_t*q}_{H^\gamma}\le\n{n_t*q}_{H^{\gamma+1}}$.
	Since $t\lesssim\log m$, it follows that
	\[
		|\varphi_{1, t}(u)-n_t*q(u)|
		\lesssim \mathrm{Poly}(\log m)\n{n_t*q}_{H^\gamma}(m')^{-\frac{\gamma}{d}},
	\]
	where as in \eqref{eq:young} we have by Young's inequality and \eqref{eq:n_t_sobolev} that
    \[
        \n{n_t*q}_{H^\gamma}
        \le\n{q}_{L^1}\n{n_t}_{H^\gamma}
        \lesssim(t\wedge1)^{-\frac{d+2\gamma}{4}}.
    \]
    Inserting the definition of $m'$ and using the assumption that $t>\frac{1}{2}m^{-\frac{2-\delta}{d}}$ we see
	\begin{align*}
		\n{n_t*q}_{H^\gamma}(m')^{-\frac{\gamma}{d}}
		&\lesssim (t\wedge1)^{-\frac{d+2\gamma}{4}}\big((t\wedge1)^{-\frac{d}{2}}m^{\frac{\delta}{2}}\big)^{-\frac{\gamma}{d}} \\
		&=(t\wedge1)^{-\frac{d}{4}}m^{-\frac{\delta\gamma}{2d}} \\
		&\lesssim m^{\frac{2-\delta}{4}-\frac{\gamma\delta}{2d}}.
	\end{align*}
	Similarly, we have
	\[
		|\varphi_{2, t}(u)-\nabla n_t*q(u)|
		\lesssim \mathrm{Poly}(\log m)m^{\frac{2-\delta}{4}\frac{d+2}{d}-\frac{\gamma\delta}{2d}}.
	\]
	Setting $\gamma=\lceil\frac{2d}{\delta}(\frac{2-\delta}{4}\frac{d+2}{d}+\frac{\kappa+1}{d}) \rceil$, it follows that
	\[
		|\varphi_{2, t}(u)-\nabla n_t*q(u)|
		\lesssim \big(\mathrm{Poly}(\log m)m^{-\frac{\delta}{2d}}\big)m^{-\frac{\kappa+1}{d}}
		\lesssim m^{-\frac{\kappa+1}{d}},
	\]
	and hence also $|\varphi_{1, t}(u)-n_t*q(u)|\lesssim m^{-\frac{\kappa+1}{d}}$ as desired.
	Notice also that since $\kappa\ge 2d$, we have $\gamma>5d$ and hence $\nu=\frac{2d}{2\gamma-d}+\frac{1}{d}\le 1$.
\end{proof}

\paragraph{Step 3: extend fixed time approximations to time intervals}

\begin{lemma}
	\label{lem:time_interpolation}
	Under assumptions \ref{ass:H1}--\ref{ass:H3}, for $\delta>0$, large enough $m\in\N$ and $\underline{t}>0$ with $m^{-\frac{2\alpha+2}{2\alpha+d}}\lesssim \underline{t}\lesssim\log m$ there exists neural networks
	\[
	\varphi_1, \varphi_2
	\in \begin{cases}
		\widetilde{\Phi}(\log m\log\log m, m\log m, m\log^2 m, m^{\nu}\vee\underline{t}^{-1}), &\text{ if }\underline{t}\le \frac{1}{2}m^{-\frac{2-\delta}{d}}\\
		\widetilde{\Phi}(\log m\log\log m, m'\log m, m'\log^2m, m'), &\text{ if }\underline{t}>\frac{1}{2}m^{-\frac{2-\delta}{d}}
	\end{cases}
	\]
	where $\nu=\frac{2d}{2\alpha-d}+\frac{1}{d}$ and $m'=(\underline{t}\wedge 1)^{-\frac{d}{2}}m^{\frac{\delta}{2}}$ such that for $u\in\R^d$ and $t\in[\underline{t}, 2\underline{t}]$,
	\begin{align*}
		|(2\pi t)^{-\frac{d}{2}}f_1(u, t)-\varphi_{1}(u, t)|
		&\lesssim \begin{cases}
			(\log m)m^{-\frac{\alpha}{d}}, &\text{ if }\underline{t}\le \frac{1}{2}m^{-\frac{2-\delta}{d}}, u\in M^*_{-\varepsilon_M/2}\\
			\mathrlap{(\log m)^{\frac{d+2}{2}}m^{-\frac{\kappa}{d}},}\phantom{\frac{1}{\sqrt{\underline{t}\wedge 1}}(\log m)^{\frac{d+3}{2}}m^{-\frac{\kappa}{d}},} &\text{ if }\underline{t}\le \frac{1}{2}m^{-\frac{2-\delta}{d}}, u\notin M^*_{-\varepsilon_M/2} \\
			(\log m)m^{-\frac{\kappa+1}{d}}, &\text{ if }\underline{t}>\frac{1}{2}m^{-\frac{2-\delta}{d}}
		\end{cases}
		\intertext{and}
		|(2\pi t)^{-\frac{d}{2}}f_2(u, t)-\varphi_{2}(u, t)|
		&\lesssim \begin{cases}
			\frac{1}{\sqrt{\underline{t}\wedge 1}}(\log m)m^{-\frac{\alpha}{d}}, &\text{ if }\underline{t}\le \frac{1}{2}m^{-\frac{2-\delta}{d}}, u\in M^*_{-\varepsilon_M/2} \\
			\frac{1}{\sqrt{\underline{t}\wedge 1}}(\log m)^{\frac{d+3}{2}}m^{-\frac{\kappa}{d}}, &\text{ if }\underline{t}\le \frac{1}{2}m^{-\frac{2-\delta}{d}}, u\notin M^*_{-\varepsilon_M/2} \\
			\frac{1}{\sqrt{\underline{t}\wedge 1}}(\log m)m^{-\frac{\kappa+1}{d}}, &\text{ if }\underline{t}>\frac{1}{2}m^{-\frac{2-\delta}{d}}
		\end{cases}
	\end{align*}
	where $f_1, f_2$ are as in \eqref{eq:f_1f_2}.
\end{lemma}

\begin{proof}
	We start by constructing networks with the desired approximation rates and consider their sizes at the end.
	To this end, for notation, let
	\begin{align*}
		\varepsilon_1(u, t)
		&=\begin{cases}
			m^{-\frac{\alpha}{d}}, &\text{ if }t\le \frac{1}{2}m^{-\frac{2-\delta}{d}}, u\in M^*_{-\varepsilon_M/2}\\
			\mathrlap{(\log m)^{\frac{d}{2}}m^{-\frac{\kappa}{d}},}\phantom{\frac{1}{\sqrt{\underline{t}\wedge 1}}(\log m)^{\frac{d+1}{2}}m^{-\frac{\kappa}{d}},} &\text{ if }t\le \frac{1}{2}m^{-\frac{2-\delta}{d}}, u\notin M^*_{-\varepsilon_M/2} \\
			m^{-\frac{\kappa+1}{d}}, &\text{ if }t>\frac{1}{2}m^{-\frac{2-\delta}{d}}
		\end{cases}
		\intertext{and}
		\varepsilon_2(u, t)
		&=\begin{cases}
			\frac{1}{\sqrt{\underline{t}\wedge 1}}m^{-\frac{\alpha}{d}}, &\text{ if }t\le \frac{1}{2}m^{-\frac{2-\delta}{d}}, u\in M^*_{-\varepsilon_M/2} \\
			\frac{1}{\sqrt{\underline{t}\wedge 1}}(\log m)^{\frac{d+1}{2}}m^{-\frac{\kappa}{d}}, &\text{ if }t\le \frac{1}{2}m^{-\frac{2-\delta}{d}}, u\notin M^*_{-\varepsilon_M/2} \\
			\frac{1}{\sqrt{\underline{t}\wedge 1}}m^{-\frac{\kappa+1}{d}}, &\text{ if }t>\frac{1}{2}m^{-\frac{2-\delta}{d}}
		\end{cases}
	\end{align*}
	and for $t>0$ let $\varphi_{i, t}$ denote either the networks in Lemma \ref{lem:fixed_time_approx_small_t} if $t\le \frac{1}{2}m^{-\frac{2-\delta}{d}}$ or those in Lemma \ref{lem:fixed_time_approx_large_t} if $t>\frac{1}{2}m^{-\frac{2-\delta}{d}}$.
	In either case, we have $|(2\pi t)^{-\frac{d}{2}}f_i(u, t)-\varphi_{i, t}|\lesssim\varepsilon_i(u, t)$.
	Also, as in the previous proofs, let $q(u)=p_0(Au+v_0)$ and $n_t(u)=(2\pi t)^{-\frac{d}{2}}\mathrm{e}^{-\frac{|u|^2}{2t}}$, such that $(2\pi t)^{-\frac{d}{2}}f_1=n_t*q$ and $(2\pi t)^{-\frac{d}{2}}f_2=\nabla n_t*q$.
	The idea of the proof is, as in the proof of \cite[Lemma 3.13]{holk24}, to use polynomial interpolation in time between $\varphi_{1, t_i}$ and $\varphi_{2, t_i}$ for appropriate time points $\{t_i\}$.
	Since the time dependence of both $n_t*q$ and $\nabla n_t*q$ are well-behaved, for any fixed $u\in\R^d$, the functions $t\mapsto n_t*q(u)$ and $t\mapsto\nabla n_t*q(u)$ can be efficiently approximated by polynomial interpolation, and this property carries over to the neural network approximations, as we will show.

	To this end, we first center our time interval, as this makes analysis easier, so let $a=\frac{1}{2}\underline{t}$ and $b=\frac{3}{2}\underline{t}$, and set $n_{ t}^*(u)=n_{at+b}(u)$ for $t\in(-1, 1)$ such that $n_{(-1, 1)}^*(u)=n_{(\underline{t}, \overline{t})}(u)$.
	Then, for some $k\in\N$ to be determined later, let $\{t_i\}_{i=0}^k=\{\cos\frac{i\pi}{k}\}_{i=0}^k$ be the first $k+1$ Chebyshev nodes on $(-1, 1)$.
	Then, for $i=0,\ldots,k$, let $p_i(t)=\prod_{j\neq i}(t-t_j)$ and set $c_i=\frac{1}{p_i(t_i)}$.
	Furthermore, set
	\begin{align*}
		\varphi_{1}^*(u, t)
		&=\sum_{i=0}^{k}c_i\varphi_{\ell}^{\mathrm{mult}}\big(\varphi^{p_i}_{\ell}(t), \varphi_{1, t_i}^*(u)\big), \\
		\psi(u, t)
		&=\sum_{i=0}^{k}c_ip_i(t)\varphi_{1, t_i}^*(u),\qquad\text{and} \\
		P(u, t)
		&=\sum_{i=0}^{k}c_ip_i(t)n_{t_i}^**q(u).
	\end{align*}
	Here, $\varphi_{1, t}^*=\varphi_{1, at+b}$, while $\varphi^{\mathrm{mult}}_{\ell}$ is as in Lemma \ref{lem:mult_network_asymp} and $\varphi^{p_i}_{\ell}$ is a neural network approximations of $p_i$.
	In particular, we can construct $\varphi^{p_i}_{\ell}\in\widetilde{\Phi}(\ell\log k, k, k\ell, 1)$ such that
	\[
		|\varphi_{\ell}^{p_i}(t)-p_i(t)|\lesssim k2^{-\ell},\qquad\forall t\in[-1, 1].
	\]
	We defer this construction to (the proof of) \cite[Lemma 3.13]{holk24}.
	We then find by the triangle inequality that
	\begin{equation}
		\label{eq:interp_error}
		|\varphi_{1}^*(u, t)-n_{ t}^**q(u)|
		\le|\varphi_{1}^*(u, t)-\psi(u, t)|
		+|\psi(u, t)-P(u, t)|
		+|P(u, t)-n_{ t}^**q(u)|,
	\end{equation}
	and so setting $\varphi_{1}(u, t)=\varphi_{1}^*(u, \frac{2}{\underline{t}}t-3)$, the error analysis is completed if we can show that each of the above terms can be bounded by $\varepsilon_1(u, t)\log m$.
	Recalling from the proof of Lemma \ref{lem:fixed_time_approx_small_t} that $\n{\varphi_{1, t}}_{\infty}\le p_{\mathrm{max}}$, we find that
	\begin{align*}
		\big|\varphi_{\ell}^{\mathrm{mult}}\big(\varphi_{\ell}^{p_i}(t), \varphi_{1, t_i}^*(u)\big)-p_i(t)\varphi_{1, t_i}^*(u)\big|
		&\le\big|\varphi_{\ell}^{\mathrm{mult}}\big(\varphi_{\ell}^{p_i}(t), \varphi_{1, t_i}^*(u)\big)-\varphi_{\ell}^{p_i}(t)\varphi_{1, t_i}^*(u)\big|
		+p_{\mathrm{max}}|\varphi^{p_i}_{\ell}(t)-p_i(t)|
		\lesssim k2^{-\ell}.
	\end{align*}
	Furthermore, by \cite[Theorem 5.2]{trefethen13}, it holds that $|c_i|\le \frac{2^{k-1}}{k}$, and so the first term of \eqref{eq:interp_error} is upper bounded by
	\begin{align*}
		|\varphi_{1}^*(u,t)-\psi(u,t)|
		&\le\sum_{i=0}^{k}|c_i| |\varphi_{\ell_1}^{\mathrm{mult}}\big(\varphi^{p_i}_{\ell_2}(t), \varphi_{1, t_i}^*(u)\big)-p_i(t)\varphi_{1, t_i}^*(u)|
		\lesssim k2^{k-\ell}
	\end{align*}
	and choosing $\ell=\lceil k+\log_2 k+\frac{\kappa+1}{d}\log_2 m\rceil\asymp k+\log m$ bounds this term by $m^{-\frac{\kappa+1}{d}}\le\varepsilon_1 (u, t)$.
	For the second term of \eqref{eq:interp_error}, it can be shown that $|p_i(t)c_i|\lesssim 1$ (see Appendix), whence
	\[
		|\psi(u,t)-P(u,t)|
		\le\sum_{i=0}^{k}|c_ip_i(t)||\varphi_{1, t_i}^*(u)-n_{t_i}^**q(u)|
		\lesssim k\varepsilon_1(u, t).
	\]
	Finally, for the third term of \eqref{eq:interp_error}, we start by showing that for each fixed $u\in\R^d$, the function $t\mapsto n_{ t}*q(u)$ is analytically extendable to $\C^+\coloneqq\{w\in\C\mid\Re{w}>0\}$.
	To this end, we first see that $t\mapsto n_t(u)$ is analytic on $\C^+$ for all $u\in\R^d$ as the composition of an analytic function with a rational function with a pole at $0$.
	Thus, for each $w_0\in\C^+$, there exists an open neighbourhood $D_0$ of $w_0$ and integrable functions $\{a_n\}_{n\in\N_0}$ such that
	\[
		n_w(u)
		=\sum_{n=0}^{\infty}a_n(u)(w-w_0)^n,\qquad\forall w\in D_0,
	\]
	where this sum converges uniformly and absolutely on $D_0$.
	Since $q$ is a probability density, it then follows by dominated convergence that for $w\in D_0$
	\[
		n_w*q(u)
		=\int_{M^*}\Big(\sum_{n=0}^{\infty}a_n(u-v)(w-w_0)^n\Big)q(v)\diff{v}
		=\sum_{n=0}^{\infty}\Big(\int_{M^*}a_n(u-v)q(v)\diff{v}\Big)(w-w_0)^n,
	\]
	showing that $n_t*q$ is analytic as claimed.
	It then follows by \cite[Theorem 8.2]{trefethen13} that for $\rho>1$ satisfying $b-a(\frac{\rho+\rho^{-1}}{2})>0$, we have
	\[
		|P(u, t)-n_{t}^**q(u)|
		\le\frac{4R_{\rho}(u)\rho^{-k}}{\rho-1},\qquad\forall t\in(-1, 1),
	\]
	where
	\[
		R_{\rho}(u)
		=\max_{w\in\partial E_\rho}|n_{ w}^*(u)|,\quad\text{and}\quad
		\partial E_\rho
		=\left\{\frac{w+w^{-1}}{2}\mid |w|=\rho\right\}.
	\]
	We claim that $\rho=2$ works for our purposes.
	Indeed, we have $b-a(\frac{2+2^{-1}}{2})=\frac{7}{8}\underline{t}>0$, and since one readily checks that
	\[
		\min_{w\in\partial E_2}\frac{a\Re w+b}{|aw+b|^2}
		=\frac{1}{\frac{5}{4}a+b}
		=\frac{8}{17\underline{t}},
	\]
	we find that for $u\in\R^d$ and $w\in\partial E_2$
	\[
		|n_w^*(u)|
		=|(2\pi w)^{-\frac{d}{2}}\mathrm{e}^{-\frac{|u|^2}{2(aw+b)}}|
		=(2\pi|w|)^{-\frac{d}{2}}\mathrm{e}^{-\frac{|u|^2}{2}\cdot\frac{a\Re w+b}{|aw+b|^2}}
		\le\Big(\frac{14\pi}{8}\underline{t}\Big)^{-\frac{d}{2}}\mathrm{e}^{-\frac{4|u|^2}{17\underline{t}}},
	\]
	whence
	\[
		R_2(u)
		\le\int_{M^*}\Big(\frac{7\pi}{4}\underline{t}\Big)^{-\frac{d}{2}}\mathrm{e}^{-\frac{4|u-v|^2}{17\underline{t}}}q(v)\diff{v}
		=\Big(\frac{17}{7}\Big)^{\frac{d}{2}}\int_{M^*}\Big(2\pi\Big(\frac{17}{8}\underline{t}\Big)\Big)^{-\frac{d}{2}}\mathrm{e}^{-\frac{|u-v|^2}{2(\frac{17}{8}\underline{t})}}q(v)\diff{v}
		=\Big(\frac{17}{7}\Big)^{\frac{d}{2}}n_{\frac{17}{8}\underline{t}}*q(u).
	\]
	Now by Young's convolution inequality, we find that
	\[
		\n{R_2}_{L^\infty}
		\le \Big(\frac{17}{7}\Big)^{\frac{d}{2}}\n{n_{\frac{17}{8}\underline{t}}}_{L^1}\n{q}_{L^\infty}
		\le\Big(\frac{17}{7}\Big)^{\frac{d}{2}} p_{\mathrm{max}},
	\]
	whereby
	\[
		|P(u, t)-n_t^**q(u)|
		\lesssim2^{-k},
	\]
	and choosing $k=\lceil\frac{\kappa+1}{d}\log m\rceil\asymp\log m$ ensures that this is bounded by $m^{-\frac{\kappa+1}{d}}\le\varepsilon_1(u, t)$.

	The exact same strategy can be used to approximate $\nabla n_t*q$, i.e. if we set in a recycling of notation
	\begin{align*}
		\varphi_{2}^*(u, t)
		&=\sum_{i=0}^{k}c_i\varphi_{\ell}^{\mathrm{mult}, d}\big(\varphi^{p_i}_{\ell}(t), \varphi_{2, t_i}^*(u)\big), \\
		\psi(u, t)
		&=\sum_{i=0}^{k}c_ip_i(t)\varphi_{2, t_i}^*(u),\qquad\text{and} \\
		P(u, t)
		&=\sum_{i=0}^{k}c_ip_i(t)\nabla n_{t_i}^**q(u),
	\end{align*}
	then we once again have by the triangle inequality
	\begin{equation}
		\label{eq:interp_error_2}
		|\varphi_{2}^*(u, t)-\nabla n_{t}^**q(u)|
		\le|\varphi_{2}^*(u, t)-\psi(u, t)|
		+|\psi(u, t)-P(u, t)|
		+|P(u, t)-\nabla n_{ t}^**q(u)|.
	\end{equation}
	Here, using the exact same approach as before, we can bound the first two terms by $\varepsilon_2(u, t)$.
	As for the third term, noting that $|P(u, t)-\nabla n_t*q(u)|^2=\sum_{j=1}^{d}|P_j(u, t)-\partial_jn_t*q(u)|^2$, we can use the same method as before to bound each of these summands.
	In particular, following the same steps as above and defining $R_{2, j}$ as above for the $j$'th summand, we would find that
	\[
		\n{R_{2, j}}_{L^\infty}
		\le\Big(\frac{17}{7}\Big)^{\frac{d}{2}}\n{\partial_jn_{\frac{17}{8}\underline{t}}}_{L^1}\n{q}_{L^\infty},
	\]
	and thus by \eqref{eq:dn_t_l1},
	\[
		|P(u, t)-\nabla n_t*q(u)|^2
		\lesssim 2^{-2k}\sum_{j=1}^{d}\n{\partial_jn_{\frac{17}{8}\underline{t}}}_{L^1}^2
		\lesssim\frac{1}{\underline{t}}m^{-\frac{2(\kappa+1)}{d}}
		\le \varepsilon_2(u, t)^2.
	\]
	As for the sizes of the networks, we first recall that if $\underline{t}\le \frac{1}{2}m^{-\frac{2-\delta}{d}}$, we have for each $i$ that $\varphi_{1, t_i}\in\widetilde{\Phi}(\log m, m, m \log m, m^{\nu}\vee(\underline{t}\log m)^{-\frac{1}{2}})$, while $\varphi_{\ell}^{p_i}\in\widetilde{\Phi}(\log m\log\log m, \log m, \log^2m, 1)$, whereby each summand in the definition of $\varphi_1^*$ is in $\widetilde{\Phi}(\log m\log\log m, m, m\log m, m^{\nu}\vee(\underline{t}\log m)^{-\frac{1}{2}})$, and since there are $k\asymp\log m$ such terms, we have $\varphi_1\in\widetilde{\Phi}(\log m\log\log m, m\log m, m\log^2 m, m^{\nu}\vee\underline{t}^{-1})$.
	Similarly, if $\underline{t}>\frac{1}{2}m^{-\frac{2-\delta}{d}}$, we have $\varphi_{1, t_i}\in\widetilde{\Phi}(\log m', m', m'\log m', m')$, and hence by the same argumentation $\varphi_1\in\widetilde{\Phi}(\log m\log\log m, m', m'\log m', m')$.
	Similar analysis shows that the same is true of $\varphi_2$, finishing the proof.
\end{proof}

\subsubsection*{Step 4: Putting things together}
With all of the above we are now in a position to prove Theorem \ref{theo:score_approx}.

\begin{proof}[Proof of Theorem \ref{theo:score_approx}]
	First, by Lemmas \ref{lem:affine_approx_b} and \ref{lem:truncation_error} and the triangle inequality we need only approximate $s_0^K\bm{1}_{M_{\rho, \underline{t}}}$ with
	\begin{align*}
		\rho
		=K
		=\frac{2(\alpha+1)}{d}\log m+\log\underline{t}^{-1}
		\lesssim\log m.
	\end{align*}
	To this end, we let for sake of notation $x^*=A^\top(x-v_0)\in\R^d$ and $x^\perp=(I-P)(x-v_0)\in\R^D$ for $x\in\R^D$ such that $x^*$ is the local coordinates of the projection $Px$ of $x$ onto $M$, while $x^\perp$ is the perpendicular component of $x$ with respect to $M$.
	Recalling then that
	\begin{align*}
		\int_M \mathrm{e}^{-\frac{|x-y|^2}{2t}}\,\mu(\diff{y})
		&=\mathrm{e}^{-\frac{|x^{\perp}|^2}{2t}}\int_{M^*}\mathrm{e}^{-\frac{|x^*-u|}{2t}}p_0(Au+v_0)\diff{u}
		=\mathrm{e}^{-\frac{|x^{\perp}|^2}{2t}}f_1(x^*, t),
	\end{align*}
	and
	\begin{align*}
		\int_M \frac{x-y}{t}\mathrm{e}^{-\frac{|x-y|^2}{2t}}\,\mu(\diff{y})
		&=\mathrm{e}^{-\frac{|x^{\perp}|^2}{2t}}\bigg(\frac{x^{\perp}}{t}\int_{M^*}\mathrm{e}^{-\frac{|x^*-u|^2}{2t}}p_0(Au+v_0)\diff{u}+A\int_{M^*}\frac{x^*-u}{t}\mathrm{e}^{-\frac{|x^*-u|^2}{2t}}p_0(Au+v_0)\diff{u}\bigg) \\
		&=\mathrm{e}^{-\frac{|x^{\perp}|^2}{2t}}\bigg(\frac{x^{\perp}}{t}f_1(x^*, t)+Af_2(x^*, t)\bigg),
	\end{align*}
	we let $h_1(x, t)=(2\pi t)^{-\frac{d}{2}}\mathrm{e}^{-\frac{|x^{\perp}|^2}{2t}}f_1(x^*, t)$ and $h_2(x, t)=(2\pi t)^{-\frac{d}{2}}\mathrm{e}^{-\frac{|x^{\perp}|^2}{2t}}\big(\frac{x^{\perp}}{t}f_1(x^*, t)+Af_2(x^*, t)\big)$ such that
	\[
		s_0^K(x, t)
		=\frac{-\sum_{\substack{z\in\Z^D \\ \n{z}_{\infty}\le K_{\underline{t}}}}(-1)^zh_2(R_z(x)+z, t)}{\sum_{\substack{z\in\Z^D \\ \n{z}_{\infty}\le K_{\underline{t}}}}h_1(R_z(x)+z, t)},
	\]
	where $K_{\underline{t}}=\sqrt{2\underline{t}(D+2K)}$.
	Furthermore, let $\widehat{h}_1(x, t)=\mathrm{e}^{-\frac{|x^{\perp}|^2}{2t}}\varphi_1(x^*, t)$ and $\widehat{h}_2(x, t)=\mathrm{e}^{-\frac{|x^{\perp}|^2}{2t}}\big(\frac{x^{\perp}}{t}\varphi_1(x^*, t)+A\varphi_2(x^*, t)\big)$, where $\varphi_1, \varphi_2$ are as in Lemma \ref{lem:time_interpolation}. By (the proof of) Lemma \ref{lem:affine_approx_d} there exists a constant $c > 0$ such that for all $(x, t)\in M_{\rho, \underline{t}}\times[\underline{t}, 2\underline{t}]$,
	\begin{equation}\label{eq:lower_main_proof}
		p_t^K(x)
		\geq c t^{\frac{c_0-D}{2}}\mathrm{e}^{-\rho}
		= ct^{\frac{c_0-D}{2}+1}m^{-\frac{2(\alpha+1)}{d}},
	\end{equation} 
    and we set
	\[
		\widehat{s}_0^K(x, t)
		=\frac{-\sum_{\substack{z\in\Z^D \\ \n{z}_{\infty}\le K_{\underline{t}}}}(-1)^z\widehat{h}_2(R_z(x)+z, t)}{\Big(\sum_{\substack{z\in\Z^D \\ \n{z}_{\infty}\le K_{\underline{t}}}}\widehat{h}_1(R_z(x)+z, t)\Big)\vee c\underline{t}^{\frac{c_0-d}{2}+1}m^{-\frac{2(\alpha+1) }{d}}}
		\eqcolon\frac{\widehat{\nabla p}_t^K(x)}{\widehat{p}_t^K(x)}.
	\]
	To clearly separate our sources of error, we first show that $\widehat{s}_0^K$ is a good approximation of $s_0^K$ on $M_{\rho, \underline{t}}\times[\underline{t}, 2\underline{t}]$ and only then approximate $\widehat{s}_0^K$ by a neural network.
	To this end, we have first
	\begin{align*}
		|s_0^K(x, t)-\widehat{s}_0^K(x, t)|
		&=\Big|\frac{\widehat{\nabla p}_t^K(x)}{\widehat{p}_t^K(x)}-\frac{(2\pi t)^{\frac{D-d}{2}}\nabla p_t^K(x)}{(2\pi t)^{\frac{D-d}{2}}p_t^K(x)}\Big|\\
		&\le\frac{1}{p_t^K(x)\widehat{p}_t^K(x)}\Big(|\nabla p_t^K(x)|\big|(2\pi t)^{\frac{D-d}{2}}p_t^K(x)-\widehat{p}_t^K(x)\big| \\
		&\qquad\qquad\qquad\qquad+p_t^K(x)\big|(2\pi t)^{\frac{D-d}{2}}\nabla p_t^K(x)-\widehat{\nabla p}_t^K(x)\big|\Big) \\
		&=\frac{1}{\widehat{p}_t^K(x)}\Big(|s_0^K(x, t)|\big|(2\pi t)^{\frac{D-d}{2}}p_t^K(x)-\widehat{p}_t^K(x)\big|+\big|(2\pi t)^{\frac{D-d}{2}}\nabla p_t^K(x)-\widehat{\nabla p}_t^K(x)\big|\Big).
	\end{align*}
	Then by \eqref{eq:lower_main_proof},
	\[
		|(2\pi t)^{\frac{D-d}{2}}p_t^K(x)-\widehat{p}_t^K(x)|
		\lesssim\Big|(2\pi t)^{\frac{D-d}{2}}p_t^K(x)-\sum_{\substack{z\in\Z^D \\ \n{z}_{\infty}\le K_{\underline{t}}}}\widehat{h}_1(R_z(x)+z, t)\Big|.
	\]
	Next, we further split this error into four parts to analyse separately.
	In particular, let $S_1=\{x\in\R^D\mid x^*\in M^*_{-\varepsilon_M/2}\}$ and $S_2=\R^D\setminus S_1$.
	Furthermore, for $x\in[0, 1]^D$, and $z\in\Z^D$ let $x_z=R_z(x)+z$ and set $Z^K_i(x)=\{z\in\Z^D\mid \n{z}_{\infty}\le K_{\underline{t}}, x_z\in S_i\}$.
	Repeated use of the triangle inequality (along with the fact that $A$ and $(-1)^z$ are orthogonal matrices) then yields that the distance $|s_0^K(x, t)-\widehat{s}_0^K(x, t)|$ is upper bounded by the sum
	\begin{align}
		&\frac{1}{\widehat{p}_t^K(x)}|s_0^K(x, t)|\sum_{z\in Z_1^K(x)}\mathrm{e}^{-\frac{|x_z^\perp|^2}{2t}}|(2\pi t)^{-\frac{d}{2}}f_1(x_z^*, t)-\varphi_1(x_z^*, t)| \label{eq:err_decomp_1} \\
		&\quad +\frac{1}{\widehat{p}_t^K(x)}|s_0^K(x, t)|\sum_{z\in Z_2^K(x)}\mathrm{e}^{-\frac{|x_z^\perp|^2}{2t}}|(2\pi t)^{-\frac{d}{2}}f_1(x_z^*, t)-\varphi_1(x_z^*, t)| \label{eq:err_decomp_2} \\
		&\quad +\frac{1}{\widehat{p}_t^K(x)}\sum_{z\in Z_1^K(x)}\mathrm{e}^{-\frac{|x_z^\perp|^2}{2t}}\Big(\frac{|x_z^\perp|}{t}|(2\pi t)^{-\frac{d}{2}}f_1(x_z^*, t)-\varphi_1(x_z^*, t)|+|(2\pi t)^{-\frac{d}{2}}f_2(x_z^*, t)-\varphi_2(x_z^*, t)|\Big) \label{eq:err_decomp_3} \\
		&\quad +\frac{1}{\widehat{p}_t^K(x)}\sum_{z\in Z_2^K(x)}\mathrm{e}^{-\frac{|x_z^\perp|^2}{2t}}\Big(\frac{|x_z^\perp|}{t}|(2\pi t)^{-\frac{d}{2}}f_1(x_z^*, t)-\varphi_1(x_z^*, t)|+|(2\pi t)^{-\frac{d}{2}}f_2(x_z^*, t)-\varphi_2(x_z^*, t)|\Big) \label{eq:err_decomp_4}.
	\end{align}
	We will analyse each of these terms separately.
	To ease notation, let $\varepsilon_{\underline{t}}$ denote either $m^{-\frac{1}{d}}$ if $\underline{t}>m^{-\frac{2-\delta}{d}}$ and $1$ otherwise.

\smallskip

\noindent \textbf{Term \eqref{eq:err_decomp_1}}:
	For $z\in Z_1^K(x)$, we have $|(2\pi t)^{-\frac{d}{2}}f_1(x_z^*, t)-\varphi_1(x_z^*, t )|\lesssim m^{-\frac{\alpha}{d}}\varepsilon_{\underline{t}}\log m$ by Lemma \ref{lem:time_interpolation}, while also
	\[
		f_1(x_z^*, t)
		\ge\int_{\mathcal B(x^*_z, \sqrt{t})\cap M^*_{-\varepsilon_M/2}}\mathrm{e}^{-\frac{|x_z^*-u|^2}{2t}}p_0(Au+v_0)\diff{u}
		\ge \mathrm{e}^{-\frac{1}{2}}p_{\mathrm{min}}\mathrm{Vol}_d(\mathcal B(x_z^*, \sqrt{t})\cap M_{-\varepsilon_M/2}^*)
		\gtrsim (t\wedge r_0^2)^{\frac{d}{2}}
	\]
	by assumption \ref{ass:H1}.
	This implies in particular that
	\begin{align*}
		\widehat{p}_t^K(x)
		&\ge \sum_{z\in Z_1^K(x)}\mathrm{e}^{-\frac{|x_z^\perp|^2}{2t}}\varphi_{1}(x_z^*, t) \\
		&\ge \sum_{z\in Z_1^K(x)}\mathrm{e}^{-\frac{|x_z^\perp|^2}{2t}}\big( (2\pi t)^{-\frac{d}{2}}f_1(x_z^*, t)-|(2\pi t)^{-\frac{d}{2}}f_1(x_z^*, t)-\varphi_{1}(x_z^*, t)|\big) \\
		&\gtrsim \sum_{z\in Z_1^K(x)}\mathrm{e}^{-\frac{|x_z^\perp|^2}{2t}}\big( (1\wedge t^{-\frac{d}{2}}r_0^d)-m^{-\frac{\alpha}{d}}\log m\big) \\
		&\gtrsim (\log m)^{-\frac{d}{2}}\sum_{z\in Z_1^K(x)}\mathrm{e}^{-\frac{|x_z^\perp|^2}{2t}},
	\end{align*}
	where we use that $t^{-1}\ge(2\underline{t})^{-1}\gtrsim(\log m)^{-1}$.
	Combining these, we have
	\begin{align*}
		\frac{1}{\widehat{p}_t^K(x)}|s_0^K(x, t)|\sum_{z\in Z_1^K(x)}\mathrm{e}^{-\frac{|x_z^\perp|^2}{2t}}|(2\pi t)^{-\frac{d}{2}}f_1( x_z^*, t)-\varphi_1( x_z^*, t )|
		\lesssim |s_0^K(x, t)|m^{-\frac{\alpha}{d}}\varepsilon_{\underline{t}}(\log m)^{\frac{d+2}{2}}.
	\end{align*}

\smallskip

\noindent \textbf{Term \eqref{eq:err_decomp_2}}:
	For $z\in Z_2^K(x)$, we can no longer lower bound $\widehat{p}_t^K(x)$ as we did above.
	Instead, we have by definition that $\widehat{p}_t^K(x)\ge\underline{t}^{\frac{c_0-d}{2}+1}m^{-\frac{2(\alpha+1) }{d}}$.
	Furthermore, using Lemma \ref{lem:time_interpolation},
	\[
		|(2\pi t)^{-\frac{d}{2}}f_1(x_z^*, t)-\varphi_1(x_z^*, t )|
		\lesssim (\log m)^{\frac{d+2}{2}}m^{-\frac{\kappa}{d}}\varepsilon_{\underline{t}}
		\lesssim\underline{t}^{\frac{c_0-d}{2}+1}(\log m)^{\frac{d+2}{2}}m^{-\frac{3\alpha+2}{d}}\varepsilon_{\underline{t}},
		\]
		where we used that $\underline{t}\gtrsim m^{-\frac{2\alpha+2}{2\alpha+d}}\ge m^{-1}$.
	Thus it follows that
	\begin{align*}
		\frac{1}{\widehat{p}_t^K(x)}&|s_0^K(x, t)|\sum_{z\in Z_2^K(x)}\mathrm{e}^{-\frac{|x_z^\perp|^2}{2t}}|(2\pi t)^{-\frac{d}{2}}f_1(x_z^*, t)-\varphi_1(x_z^*, t)| \\
		&\lesssim\frac{m^{\frac{2(\alpha+1)}{d}}}{\underline{t}^{\frac{c_0-d}{2}+1}}|s_0^K(x, t)|\sum_{z\in Z_2^K(x)}\mathrm{e}^{-\frac{|x_z^\perp|^2}{2t}}\underline{t}^{\frac{c_0-d}{2}+1}(\log m)^{\frac{d+2}{2}}m^{-\frac{3\alpha+2}{d}}\varepsilon_{\underline{t}} \\
		&= |s_0^K(x, t)|(\log m)^{\frac{d+2}{2}}m^{-\frac{\alpha}{d}}\varepsilon_{\underline{t}}\sum_{z\in Z_2^K(x)}\mathrm{e}^{-\frac{|x_z^\perp|^2}{2t}} \\
		&\lesssim |s_0^K(x, t)|(\log m)^{\frac{d+2D+2}{2}}m^{-\frac{\alpha}{d}}\varepsilon_{\underline{t}},
	\end{align*}
	where in the last step we use that
	\[
		\#Z_2^K(x)
		\le\#\{z\in\Z^D\mid \n{z}_{\infty}\le K_{\underline{t}}\}
		\leq (2K_{\underline{t}}+1)^D
		\lesssim (\log m)^D.
	\]

\smallskip

\noindent \textbf{Term \eqref{eq:err_decomp_3}}:	Here again using Lemma \ref{lem:time_interpolation} we have $|(2\pi t)^{-\frac{d}{2}}f_1(x_z^*, t)-\varphi_1(x_z^*, t )|\lesssim m^{-\frac{\alpha}{d}}\varepsilon_{\underline{t}}\log m$ and $|(2\pi t)^{-\frac{d}{2}}f_2(x_z^*, t)-\varphi_2(x_z^*, t )|\lesssim \frac{1}{\sqrt{\underline{t}\wedge 1}}m^{-\frac{\alpha}{d}}\varepsilon_{\underline{t}}\log m$, and so by the exact same reasoning as in case \eqref{eq:err_decomp_1},
	\[
		\frac{1}{\widehat{p}_t^K(x)}\sum_{z\in Z_1^K(x)}\mathrm{e}^{-\frac{|x_z^\perp|^2}{2t}}|(2\pi t)^{-\frac{d}{2}}f_2(x_z^*, t)-\varphi_2(x_z^*, t)|
		\lesssim \frac{1}{\sqrt{\underline{t}\wedge 1}}m^{-\frac{\alpha}{d}}\varepsilon_{\underline{t}}(\log m)^{\frac{d+2}{2}}.
	\]
	Also, like in case \eqref{eq:err_decomp_1}, we have
	\begin{align*}
		\frac{1}{\widehat{p}_t^K(x)}\sum_{z\in Z_{1}^K(x)}\mathrm{e}^{-\frac{|x_z^\perp|^2}{2t}}\frac{|x_z^\perp|}{t}|(2\pi t)^{-\frac{d}{2}}f_1(x_z^*, t)-\varphi_1(x_z^*, t)|
		&\lesssim (\log m)^{\frac{d+2}{2}}m^{-\frac{\alpha}{d}}\varepsilon_{\underline{t}}\frac{\sum_{z\in Z_{1}^K(x)}\mathrm{e}^{-\frac{|x_z^\perp|^2}{2t}}\frac{|x_z^\perp|}{t}}{\sum_{z\in Z_{1}^K(x)}\mathrm{e}^{-\frac{|x_z^\perp|^2}{2t}}},
	\end{align*}
	and to bound this, we first note that for all $z\in\Z^D$ with $\n{z}_{\infty}\le K_{\underline{t}}$, we have $x_z\in[-K_{\underline{t}}, K_{\underline{t}}]^D$, whence $|x_z^\perp|\le 2\sqrt{D}K_{\underline{t}}\lesssim\sqrt{t}\log m$ and so
	\[
		(\log m)^{\frac{d+2}{2}}m^{-\frac{\alpha}{d}}\varepsilon_{\underline{t}}\frac{\sum_{z\in Z_{1}^K(x)}\mathrm{e}^{-\frac{|x_z^\perp|^2}{2t}}\frac{|x_z^\perp|}{t}}{\sum_{z\in Z_{1}^K(x)}\mathrm{e}^{-\frac{|x_z^\perp|^2}{2t}}}
		\lesssim\frac{1}{\sqrt{\underline{t}\wedge 1}}(\log m)^{\frac{d+4}{2}} m^{-\frac{\alpha}{d}}\varepsilon_{\underline{t}}.
	\]

\smallskip

\noindent \textbf{Term \eqref{eq:err_decomp_4}}: Repeating the arguments used in cases \eqref{eq:err_decomp_2} and \eqref{eq:err_decomp_3}, we obtain
	\begin{align*}
		\frac{1}{\widehat{p}_t^K(x)}&\sum_{z\in Z_2^K(x)}\mathrm{e}^{-\frac{|x_z^\perp|^2}{2t}}|(2\pi t)^{-\frac{d}{2}}f_2(x_z^*, t)-\varphi_2(x_z^*, t)| \\
		&\lesssim \frac{1}{\sqrt{\underline{t}\wedge 1}}(\log m)^{\frac{d+3}{2}}m^{-\frac{\alpha}{d}}\varepsilon_{\underline{t}}\sum_{z\in Z_2^K(x)}\mathrm{e}^{-\frac{|x_z^\perp|^2}{2t}} \\
		&\lesssim \frac{1}{\sqrt{\underline{t}\wedge 1}}(\log m)^{\frac{d+2D+3}{2}}m^{-\frac{\alpha}{d}}\varepsilon_{\underline{t}},
	\end{align*}
	and, since $\mathrm{e}^{-\frac{r}{2t}}\frac{r}{t}\le\frac{1}{\sqrt{t}}$ for all $r>0$,
	\begin{align*}
		\frac{1}{\widehat{p}_t^K(x)}&\sum_{z\in Z_{2}^K(x)}\mathrm{e}^{-\frac{|x_z^\perp|^2}{2t}}\frac{|x_z^\perp|}{t}|(2\pi t)^{-\frac{d}{2}}f_1(x_z^*, t)-\varphi_1(x_z^*, t)|
		\lesssim\frac{1}{\sqrt{\underline{t}\wedge 1}}(\log m)^{\frac{d+2D+3}{2}}m^{-\frac{\alpha}{d}}\varepsilon_{\underline{t}}.
	\end{align*}

\smallskip

\noindent Combining all of these different cases, we see that for $x\in M_{\rho, \underline{t}}$
	\begin{equation}
		\label{eq:s_0^K_bound}
		|s_0^K(x, t)-\widehat{s}_0^K(x, t)|
		\lesssim\Big(|s_0^K(x, t)|+\frac{1}{\sqrt{\underline{t}\wedge 1}}\Big)(\log m)^{\frac{d+2D+3}{2}}m^{-\frac{\alpha}{d}}\varepsilon_{\underline{t}},
	\end{equation}
	whence
	\[
		\int_{\underline{t}}^{2\underline{t}}\mathbb{E}\big[|s_0^K(X_t, t)-\widehat{s}_0^K(X_t, t)|^2\bm{1}_{M_{\rho, t}}(X_t)\big]\diff{t}
		\lesssim \Big(\int_{\underline{t}}^{2\underline{t}}\mathbb{E}[|s_0^K(X_t, t)|^2]\diff{t}+\frac{\underline{t}}{\underline{t}\wedge 1}\Big)(\log m)^{d+2D+3}m^{-\frac{2\alpha}{d}}\varepsilon_{\underline{t}}.
	\]
	Now, to estimate the remaining integral, we first have by Lemma \ref{lem:truncation_error} and our choice of $K$ that
	\begin{align*}
		\int_{\underline{t}}^{2\underline{t}}\mathbb{E}[|s_0^K(X_t, t)|^2]\diff{t}
		&\leq 2 \int_{\underline{t}}^{2\underline{t}}\big(\mathbb{E}[|s_0(X_t, t)|^2]+\mathbb{E}[|s_0(X_t, t)-s_0^K(X_t, t)|^2]\big)\diff{t} \\
		&\lesssim \int_{\underline{t}}^{2\underline{t}}\mathbb{E}[|s_0(X_t, t)|^2]\diff{t}+(\log m)^{\frac{D}{2}}m^{-\frac{2(\alpha+1)}{d}}.
	\end{align*}
 	Furthermore, since $(x,t) \mapsto p_t(x)$ is a positive solution of the heat equation $\partial_t u(t,x) = \frac{1}{2}\Delta u(t,x)$ with Neumann boundary conditions on $M \times (0,\infty)$ for the compact, convex set $M = [0,1]^D$, the Li--Yau bound \cite[Theorem 1.1]{liyau} yields that
	\[\lvert s_0(x,t) \rvert^2 = \lvert \nabla \log p_t(x) \rvert^2 \lesssim \partial_t \log p_t(x) + \frac{D}{t}, \quad (x,t) \in [0,1]^d \times (0,\infty).\]
	Thus,
	\[
		\int_{\underline{t}}^{2\underline{t}}\mathbb{E}[|s_0(X_t, t)|^2]\diff{t}
		\lesssim \int_{\underline{t}}^{2\underline{t}}\mathbb{E}\Big[\Big(\frac{\partial_tp_t(X_t)}{p_t(X_t)}\Big]+\frac{D}{t}\Big)\diff{t}
		=D\int_{\underline{t}}^{2\underline{t}}\frac{1}{t}\diff{t}
		= D \log 2,
	\]
	where we use that by Fubini--Tonelli's theorem
	\[
		\int_{\underline{t}}^{2\underline{t}}\mathbb{E}\Big[\frac{\partial_tp_t(X_t)}{p_t(X_t)}\Big]\diff{t}
		=\int_{\underline{t}}^{2\underline{t}}\int_{[0, 1]^D}\partial_tp_t(x)\,\mathrm{d}x\diff{t}
		=\int_{[0, 1]^D}\int_{\underline{t}}^{2\underline{t}}\partial_tp_t(x)\diff{t}\,\mathrm{d}x
		=0.
	\]
	Inserting this into the above, we have that
	\begin{equation}
		\label{eq:s^K_hat_error}
		\int_{\underline{t}}^{2\underline{t}}\mathbb{E}\big[|s_0^K(X_t, t)-\widehat{s}_0^K(X_t, t)|^2\bm{1}_{M_{\rho, t}}(X_t)\big]\diff{t}
		\lesssim (\log m)^{d+2D+3}m^{-\frac{2\alpha}{d}}\varepsilon_{\underline{t}}^2,
	\end{equation}
	as desired.
	With this established, all that is left is to approximate $\widehat{s}_0^K\bm{1}_{M_{\rho, \underline{t}}}$ by a neural network $\varphi_{s_0}$.

	To this end, letting $\varphi^{\mathrm{exp}}_{\ell}$, $\varphi^{\mathrm{mult}}_{\ell}$, $\varphi^{\mathrm{rec}}_{\ell}$ and $\varphi^{\mathrm{norm}}_{\ell}$ be as in Lemmas \ref{lem:exp_network}, \ref{lem:mult_network_asymp}, \ref{lem:rec_network_asymp} and \ref{lem:norm_network} and setting
	\[
		\widehat{\varphi}^{\mathrm{exp}}_{\ell}(x, t)
		=\varphi^{\mathrm{exp}}_{\ell}\Big(\frac{1}{2}\varphi^{\mathrm{mult}}_{\ell}\big(\varphi^{\mathrm{rec}}_{\ell}(t), \varphi^{\mathrm{norm}}_{\ell}(x)\big)\Big),
	\]
	we have
	\begin{align*}
		\big|\mathrm{e}^{-\frac{|x|^2}{2t}}-\widehat{\varphi}_{\ell}^{\mathrm{exp}}(x, t)\big|
		&\le 2^{-\ell}+\big|\mathrm{e}^{-\frac{|x|^2}{2t}}-\mathrm{e}^{-\frac{1}{2}\varphi^{\mathrm{mult}}_{\ell}(\varphi^{\mathrm{rec}}_{\ell}(t), \varphi^{\mathrm{norm}}_{\ell}(x))}\big| \\
		&\lesssim 2^{-\ell}+\Big|\frac{|x|^2}{t}-\varphi^{\mathrm{mult}}_{\ell}\big(\varphi^{\mathrm{rec}}_{\ell}(t), \varphi^{\mathrm{norm}}_{\ell}(x)\big)\Big| \\
		&\lesssim K_{\underline{t}}2^{-\ell}+\Big|\frac{|x|^2}{t}-\varphi^{\mathrm{rec}}_{\ell}(t)\varphi^{\mathrm{norm}}_{\ell}(x)\Big| \\
		&\lesssim K_{\underline{t}}2^{-\ell}+K_{\underline{t}}^2\Big|\frac{1}{t}-\varphi^{\mathrm{rec}}_{\ell}(t)\Big|+\frac{1}{\underline{t}}\big||x|^2-\varphi^{\mathrm{norm}}_{\ell}(x)\big| \\
		&\lesssim \frac{K_{\underline{t}}^2}{\underline{t}}2^{-\ell}
	\end{align*}
	for all $x\in\R^D$ with $\n{x}_{\infty}\le K_{\underline{t}}$ and all $t\in[\underline{t}, 2\underline{t}]$.
	Next, let
	\begin{align*}
		\varphi_{h_1}(x, t)
		&=\varphi^{\mathrm{mult}}_{\ell}\big(\widehat{\varphi}^{\mathrm{exp}}_{\ell}(x^\perp, t), \varphi_1(x^*, t)\big), \\
		\widetilde{\varphi}_1(x, t)
		&=\varphi^{\mathrm{mult}, D}_{\ell}\big(\varphi_1(x^*, t), \varphi^{\mathrm{mult}, D}_{\ell}(\varphi^{\mathrm{rec}}_{\ell}(t), x^{\perp})\big),\quad\text{and} \\
		\varphi_{h_2}(x, t)
		&=\varphi_{\ell}^{\mathrm{mult}, D}\big(\widehat{\varphi}_{\ell}^{\mathrm{exp}}( x^{\perp}, t), \widetilde{\varphi}_1(x, t)+A\varphi_2(x^*, t)\big),
	\end{align*}
	and it follows that
	\[
		|\widehat{h}_1(x, t)-\varphi_{h_1}(x, t)|
		\lesssim |\varphi_1(x^*, t)|\Big(1+\frac{K_{\underline{t}}^2}{\underline{t}}\Big)2^{-\ell}
		\lesssim (2\pi\underline{t})^{-\frac{d}{2}}\frac{K_{\underline{t}}^2}{\underline{t}}2^{-\ell},
	\]
	where we once again use that we can assume $|\varphi_1|\le(2\pi\underline{t})^{-\frac{d}{2}}|f_1|\le(2\pi \underline{t})^{-\frac{d}{2}}$ by Lemma \ref{lem:capping_network}.
	Similarly, 
	\begin{align*}
		\Big|\widetilde{\varphi}_1(x, t)-\frac{x^{\perp}}{t}\varphi_1(x, t)\Big|
		\lesssim K_{\underline{t}}(2\pi\underline{t})^{-\frac{d}{2}}\Big(1+\frac{1}{\underline{t}}\Big)2^{-\ell},
	\end{align*}
	and so, since we may once again assume $|\widetilde{\varphi}_1+A\varphi_2|\lesssim(2\pi\underline{t})^{-\frac{d}{2}}\frac{K_{\underline{t}}}{\underline{t}}$,
	\begin{align*}
		|\widehat{h}_2(x, t)-\varphi_{h_2}(x, t)|
		\lesssim (2\pi\underline{t})^{-\frac{d}{2}}\frac{K_{\underline{t}}^2}{\underline{t}} 2^{-\ell}.
	\end{align*}
	Setting $\varphi_{p_t^K}(x, t)=\sum_{\substack{z\in\Z^D \\ \n{z}_{\infty}\le K_{\underline{t}}}}\varphi_{h_1}(x_z, t)$ and $\varphi_{\nabla p_t^K}(x, t)=-\sum_{\substack{z\in\Z^D \\ \n{z}_{\infty}\le K_{\underline{t}}}}(-1)^z\varphi_{h_2}(x_z, t)$, it thus follows that
	\[
		|\varphi_{p_t^K}(x, t)-\widehat{p}_t^K(x, t)|,
		|\varphi_{\nabla p_t^K}(x, t)-\widehat{\nabla p}_t^K(x, t)|
		\lesssim (2\pi \underline{t})^{-\frac{d}{2}}\frac{K_{\underline{t}}^{D+2}}{\underline{t}}2^{-\ell}.
	\]
	Finally, let
	\[
		\varphi_{s_0}(x, t)
		=\varphi^{\mathrm{mult}, D}_{\ell}\big(\varphi^{\mathrm{rec}}_{\ell}(\varphi_{p_t^K}(x, t)\vee \underline{t}^{\frac{c_0-d}{2}+1}m^{-\frac{2(\alpha+1) }{d}}), \varphi_{\nabla p_t^K}(x, t)\big),
	\]
	and we have for $x\in M_{\rho, \underline{t}}$ that
	\begin{align*}
		|\varphi_{s_0}(x, t)-\widehat{s}_0^K(x, t)|
		&\le|\varphi_{s_0}(x, t)-\varphi^{\mathrm{rec}}_{\ell}(\varphi_{p_t^K}(x, t)\vee c\underline{t}^{\frac{c_0-d}{2}+1}m^{-\frac{2(\alpha+1) }{d}})\varphi_{\nabla p_t^K}(x, t)| \\
		&\qquad+|\varphi_{\nabla p_t^K}(x, t)|\Big|\varphi^{\mathrm{rec}}_{\ell}(\varphi_{p_t^K}(x, t)\vee c\underline{t}^{\frac{c_0-d}{2}+1}m^{-\frac{2(\alpha+1) }{d}})-\frac{1}{\varphi_{p_t^K}(x, t)\vee c\underline{t}^{\frac{c_0-d}{2}+1}m^{-\frac{2(\alpha+1) }{d}}}\Big| \\
		&\qquad+\Big|\frac{\varphi_{\nabla p_t^K}(x, t)}{\varphi_{p_t^K}(x, t)\vee c\underline{t}^{\frac{c_0-d}{2}+1}m^{-\frac{2(\alpha+1) }{d}}}-\frac{\widehat{\nabla p}_t^K(x)}{\widehat{p}_t^K(x)}\Big|.
	\end{align*}
	Here, the first two terms together are bounded by $|\varphi_{\nabla p_t^K}(x, t)|(c\underline{t}^{-\frac{c_0-d}{2}-1}m^{\frac{2(\alpha+1) }{d}}+1)2^{-\ell}$, where again $|\varphi_{\nabla p_t^K}(x, t)|\lesssim(2\pi\underline{t})^{-\frac{d}{2}}\frac{K_{\underline{t}}^D}{\underline{t}}$ by Lemma \ref{lem:capping_network}.
	For the third term, we have
	\begin{align*}
		\Big|\frac{\varphi_{\nabla p_t^K}(x, t)}{\varphi_{p_t^K}(x, t)\vee c\underline{t}^{\frac{c_0-d}{2}+1}m^{-\frac{2(\alpha+1) }{d}}}-\frac{\widehat{\nabla p}_t^K(x)}{\widehat{p}_t^K(x, t)}\Big|
		&\le\frac{m^{\frac{2(\alpha+1)}{d}}}{c\underline{t}^{\frac{c_0-d}{2}+1}}\Big(|\widehat{s}_0^K(x, t)| |\widehat{p}_t^K(x, t)-\varphi_{p_t^K}(x, t)|+|\widehat{\nabla p}_t^K(x, t)-\varphi_{\nabla p_t^K}(x, t)|\Big) \\
		&\lesssim \frac{K_{\underline{t}}^{D+2}m^{\frac{2(\alpha+1)}{d}}}{\underline{t}^{\frac{c_0}{2}+2}}(|\widehat{s}_0^K(x, t)|+1)2^{-\ell}.
	\end{align*}
	Here, since $|s_0^K(x, t)|\lesssim\frac{1}{t}$ by (the proof of) Lemma \ref{lem:affine_approx_a} and $|s_0^K(x, t)-\widehat{s}_0^K(x, t)|\lesssim \frac{1}{t}$ by \eqref{eq:s_0^K_bound}, we have also that $|\widehat{s}_0^K(x, t)|\lesssim\frac{1}{t}$ for $x\in M_{\rho, \underline{t}}$, and hence
	\[
		\Big|\frac{\varphi_{\nabla p_t^K}(x, t)}{\varphi_{p_t^K}(x, t)\vee c\underline{t}^{\frac{c_0-d}{2}+1}m^{-\frac{2(\alpha+1) }{d}}}-\frac{\widehat{\nabla p}_t^K(x)}{\widehat{p}_t^K(x, t)}\Big|
		\lesssim \frac{K_{\underline{t}}^{D+2}m^{\frac{2(\alpha+1)}{d}}}{\underline{t}^{\frac{c_0}{2}+3}}2^{-\ell},
	\]
	whereby all in all
	\[
		|\varphi_{s_0}(x, t)-\widehat{s}_0(x, t)|
		\lesssim\frac{K_{\underline{t}}^{D+2}m^{\frac{2(\alpha+1)}{d}}}{\underline{t}^{\frac{c_0}{2}+3}}2^{-\ell}
		\lesssim (\log m)^{D+2}m^{\frac{2(\alpha+1)}{d}+\frac{c_0}{2}+3}2^{-\ell}.
	\]
	Thus, setting $\ell=\lceil(\frac{3(\alpha+1)}{d}+\frac{c_0}{2}+3)\log_2 m+(D+2)\log_2\log m\rceil\lesssim\log m$ ensures that this is bounded by $m^{-\frac{\alpha+1}{d}}$.
	As for $x\notin M_{\rho, t}$, we can once again assume by Lemmas \ref{lem:capping_network} and \ref{lem:affine_approx_e} that $|\varphi_{s_0}(x, t)|\lesssim\frac{\sqrt{\rho+\log \underline{t}^{-1}}}{\sqrt{\underline{t}\wedge 1}}\lesssim\frac{\sqrt{\log m}}{\sqrt{\underline{t}\wedge 1}}$ (since this is true of $s_0^K\bm{1}_{M_{\rho, t}}$), and so it follows by Lemma \ref{lem:brownian_bound_a}
	\begin{align*}
		\mathbb{E}[|\varphi_{s_0}(X_t, t)-\widehat{s}_0^K(X_t, t)\bm{1}_{M_{\rho, t}}(X_t)|^2]
		&\lesssim m^{-\frac{2(\alpha+1)}{d}}\mathbb{P}(X_t\in M_{\rho, t})+\frac{\log m}{\underline{t}\wedge 1}\mathbb{P}(X_t\notin M_{\rho, t}) \\
		&\lesssim m^{-\frac{2(\alpha+1)}{d}}+\frac{\log m}{\underline{t}\wedge 1}\rho^{\frac{D}{2}}\mathrm{e}^{-\rho} \\
		&\lesssim (\log m)^{\frac{D+2}{2}}m^{-\frac{2(\alpha+1)}{d}}.
	\end{align*}
	Finally, combining this, \eqref{eq:s^K_hat_error} and Lemmas \ref{lem:affine_approx_b} and \ref{lem:truncation_error} along with repeated use of the triangle inequality, we have that
	\begin{align*}
		\int_{\underline{t}}^{2\underline{t}}\mathbb{E}[|\varphi_{s_0}(X_t, t)-s_0(X_t, t)|^2]\diff{t}
		&\lesssim \int_{\underline{t}}^{2\underline{t}}\mathbb{E}[|\varphi_{s_0}(X_t, t)-\widehat{s}_0^K(X_t, t)\bm{1}_{M_{\rho, t}}(X_t)|^2]\diff{t} \\
		&\qquad +\int_{\underline{t}}^{2\underline{t}}\mathbb{E}\big[|\widehat{s}_0^K(X_t, t)-s_0^K(X_t, t)|^2\bm{1}_{M_{\rho, t}}(X_t)\big]\diff{t}\\
		&\qquad +\int_{\underline{t}}^{2\underline{t}}\mathbb{E}\big[|s_0^K(X_t, t)-s_0(X_t, t)|^2\bm{1}_{M_{\rho, t}}(X_t)\big]\diff{t}\\
		&\qquad +\int_{\underline{t}}^{2\underline{t}}\mathbb{E}[s_0(X_t, t)\bm{1}_{M_{\rho, t}}(X_t)-s_0(X_t, t)|^2]\diff{t} \\
		&\lesssim (\log m)^{d+2D+3}m^{-\frac{2\alpha}{d}}\varepsilon_{\underline{t}}^2,
	\end{align*}
	as desired.
	As for the size of the network, we first have that for our choice of $\ell$ (and recalling that all sizes being multiplied or divided are bounded by $\underline{t}^{\frac{d-c_0}{2}-1}m^{\frac{2(\alpha+1) }{d}}\lesssim m^{\frac{2(\alpha+1)}{d}+\frac{c_0-d}{2}+1}$)
	\begin{align*}
		\varphi^{\mathrm{exp}}_{\ell}
		&\in\widetilde{\Phi}( (\log m)^2(\log\log m)^2, \log m \log\log m,  (\log m)^3(\log\log m)^3, 1) \\
		\varphi^{\mathrm{mult}}_{\ell}
		&\in\widetilde{\Phi}(\log m, 1, \log m, m^{\frac{2(\alpha+1)}{d}+\frac{c_0-d}{2}+1}) \\
		\varphi^{\mathrm{rec}}_{\ell}
		&\in\widetilde{\Phi}(\log m\log\log m, \log m, \log m\log\log m, m^{\frac{2(\alpha+1)}{d}+\frac{c_0-d}{2}+1}) \\
		\varphi^{\mathrm{norm}}_{\ell}
		&\in\widetilde{\Phi}(\log m, 1, \log m, \log m),
	\end{align*}
	whereby $\widehat{\varphi}^{\mathrm{exp}}_{\ell}\in\widetilde{\Phi}((\log m)^2(\log\log m)^2, \log m \log\log m,  (\log m)^3(\log\log m)^3, m^{\frac{2(\alpha+1)}{d}+\frac{c_0-d}{2}+1})$.
	Since the size of $\varphi^{\mathrm{mult}}_{\ell}$ is comparably negligible to those of $\varphi_1$ and $\widehat{\varphi}^{\mathrm{exp}}_{\ell}$, this implies that 
	\[
		\varphi_{h_1}
		\in \begin{cases}
			\widetilde{\Phi}\big((\log m)^2(\log\log m)^2, m\log m,  m\log^2m, m^{\frac{2(\alpha+1)}{d}+\frac{c_0-d}{2}+1}\vee m^{\nu}\big), &\text{ if }\underline{t}\le\frac{1}{2}m^{-\frac{2-\delta}{d}}\\
			\widetilde{\Phi}\big((\log m)^2(\log\log m)^2, m'\log m, m'(\log m)^2, m^{\frac{2(\alpha+1)}{d}+\frac{c_0-d}{2}+1}\big), &\text{ if }\underline{t}>\frac{1}{2}m^{-\frac{2-\delta}{d}}
		\end{cases}
	\]
	and similar analysis shows that the same is true of $\varphi_{h_2}$.
	Finally, summing $\lceil K_{\underline{t}}\rceil^D\lesssim(\log m)^D$ copies of these yields that
	\[
		\varphi_{p_t^K}, \varphi_{\nabla p_t^K}
		\in\begin{cases}
			\widetilde{\Phi}\big((\log m)^2(\log\log m)^2, m(\log m)^{D+1},  m(\log m)^{D+2}, m^{\frac{2(\alpha+1)}{d}+\frac{c_0-d}{2}+1}\vee m^{\nu}\big), &\text{ if }\underline{t}\le\frac{1}{2}m^{-\frac{2-\delta}{d}}\\
			\widetilde{\Phi}\big((\log m)^2(\log\log m)^2, m'(\log m)^{D+1}, m'(\log m)^{D+2}, m^{\frac{2(\alpha+1)}{d}+\frac{c_0-d}{2}+1}\big), &\text{ if }\underline{t}>\frac{1}{2}m^{-\frac{2-\delta}{d}}
		\end{cases}
	\]
	and since the remaining networks are once again negligible to this, we have also
	\[
		\varphi_{s_0}
		\in\begin{cases}
			\widetilde{\Phi}\big((\log m)^2(\log\log m)^2, m(\log m)^{D+1},  m(\log m)^{D+2}, m^{\frac{2(\alpha+1)}{d}+\frac{c_0-d}{2}+1}\vee m^{\nu}\big), &\text{ if }\underline{t}\le\frac{1}{2}m^{-\frac{2-\delta}{d}}\\
			\widetilde{\Phi}\big((\log m)^2(\log\log m)^2, m'(\log m)^{D+1}, m'(\log m)^{D+2}, m^{\frac{2(\alpha+1)}{d}+\frac{c_0-d}{2}+1}\big), &\text{ if }\underline{t}>\frac{1}{2}m^{-\frac{2-\delta}{d}}
		\end{cases},
	\]
	as desired.
\end{proof}

\section{Basic neural network approximation results}\label{app:basic_neural}
\begin{lemma}[\cite{holk24} Lemma 3.10]
	\label{lem:mult_network_asymp}
	For $m\in\N$ and $C\ge1$, there exist neural networks $\varphi_{m}^{\mathrm{mult}}\in\widetilde{\Phi}(m, 1, m, C)$ and $\varphi_{m}^{\mathrm{mult}, d}\in\widetilde{\Phi}(m, d, dm, C)$ satisfying
	\[
		|\varphi^{\mathrm{mult}}_m(x, y)-xy|
		\le C2^{-m},\quad x\in[0, 1],y\in[-C, C],
	\]
	and
	\[
		|\varphi^{\mathrm{mult}, d}_m(x, y)-xy|
		\le\sqrt{d}C2^{-m},\quad x\in[0, 1],y\in[-C, C]^d.
	\]
	These also satisfy $\varphi^{\mathrm{mult}}_m(x, 0)=\varphi^{\mathrm{mult}}_m(0, y)=0$.
\end{lemma}

\begin{lemma}[\cite{holk24} Lemma 3.11]
	\label{lem:rec_network_asymp}
	For $m,\underline{k},\overline{k}\in\N$, there exists a neural network $\varphi_{m}^{\mathrm{rec}}\in\widetilde{\Phi}( (k+m)\log(k+m), k, (k+m)\log(k+m), 2^k)$, where $k=\underline{k}+\overline{k}$, satisfying
	\[
		|\varphi_{m}^{\mathrm{rec}}(x)-x^{-1}|
		\le2^{-m},\qquad x\in[2^{-\underline{k}}, 2^{\overline{k}}].
	\]
\end{lemma}

\begin{lemma}
	\label{lem:exp_network}
	For $m\in\N$ there exists a neural network $\varphi_m^{\mathrm{exp}}\in\widetilde{\Phi}(m^2\log^2m, m\log m, m^3\log^3 m, 1)$ satisfying
	\[
		|\varphi_m^{\mathrm{exp}}(x)-\mathrm{e}^{-x}|
		\le 2^{-m},\qquad x\ge0.
	\]
\end{lemma}

\begin{proof}
	First note that for $x\ge m\log 2\coloneqq K$, we have $\mathrm{e}^{-x}\le 2^{-m}$, and so we need find an approximation $\varphi^{\mathrm{exp}}_m$ satisfying $|\varphi_m^{\mathrm{exp}}(x)-\mathrm{e}^{-x}|\le 2^{-m}$ for $x\in[0, K]$ with $|\varphi^{\mathrm{exp}}_m(x)|\le 2^{-m}$ for $x\ge K$.
	To this end, note then that $\mathrm{e}^{-x}=\big(\mathrm{e}^{-\frac{x}{\lceil K\rceil}}\big)^{\lceil K\rceil}$, whereby we need only approximate $\mathrm{e}^{-x}$ and $x^n$ for $x\in[0, 1]$ and $n\in\N$.
	To this end, let $\varphi_{k_1}^{\mathrm{mult}}$ be as in Lemma \ref{lem:mult_network_asymp} with $C=1$ for some $k_1\in\N$ to be determined later and set $\varphi_{k_1}^{\mathrm{pow}, n}(x)=\varphi_{k_1}^{\mathrm{mult}}(x, \varphi_{k_1})^{\mathrm{pow}, n-1}$ for $n\ge3$ with $\varphi_{k_1}^{\mathrm{pow}, 2}(x)=\varphi_{k_1}^{\mathrm{mult}}(x, x)$.
	Then, $\varphi^{\mathrm{pow}, n}_{k_1}\in\widetilde{\Phi}(nk_1, 1, nk_1, 1)$, and we have
	\[
		|\varphi^{\mathrm{pow}, n}_{k_1}(x)-x^{n}|
		\le|\varphi^{\mathrm{pow}, n}_{k_1}(x)-x\varphi^{\mathrm{pow}, n-1}_{k_1}(x)|+x|\varphi^{\mathrm{pow}, n-1}_{k_1}(x)-x^{n-1}|
		\le 2^{-k_1}+|\varphi^{\mathrm{pow}, n-1}_{k_1}(x)-x^{n-1}|,
	\]
	from which it follows by elementary induction that $|\varphi^{\mathrm{pow}, n}_{k_1}(x)-x^{n}|\lesssim n2^{-k_1}$.
	Next, for some $k_2\in\N$, also to be determined later, let
	\[
		\widetilde{\varphi}_{k_1, k_2}^{\mathrm{exp}}(x)
		=1-x+\sum_{k=2}^{k_2}\frac{(-1)^k\varphi_{k_1}^{\mathrm{pow}, k}(x)}{k!},
	\]
	such that $\widetilde{\varphi}_{k_1, k_2}^{\mathrm{exp}}\in\widetilde{\Phi}(k_1k_2, k_2, k_1k_2^2, 1)$ and
	\[
		|\widetilde{\varphi}_{k_1, k_2}^{\mathrm{exp}}(x)-\mathrm{e}^{-x}|
		\le\Big|\mathrm{e}^{-x}-\sum_{k=0}^{k_2}\frac{(-x)^k}{k!}\Big|+\sum_{k=2}^{k_2}\frac{|\varphi_{k_1}^{\mathrm{pow}, k}(x)-x^k|}{k!}
		\lesssim\frac{1}{k_2!}+2^{-k_1}\sum_{k=2}^{k_2}\frac{k-1}{k!}
		\le\frac{1}{k_2!}+2^{-k_1}
	\]
	for all $x\in[0, 1]$.
	Setting $\overline{\varphi}_m^{\mathrm{exp}}(x)=\varphi_{k}^{\mathrm{pow}, \lceil K\rceil}\circ\widetilde{\varphi}_{k, k}^{\mathrm{exp}}\big(\frac{x}{\lceil K\rceil}\big)$ with $k\ge m\log m\vee 4$ (ensuring $\frac{1}{k!}\le 2^{-k}$) for $x\in[0, K]$, we then see that $\overline{\varphi}_m^{\mathrm{exp}}\in\widetilde{\Phi}(m^2\log^2m, m\log m, m^3\log^3 m, 1)$ 
	\begin{align*}
		\big|\overline{\varphi}_m^{\mathrm{exp}}(x)-\mathrm{e}^{-x}\big|
		&\le\Big|\overline{\varphi}_m^{\mathrm{exp}}(x)-\widetilde{\varphi}_{k, k}^{\mathrm{exp}}\Big(\frac{x}{\lceil K\rceil}\Big)^{\lceil K\rceil}\Big|+\Big|\widetilde{\varphi}_{k, k}^{\mathrm{exp}}\Big(\frac{x}{\lceil K\rceil}\Big)^{\lceil K\rceil}-\Big(\mathrm{e}^{-\frac{x}{\lceil K\rceil}}\Big)^{\lceil K\rceil}\Big| \\
		&\lesssim \lceil K\rceil\Big(2^{-k}+\Big|\widetilde{\varphi}_{k, k}^{\mathrm{exp}}\Big(\frac{x}{\lceil K\rceil}\Big)-\mathrm{e}^{-\frac{x}{\lceil K\rceil}}\Big|\Big) \\
		&\lesssim m2^{-k} \\
		&\le 2^{-m}.
	\end{align*}
	Finally, the function
	\[
		p(x)
		=\begin{cases}
			1, &\text{ if }x\le K-1\\
			K-x, &\text{ if }K-1\le x\le K\\
			0, & \text{ if }x\ge K
		\end{cases}
	\]
	is exactly representable as a neural network, and setting $\varphi^{\mathrm{exp}}_m(x)=\varphi^{\mathrm{mult}}_m(p(x), \overline{\varphi}^{\mathrm{exp}}_m(x))$ ensures that $|\varphi^{\mathrm{exp}}_m(x)|\le 2^{-m}$ for $x\ge K$ without altering the asymptotic size of the network.
\end{proof}

\begin{lemma}
	\label{lem:norm_network}
	For $m\in\N$ and $K\in\N_0$ there exists a neural network $\varphi^{\mathrm{norm}}_{m}\in\widetilde{\Phi}(m, D, Dm, K)$ satisfying
	\[
		\big|\varphi_m^{\mathrm{norm}}(x)-|x|^2\big|
		\le DK^22^{-m},\qquad
		\n{x}_{\infty}\le K.
	\]
\end{lemma}

\begin{proof}
	A small modification of Lemma \ref{lem:mult_network_asymp} yields a network $\widetilde{\varphi}^{\mathrm{mult}}_m\in\widetilde{\varphi}(m, 1, m, K)$ with
	\[
		|\widetilde{\varphi}^{\mathrm{mult}}_m(x, y)-xy|
		\le K^22^{-m},\qquad x, y\in[-K, K].
	\]
	Setting $\varphi^{\mathrm{norm}}_m(x)=\sum_{i=1}^{D}\widetilde{\varphi}^{\mathrm{mult}}_m(x_i, x_i)$, we have that $\varphi^{\mathrm{norm}}_m\in\widetilde{\Phi}(m, D, Dm, K)$, and for all $x\in[-K, K]^D$, we have
	\[
		\big|\varphi^{\mathrm{norm}}_m(x)-|x|^2\big|
		\le\sum_{i=1}^D|\widetilde{\varphi}_m^{\mathrm{mult}}(x_i, x_i)-x_i^2|
		\le DK^22^{-m}.
	\]
\end{proof}

\begin{lemma}
	\label{lem:covering_network}
	For every compact set $K\subseteq\R^d$ with diameter $R>0$ and every $\varepsilon>0$, there exists a neural network $\varphi_{\bm{1}_K}\in\widetilde{\Phi}\big(\log \frac{R}{\varepsilon}, (\frac{R}{\varepsilon})^d, (\frac{R}{\varepsilon})^d, \frac{1}{\varepsilon}\big)$ satisfying $\varphi_{\bm{1}_K}(x)\in[0, 1]$ for all $x\in\R^d$, $\varphi_{\bm{1}_K}(K)=1$ and $\varphi_{\bm{1}_K}(K_{\varepsilon}^\mathsf{c})=0$, where $K_\varepsilon$ is the $\varepsilon$-fattening of $K$.
\end{lemma}
\begin{proof}
	First, for $r>0$, let $K_r^{\infty}=\{x\in\R^d: \exists y\in K \text{ s.t. }\n{x-y}_{\infty}<r\}$ denote the $r$-fattening of $K$ with respect to $\n{\cdot}_{\infty}$.
	Since $|x|\le\sqrt{d}\n{x}_{\infty}$, we then have that $K_\varepsilon^\mathsf{c}\subseteq (K_{\varepsilon'}^\infty)^\mathsf{c}$, where $\varepsilon'=\frac{\varepsilon}{\sqrt{d}}$, and so we need only find a neural network $\varphi_{\bm{1}_K}$ satisfying $\varphi_{\bm{1}_K}(x)\in[0, 1]$ with $\varphi_{\bm{1}_K}(K)=1$ and $\varphi_{\bm{1}_K}( (K_{\varepsilon'}^\infty)^{\mathsf{c}} )=0$.
	The reason for working with $\n{\cdot}_{\infty}$ rather than $|\cdot|$ is that while $|x|$ needs to be approximated by neural networks, $\n{x}_{\infty}$ is itself exactly representable as a neural network, as $a\vee b=a+\sigma(b-a)$ and $\n{x}_{\infty}=(-x_1\vee x_1)\vee(-x_2\vee x_2)\vee\ldots\vee(-x_d\vee x_d)$.
	Now, let $y_1, y_2, \ldots, y_N$ be a minimal $\frac{\varepsilon'}{4}$-covering of $K$ with respect to $\n{\cdot}_{\infty}$, and set $\varphi^{\mathrm{dist}}(x)=\min_{i\in[N]}\n{x-y_i}_{\infty}$.
	Since also $a\wedge b=b-\sigma(b-a)$, $\varphi^{\mathrm{dist}}$ is also representable as a neural network.
	In particular, by using a divide and conquer strategy, we have that $\varphi^{\mathrm{dist}}\in\widetilde{\Phi}(\log N, N, N, 1)$, and we see that $\varphi^{\mathrm{dist}}$ satisfies $\varphi^{\mathrm{dist}}(x)>\frac{3\varepsilon'}{4}$ for $x\notin K_{\varepsilon'}^{\infty}$, while $\varphi^{\mathrm{dist}}(x)<\frac{\varepsilon'}{4}$ for $x\in K$.
	Lastly, set
	\[
		\varphi_{\bm{1}_{K}}(x)
		=\Big(1\wedge\Big(\frac{3}{2}-\frac{2\varphi^{\mathrm{dist}}(x)}{\varepsilon'}\Big)\Big)\vee0,
	\]
	and we see that $\varphi_{\bm{1}_K}$ satisfies our criteria and that $\varphi_{\bm{1}_K}\in\widetilde{\Phi}(\log N, N, N, \frac{1}{\varepsilon'})$.
	Finally, noting that since $K$ is of diameter $R$ and hence contained in $[0, R]^d+y_0$ for some $y_0\in\R^d$, its covering number is less than that of $[0, R]^D$, i.e.
	\[
		N
		=N\Big(K, \n{\cdot}_{\infty}, \frac{\varepsilon'}{4}\Big)
		\le N\Big([0, R]^d, \n{\cdot}_{\infty}, \frac{\varepsilon'}{4}\Big)
		\le\Big\lceil\frac{4\sqrt{d}R}{\varepsilon}\Big\rceil^d
		\lesssim\Big(\frac{R}{\varepsilon}\Big)^d,
	\]
	whence $\varphi_{\bm{1}_K}\in\widetilde{\Phi}\big(\log \frac{R}{\varepsilon}, (\frac{R}{\varepsilon})^d, (\frac{R}{\varepsilon})^d, \frac{1}{\varepsilon}\big)$ as desired.
\end{proof}

\begin{lemma}[Proposition 1 in \cite{suzuki19}]
	\label{lem:sobolev_network}
	Let $S$ be a Lipschitz domain with $S\subseteq[-K, K]^d$ for some $K\ge1$ and let $g:S\to\R$ have Sobolev smoothness $\gamma$ for some $\gamma>\frac{d}{2}$, i.e. $\n{g}_{H^\gamma}<\infty$.
	Then, for large enough $m\in\N$, there exists a neural network $\varphi_g\in\widetilde{\Phi}( \gamma^2\log m, \gamma^2m, \gamma^4m\log m, m^{\nu})$ where $\nu=\frac{d}{\gamma-\frac{d}{2}}+\frac{1}{d}$, satisfying
	\[
		|g(u)-\varphi_g(u)|
		\lesssim K^{\gamma-\frac{d}{2}}\n{g}_{H^\gamma}m^{-\frac{\gamma}{d}},\qquad u\in[-K, K]^d
	\]
\end{lemma}

\begin{proof}
    First, since $S$ is Lipschitz, we may extend $g$ to a function $\widehat{g}:[-K, K]^d\to\R$, also with Sobolev smoothness $\gamma$.
    To avoid the cumbersome notation, we simply assume without loss of generality that $S=[-K, K]^d$.
	Then, let $\eta_K(u)=Ku$ and set $\widetilde{g}=\frac{g\circ\eta_K}{\n{g\circ\eta_K}_{B_{2, 2}^\gamma}}$, where $\n{\cdot}_{B_{2, 2}^\gamma}$ denotes the norm associated with the Besov space $B_{2, 2}^{\gamma}$.
	Since $H^\gamma\cong B_{2, 2}^{\gamma}$, we have $\widetilde{g}\in B_{2, 2}^{\gamma}$ and $\n{\widetilde{g}}_{B_{2, 2}^\gamma}=1$, whence it follows by \cite[Proposition 1]{suzuki19} that there exists a neural network $\widetilde{\varphi}_{g}\in\widetilde{\Phi}(\gamma^2\log m, \gamma^2 m, \gamma^4 m\log m, m^\nu)$ with
	\[
		|\widetilde{\varphi}_g(u)-\widetilde{g}(u)|
		\lesssim m^{-\frac{\gamma}{d}},\qquad u\in[-1, 1]^{d}.
	\]
	Now, letting $\varphi_g(u)=\n{g\circ\eta_K}_{B_{2, 2}^{\gamma}}\widetilde{\varphi}_g(\frac{u}{K})$, it follows that for any $u\in[-K, K]^d$, we have
	\[
		|\varphi_g(u)-g(u)|
		=\n{g\circ\eta_K}_{B_{2, 2}^{\gamma}}\Big|\widetilde{\varphi}_g\Big(\frac{u}{K}\Big)-\widetilde{g}\Big(\frac{u}{K}\Big)\Big|
		\lesssim\n{g\circ\eta_K}_{B_{2, 2}^{\gamma}} m^{-\frac{\gamma}{d}}.
	\]
	To bound $\n{g\circ\eta_K}_{B_{2, 2}^{\gamma}}$, we first note that $\n{g\circ\eta_K}_{B_{2, 2}^{\gamma}}\asymp\n{g\circ\eta_K}_{H^\gamma}$, and that for $\beta\in\N_0^d$ we have $\partial^\beta (g\circ\eta_K)=K^{|\beta|}(\partial^\beta g)\circ\eta_K$, while
	\[
		\n{(\partial^\beta g)\circ\eta_K}_{L^2}^2
		=\int_{\R^d}|\partial^\beta g(Ku)|^2\diff{u}
		=K^{-d}\int_{\R^d}|\partial^\beta g(v)|^2\diff{v}
		=K^{-d}\n{\partial^\beta g}_{L^2}^2.
	\]
	Combining these, we have
	\[
		\n{g\circ\eta_K}_{H^\gamma}
		=\sqrt{\sum_{|\beta|\le\gamma}\n{\partial^\beta(g\circ\eta_K)}_{L^2}^2}
		=\sqrt{\sum_{|\beta|\le\gamma}K^{2|\beta|-d}\n{\partial^\beta g}_{L^2}^2}
		\le K^{\gamma-\frac{d}{2}}\n{g}_{H^\gamma},
	\]
	as desired.
\end{proof}

\begin{lemma}
	\label{lem:capping_network}
	Let $g\colon E\to\R^{k}$ be a function on some  subset $E \subset \R^{W_0}$ such that $|g(s)|\le C$ for all $s\in E$ and some constant $C>0$.
	Then, for all $\widetilde{\varphi}\in\widetilde{\Phi}(L, W, S, B)$, where $W_{L+1} = k$, there exists a neural network $\varphi\in\widetilde{\Phi}(L, W, S, C\vee B)$ satisfying
	\[
		|g(s)-{\varphi}(s)|\le|g(s)-\widetilde{\varphi}(s)|,
	\]
	and $|\varphi(s)|\le \sqrt{k}C$ for all $s\in E$.
\end{lemma}
\begin{proof}
	First, note that for all $s\in E$ we have $\n{g(s)}_{\infty}\le|g(s)|\le C$, so setting
	\[
		\varphi(s)
		=\mat{\varphi_1(s) \\ \vdots \\ \varphi_k(s)}
		=\mat{\widetilde{\varphi}_1(s)\wedge C\vee(-C) \\ \vdots \\ \widetilde{\varphi}_k(s)\wedge C\vee(-C)},
	\]
	we have immediately that $|\varphi_i(s)-g_i(s)|\le|\widetilde{\varphi}_i(s)-g_i(s)|$ and hence $|\varphi(s)-g(s)|\le|\widetilde{\varphi}(s)-g(s)|$, while $|\varphi(s)|\le\sqrt{k}\n{\varphi(s)}_{\infty}\le kC$ for all $s \in E$.
	Finally, noting that
	\[
		\widetilde{\varphi}_i(s)\wedge C\vee(-C)
		=\widetilde{\varphi}_i(s)-\sigma(\widetilde{\varphi}_i(s)-C)+\sigma\big(\sigma(\widetilde{\varphi}_i(s)-C)-\widetilde{\varphi}_i(s)-C\big),
	\]
	it follows that $\varphi\in\widetilde{\Phi}(L, W, S, C\vee B)$.
\end{proof}

\section{Auxiliary technical results}
\begin{lemma}
	\label{lem:brownian_bound}
	Let $(W_t)_{t\ge0}$ be a $k$-dimensional Brownian motion for some $k\in\N$, and let $\rho>1$.
	Then, the following bounds hold:
	\begin{enumerate}[ref=\thelemma.(\alph*), label=(\alph*)]
		\item\label{lem:brownian_bound_a}
		$\mathbb{P}\bigl(|W_t|>\sqrt{t(k+2\rho)}\bigr)\lesssim \rho^{\frac{k}{2}}\mathrm{e}^{-\rho}$;
		\item\label{lem:brownian_bound_b}
		$\mathbb{E}\bigl[|W_t|\bm{1}_{\{|W_t|>\sqrt{t(k+2\rho)}\}}\bigr] \lesssim \sqrt{t}\,\rho^{\frac{k+1}{2}}\mathrm{e}^{-\rho}$.
	\end{enumerate}
\end{lemma}
\begin{proof}
	Let $Z_k\sim\mathcal{N}(0,I_k)$.
	For any $\lambda\in(0,\tfrac12)$, Markov's inequality yields
	\[
		\mathbb{P}\bigl(|Z_k|>\sqrt{k+2\rho}\bigr)
		=\mathbb{P}\Bigl(\mathrm{e}^{\lambda |Z_k|^2}>\mathrm{e}^{\lambda(k+2\rho)}\Bigr)
		\le \mathbb{E}\bigl[\mathrm{e}^{\lambda |Z_k|^2}\bigr]\mathrm{e}^{-\lambda(k+2\rho)}
		= M_{\chi^2_k}(\lambda)\,\mathrm{e}^{-\lambda(k+2\rho)},
	\]
	where $M_{\chi^2_k}(\lambda)=(1-2\lambda)^{-k/2}$ denotes the moment generating function of the $\chi^2_k$ distribution.
	Define $\psi(\lambda)\coloneqq(1-2\lambda)^{-\frac{k}{2}}\mathrm{e}^{-\lambda(k+2\rho)}$.
	A direct computation shows that
	\[
	\psi'(\lambda)
	=2\bigl(\lambda(k+2\rho)-\rho\bigr)(1-2\lambda)^{-\frac{k}{2}-1}\mathrm{e}^{-\lambda(k+2\rho)},
	\]
	which is negative for $\lambda<\tfrac{\rho}{k+2\rho}$ and positive for $\lambda>\tfrac{\rho}{k+2\rho}$.
	Hence, $\psi$ attains its minimum at $\lambda=\tfrac{\rho}{k+2\rho}$, and we obtain
	\[
		\mathbb{P}\bigl(|Z_k|>\sqrt{k+2\rho}\bigr)
		\le \psi\Bigl(\frac{\rho}{k+2\rho}\Bigr)
		=\Bigl(1-\frac{2\rho}{k+2\rho}\Bigr)^{-\frac{k}{2}}\mathrm{e}^{-\rho} 
		=\Bigl(1+\frac{2\rho}{k}\Bigr)^{\frac{k}{2}}\mathrm{e}^{-\rho}
		\lesssim \rho^{\frac{k}{2}}\mathrm{e}^{-\rho}.
	\]
	Since $W_t\stackrel{d}{=}\sqrt{t}\,Z_k$, this proves \hyperref[lem:brownian_bound_a]{(a)}.
	We next consider the truncated first moment.
	Using polar coordinates, we compute
	\begin{align*}
		\mathbb{E}\bigl[|Z_k|\bm{1}_{\{|Z_k|>\sqrt{k+2\rho}\}}\bigr]
		&=(2\pi)^{-\frac{k}{2}}\int_{\{|x|>\sqrt{k+2\rho}\}}|x|\mathrm{e}^{-\frac{|x|^2}{2}}\diff x
		=\frac{2^{-\frac{k}{2}+1}}{\Gamma(\frac{k}{2})}
		\int_{\sqrt{k+2\rho}}^{\infty} r^{k}\mathrm{e}^{-\frac{r^2}{2}}\diff r \\
		&=\frac{\sqrt{2}}{\Gamma(\frac{k}{2})}
		\int_{\frac{k+2\rho}{2}}^{\infty} u^{\frac{k+1}{2}-1}\mathrm{e}^{-u}\diff u 
		=\sqrt{2}\,
		\frac{\Gamma\bigl(\frac{k+1}{2},\frac{k+2\rho}{2}\bigr)}{\Gamma(\frac{k}{2})},
	\end{align*}
	where $\Gamma(s,x)$ denotes the Gamma function. Moreover,
	\[
	\mathbb{P}\bigl(|Z_k|>\sqrt{k+2\rho}\bigr)
	=\frac{\Gamma\bigl(\frac{k}{2},\frac{k+2\rho}{2}\bigr)}{\Gamma(\frac{k}{2})}.
	\]
	Combining this with part~\hyperref[lem:brownian_bound_a]{(a)}, we find
	\[
	\mathbb{E}\bigl[|Z_k|\bm{1}_{\{|Z_k|>\sqrt{k+2\rho}\}}\bigr]
	\lesssim
	\frac{\Gamma\bigl(\frac{k+1}{2},\frac{k+2\rho}{2}\bigr)}{\Gamma\bigl(\frac{k}{2},\frac{k+2\rho}{2}\bigr)}
	\rho^{\frac{k}{2}}\mathrm{e}^{-\rho}.
	\]
	Using the asymptotic relation $\Gamma(s,x)\sim x^{s-1}\mathrm{e}^{-x}$ as $x\to\infty$, we obtain
	\[
	\frac{\Gamma(s+\frac{1}{2}, x)}{\Gamma(s, x)}x^{-\frac{1}{2}}
	=\frac{\Gamma(s+\frac{1}{2}, x)x^{1-s-\frac{1}{2}}\mathrm{e}^{x}}{\Gamma(s, x)x^{1-s}\mathrm{e}^{x}}
	\to 1\quad\text{ as }x\to\infty,
	\]
	and hence
	\[
		\mathbb{E}\bigl[|Z_k|\bm{1}_{\{|Z_k|>\sqrt{k+2\rho}\}}\bigr]\lesssim \frac{\Gamma(\frac{k+1}{2}, \frac{k+2\rho}{2})}{\Gamma(\frac{k}{2}, \frac{k+2\rho}{2})}\rho^{\frac{k}{2}}\mathrm{e}^{-\rho}
		\asymp \sqrt{\frac{k+2\rho}{2}}\,
		\rho^{\frac{k}{2}}\mathrm{e}^{-\rho}
		\lesssim \rho^{\frac{k+1}{2}}\mathrm{e}^{-\rho}.
	\]
	Finally, since $|W_t|\stackrel{d}{=}\sqrt{t}\,|Z_k|$, the last display already yields \hyperref[lem:brownian_bound_b]{(b)}.
\end{proof}

\begin{lemma}
	Let $k\in\N$ be given and set $t_j=\cos(\frac{j\pi}{k})$ for $j=0,\ldots,k$ and let $p_i(t)=\prod_{j\neq i}(t-t_j)$ for $i=0, \ldots, k$.
	Then, $|\frac{p_i(t)}{p_i(t_i)}|\le2$ for all $i$ and $t\in(-1, 1)$.
\end{lemma}

\begin{proof}
	Let $p(t)=\prod_{j=0}^{k}(t-t_j)$ and $\widehat{p}(t)=\prod_{j=1}^{k-1}(t-t_j)$.
	By comparing roots, we see that $\widehat{p}$ is simply a re-scaling of the $k-1$'st Chebyshev polynomial of the second kind $U_{k-1}$ given by $U_{k-1}(\cos\theta)\sin\theta=\sin k\theta$.
	In particular, since by L'Hôpital we have
	\[
		U_{k-1}(1)
		=\lim_{\theta\to0}U_{k-1}(\cos\theta)
		=\lim_{\theta\to0}\frac{\sin k\theta}{\sin\theta}
		=\lim_{\theta\to0}\frac{k\cos k\theta}{\cos\theta}
		=k,
	\]
	and by \cite{trefethen13} that $\widehat{p}(1)=\frac{p_0(t_0)}{2}=\frac{k}{2^{k-1}}$, we have $\widehat{p}(t)=\frac{U_{k-1}(t)}{2^{k-1}}$.
	In particular,
	\[
		p(\cos\theta)
		=(\cos\theta-1)(\cos\theta+1)\widehat{p}(\cos\theta)
		=-\sin^2(\theta)\frac{U_{k-1}(\cos\theta)}{2^{k-1}}
		=\frac{-\sin\theta\sin k\theta}{2^{k-1}},
	\]
	and so for $\theta\neq\frac{i\pi}{k}$,
	\[
		p_i(\cos\theta)
		=-\frac{\sin\theta\sin k\theta}{2^{k-1}(\cos\theta-\cos\frac{i\pi}{k})},
	\]
	while \cite{trefethen13} again yields that
	\[
		p_i(t_i)
		=\begin{cases}
			(-1)^i\frac{k}{2^{k-1}}, &\text{ if }i\notin\{0, k\}\\
			(-1)^i\frac{k}{2^{k-2}}, &\text{ if }i\in\{0, k\}.
		\end{cases}
	\]
	Thus, for $\theta\neq\frac{i\pi}{k}$,
	\[
		\Big|\frac{p_i(\cos\theta)}{p_i(t_i)}\Big|
		\le\frac{|\sin\theta\sin k\theta|}{k|\cos\theta-\cos\frac{i\pi}{k}|}.
	\]
	From the above it follows that $|\frac{p_i(\cos\theta)}{p_i(t_i)}|=|\frac{p_{k-i}(\cos(\pi-\theta))}{p_{k-i}(t_{k-i})}|$, and so we may assume going forward that $\theta<\frac{i\pi}{k}$.
	Next, since $|\cos\theta-\cos\frac{i\pi}{k}|=2|\sin(\frac{\theta}{2}+\frac{i\pi}{2k})\sin(\frac{\theta}{2}-\frac{i\pi}{2k})|$, it follows that
	\[
		\Big|\frac{p_i(t)}{p_i(t_i)}\Big|
		\le\frac{1}{2k}\sup_{\theta\in(0, \frac{i\pi}{k})}\Big|\frac{\sin\theta}{\sin(\frac{\theta}{2}+\frac{i\pi}{2k})}\Big|\sup_{\theta\in(0, \frac{i\pi}{k})}\Big|\frac{\sin k\theta}{\sin(\frac{\theta}{2}-\frac{i\pi}{2k})}\Big|.
	\]
	Now, let $g(a, \theta)=\frac{\sin\theta}{\sin\frac{\theta+a}{2}}$ and note that $g(a, \theta)\ge0$ for $\theta\le a\le\pi$, and we have
	\[
		\sup_{\theta\in(0, \frac{i\pi}{k})}\Big|\frac{\sin\theta}{\sin(\frac{\theta}{2}+\frac{i\pi}{2k})}\Big|
		=\sup_{\theta\in(0, \frac{i\pi}{k})}g\Big(\frac{i\pi}{k}, \theta\Big)
		\le \sup_{a\in(0, \pi)}\sup_{\theta\in(0, a)}g(a, \theta)
		=\sup_{\theta\in(0, \pi)}\sup_{a\in(\pi-\theta, \pi)}g(a, \theta).
	\]
	Since for fixed $\theta\in(0, \pi)$ and $a\in(\pi-\theta, \pi)$ we have
	\[
		\frac{\mathrm{d}}{\mathrm{d}a}g(a, \theta)
		=-\frac{1}{2}g(a, \theta)\cot\frac{\theta+a}{2}
		\ge0,
	\]
	it follows that $\sup_{a\in(\pi-\theta, \pi)}g(a, \theta)=g(\pi, \theta)$.
	Therefore, since $g(\pi, \theta)=2\sin\frac{\theta}{2}$, we have
	\[
		\sup_{\theta\in(0, \frac{i\pi}{k})}\Big|\frac{\sin\theta}{\sin(\frac{\theta}{2}+\frac{i\pi}{2k})}\Big|
		\le 2.
	\]
	Next, we have that
	\[
		\sup_{\theta\in(0, \frac{i\pi}{k})}\Big|\frac{\sin k\theta}{\sin(\frac{\theta}{2}-\frac{i\pi}{2k})}\Big|
		\le\sup_{\theta\in(0, \pi)}\Big|\frac{\sin k\theta}{\sin(\frac{\theta}{2}-\frac{i\pi}{2k})}\Big|
		=\sup_{\theta\in(0, \pi)}\Big|\frac{\sin (k(\theta-\frac{i\pi}{k}))}{\sin(\frac{\theta}{2}-\frac{i\pi}{2k})}\Big|
		=\sup_{\theta\in(0, \pi)}\Big|\frac{\sin k\theta}{\sin\frac{\theta}{2}}\Big|
		=2k.
	\]
	Plugging both of these into the estimate above, we find that $|\frac{p_i(t)}{p_i(t_i)}|\le 2$ as desired.
\end{proof}

\end{document}